\documentclass[11pt]{amsart}
\usepackage[hmargin=3cm,vmargin=3cm]{geometry}

\usepackage[latin1]{inputenc}
\usepackage{xspace,amssymb,amsfonts,euscript}
\usepackage{amsthm,amsmath,amscd,stmaryrd,latexsym}
\usepackage{enumitem}

\usepackage{graphicx}
\usepackage[all]{xy}
\SelectTips{cm}{}

\usepackage{tikz, tikz-cd}
\usetikzlibrary{matrix, arrows, calc}

\usepackage[dvips]{epsfig}
\usepackage{pinlabel}
\usepackage{psfrag}

\newcommand{\ig}[2]{\vcenter{\xy (0,0)*{\includegraphics[scale=#1]{fig/#2}} \endxy}}
\newcommand{\igv}[2]{\vcenter{\xy (0,0)*{\reflectbox{\includegraphics[scale=#1, angle=180]{fig/#2}}} \endxy}}
\newcommand{\igc}[2]{\begin{center} \includegraphics[scale=#1]{fig/#2} \end{center}}

\newcommand{\plabel}[3]{{
\labellist
\small\hair 2pt
\pinlabel #1 at #2 #3
\endlabellist
}}

\IfFileExists{comments.sty}{\usepackage{comments}}

\RequirePackage{color}
\definecolor{myred}{rgb}{0.75,0,0}
\definecolor{mygreen}{rgb}{0,0.5,0}
\definecolor{myblue}{rgb}{0,0,0.65}

\RequirePackage{ifpdf}
\ifpdf
  \IfFileExists{pdfsync.sty}{\RequirePackage{pdfsync}}{}
  \RequirePackage[pdftex,
   colorlinks = true,
   urlcolor = myblue, 
   citecolor = mygreen, 
   linkcolor = myred, 
   pagebackref,
   bookmarksopen=true]{hyperref}
\else
  \RequirePackage[hypertex]{hyperref}
\fi

\newtheorem{thm}{Theorem}[section]
\newtheorem{lemma}[thm]{Lemma}

\newtheorem{prop}[thm]{Proposition}
\newtheorem{cor}[thm]{Corollary}
\newtheorem{claim}[thm]{Claim}

\newtheorem*{prop*}{Proposition}

\theoremstyle{definition}
\newtheorem{defn}[thm]{Definition}

\newtheorem{notation}[thm]{Notation}

\newtheorem{example}[thm]{Example}

\theoremstyle{remark}
\newtheorem{remark}[thm]{Remark}

\numberwithin{equation}{section}

  \def\gg{{\mathfrak g}}

    \def\RM{{\mathbb{R}}}

    \def\ZM{{\mathbb{Z}}}

    \def\CC{{\mathcal{C}}}
    \def\DC{{\mathcal{D}}}
    
    \def\FC{{\mathcal{F}}}
    
\def\HB{{\mathbf H}}

  \def\sb{{\mathbf s}}  
  \def\tb{{\mathbf t}}  \def\TC{{\mathcal{T}}}
    
  \def\vb{{\mathbf v}}  
    
  \def\xb{{\mathbf x}}  
  \def\yb{{\mathbf y}}  
  \def\zb{{\mathbf z}}

\let\phi=\varphi

\usepackage{bbm}

\def\Z{{\mathbbm Z}}

\def\1{\mathbbm{1}}

\newcommand{\ul}{\underline}

\newcommand{\mc}[1]{\mathcal{#1}}

\newcommand{\ot}{\otimes}
\newcommand{\pa}{\partial}
\newcommand{\co}{\colon}

\renewcommand{\to}{\rightarrow}

\newcommand{\fancyto}{\rightsquigarrow}
\newcommand{\downto}{\searrow} 
\newcommand{\upto}{\nearrow}

\newcommand{\define}{:=}

\renewcommand{\sl}{\mathfrak{sl}}

\newcommand{\Bim}{\textbf{Bim}}
\newcommand{\Kar}{\textbf{Kar}}
\newcommand{\Hom}{{\rm Hom}}
\newcommand{\HOM}{{\rm HOM}}
\newcommand{\End}{{\rm End}}
\newcommand{\END}{{\rm END}}

\newcommand{\Res}{{\rm Res}}
\newcommand{\Ind}{{\rm Ind}}

\newcommand{\ii}{\underline{\textbf{\textit{i}}}}
\newcommand{\jj}{\underline{\textbf{\textit{j}}}}
\newcommand{\SBim}{\mathbb{S}\textrm{Bim}}
\newcommand{\BSBim}{\mathbb{BS}\textrm{Bim}}
\newcommand{\fooBim}{g\mathbb{BS}\textrm{Bim}}
\newcommand{\mmB}{\mathfrak{B}}

\newcommand{\qtwo}{\left[2\right]}
\newcommand{\qJ}{\left[J\right]}
\newcommand{\qK}{\left[K\right]}

\newcommand{\downoneto}{\downharpoonright}
\newcommand{\uponeto}{\upharpoonright}

\newcommand{\JJ}{\underline{\textbf{\textit{J}}}}
\newcommand{\fooDC}{g\DC}
\newcommand{\JBim}{{J}\mathbb{BS}\textrm{Bim}}

\newcommand{\Zvv}{\ZM[v,v^{-1}]}

\newcommand{\amapsuptob}[2]{\xy  (0,-5)*+{#1}="1"; (0,5)*+{#2}="2"; {\ar@{|->} "1";"2"};\endxy}
\newcommand{\auptob}[2]{\xy  (0,-5)*+{#1}="1"; (0,5)*+{#2}="2"; {\ar@{->} "1";"2"};\endxy}

\let\tilde=\widetilde

\let\phi=\varphi
\let\epsilon=\varepsilon

\title[Thick Soergel Calculus]{Thicker Soergel Calculus in type $A$}
\author[Ben Elias]{or: How I learned to stop worrying and calculate some idempotents \\ \\ Ben Elias}
 
\begin{document} 

\maketitle

\begin{abstract} Let $R$ be the polynomial ring in $n$ variables, acted on by the symmetric group $S_n$. Soergel constructed a full monoidal subcategory of $R$-bimodules which
categorifies the Hecke algebra, whose objects are now known as Soergel bimodules. Soergel bimodules can be described as summands of Bott-Samelson bimodules (attached to sequences of
simple reflections), or as summands of generalized Bott-Samelson bimodules (attached to sequences of parabolic subgroups). A diagrammatic presentation of the category of Bott-Samelson
bimodules was given by the author and Khovanov in previous work. In this paper, we extend it to a presentation of the category of generalized Bott-Samelson bimodules. We also
diagrammatically categorify the representations of the Hecke algebra which are induced from trivial representations of parabolic subgroups.

The main tool is an explicit description of the idempotent which picks out a generalized Bott-Samelson bimodule as a summand inside a Bott-Samelson bimodule. This description uses a
detailed analysis of the reduced expression graph of the longest element of $S_n$, and the semi-orientation on this graph given by the higher Bruhat order of Manin and Schechtman.

This paper relies extensively on color figures. Some references to color may not be meaningful in the printed version, and we refer the reader to the online version which includes the
color figures. \end{abstract}
 
%
%

\setcounter{tocdepth}{1}
\tableofcontents
     
\newpage 

\section{Introduction} 
%

%
\subsection{Overview}
\label{subsec-overview}

The Hecke algebra $\HB$ associated to a Dynkin diagram $\Gamma$ is an algebra of fundamental importance. Its regular representation can be viewed as the decategorification of the
category $\mc{P}$ of $B$-equivariant perverse sheaves on the flag variety, or as the decategorification of the associated category $\mc{O}$, and therefore the Hecke algebra encodes
numerics associated to those categories. In the early 90s Soergel provided an additional categorification of the Hecke algebra, known as the category of Soergel Bimodules $\SBim$, which
is far more accessible than the other two approaches, and can be effectively used to study them both. Soergel bimodules are bimodules over a polynomial ring, and as such, are attractively
simple and explicit.

Let us restrict henceforth to finite type $A$ Dynkin diagrams. In \cite{EKho}, the author in conjunction with Mikhail Khovanov contributed to this explicit-ness by giving a diagrammatic presentation of $\SBim$. More precisely, we gave a diagrammatic presentation of the subcategory $\BSBim \subset \SBim$ of so-called \emph{Bott-Samelson bimodules}.
In this description, every morphism can be viewed as a linear combination of planar graphs with boundary (modulo graphical relations), and composition is given by stacking planar graphs
on top of each other. For a basic introduction to planar diagrammatics for monoidal categories, we recommend \cite[\S 4]{LauSL2}. Planar graphs now provide a simple way to
encode what are potentially very complicated maps of bimodules. (More recently, the author and Williamson \cite{EWGr4sb} have provided a similar diagrammatic calculus for all Coxeter
groups.)

This paper continues the elaboration of categorified Hecke theory on several related fronts: by explicitly finding certain important idempotents in $\BSBim$, by expanding the graphical
calculus of \cite{EKho} to the so-called \emph{generalized Bott-Samelson bimodules} $\fooBim$, and by giving a diagrammatic presentation of a categorification of induced trivial
representations from sub-Dynkin diagrams. Let us define the main players, before we provide some philosophy and motivation.

Let $R$ be the ring of polynomials in $n$ variables, equipped with its action of $W = S_n$. For a subset $J$ of the simple reflections $I$, which we call a \emph{parabolic subset}, we
let $R^J$ denote the subring of $R$ consisting of polynomials invariant under the simple reflections in $J$. Let $B_J$ be the $R$-bimodule $B_J \define R \ot_{R^J} R$, so that tensoring
with $B_J$ is isomorphic to the functor which restricts an $R$-module to $R^J$, and then induces it back to $R$. When $J=\{i\}$ is a singleton, denote the invariant ring $R^i$, and let
$B_i \define R \ot_{R^i} R$. Tensor products of various $B_i$ are known as \emph{Bott-Samelson bimodules}, and form a full monoidal subcategory $\BSBim$ of $R$-bimodules. Similarly,
tensor products of $B_J$ are \emph{generalized Bott-Samelson bimodules}, and form a category $\fooBim$.

\begin{remark} Technically, the ring $R$ is graded, and (generalized) Bott-Samelson bimodules are graded $R$-bimodules which differ from the above by certain grading shifts. We have
ignored the grading in this introduction. \end{remark}

The category $\SBim$ of \emph{Soergel bimodules} is the full (additive monoidal graded) subcategory of $R$-bimodules generated by all direct summands of Bott-Samelson bimodules. Soergel
has proven \cite{Soer07} that the isomorphism classes of indecomposable Soergel bimodules (up to grading shift) are parametrized by $W$; we denote them $\{B_w\}$. The bimodule $B_w$
appears as a summand inside $B_{i_1} \ot \ldots \ot B_{i_{\ell(w)}}$ for any reduced expression $s_{i_1} \cdots s_{i_{\ell(w)}}$ for $w$, and does not appear in any ``shorter"
Bott-Samelson bimodules. Soergel's proof uses a ``support filtration'' and is fairly technical. Using this support filtration, one can show that $B_{w_J} = B_J$, where $w_J$ is the
longest element of the parabolic subgroup generated by $J$ (see \cite{WillSingular}).

Thus $B_J$ will occur as a summand of $B_{i_1} \ot \ldots \ot B_{i_d}$ whenever $s_{i_1} \cdots s_{i_d}$ is a reduced expression for $w_J$. There is no known elementary formula for the
projection to this summand, written solely in terms of polynomials. The main result of this paper will be a diagrammatic construction of this projection.

The Grothendieck ring of $\SBim$ is isomorphic to the Hecke algebra of $S_n$. This isomorphism sends the classes $[B_i]$ and $[B_J]$ to the corresponding Kazhdan-Lusztig basis elements.
For more details, see \cite{Soer90,Soer92,Soer07}.

%
\subsection{Karoubi envelopes and thickening}
\label{subsec-philosophy}

Describing Soergel bimodules in terms of Bott-Samelson bimodules is an example of a process known as taking the \emph{Karoubi envelope} or the \emph{idempotent completion}. Given an
additive category $\CC$, one obtains the Karoubi envelope $\Kar(\CC)$ (roughly) by adding all direct summands as new objects. In this context, a ``direct summand" of an object $M \in
\CC$ is identified by the idempotent $e \in \End(M)$ which projects to it. Equivalently, if one considers $\CC$ as an algebroid, $\Kar(\CC)$ is isomorphic to the category of all
projective right modules over $\CC$. For more background on the Karoubi envelope, see \cite{BarMor}. In an abstract sense, $\CC$ contains all the information necessary to recover
$\Kar(\CC)$, and the two categories are Morita equivalent.

A general philosophy when studying a difficult-to-handle additive category is to study instead an easier, Morita-equivalent subcategory from which the original category can be recovered
via the Karoubi envelope. For example, indecomposable Soergel bimodules are difficult to compute with, and the endomorphism algebra of the sum of all indecomposable Soergel bimodules
currently defies description except in small cases. However, as noted above, the endomorphism algebra of the sum of all Bott-Samelson bimodules does have a useful description. We
mention two other examples of the same phenomenon.

\begin{remark} In geometric examples, this philosophy typically replaces the study of simple perverse sheaves with the study of pushforwards of constant sheaves from resolutions of
singularities. \end{remark}

\begin{example} Khovanov and Lauda \cite{KhoLau09, KhoLau11} use planar diagrams to present the morphisms between certain semisimple perverse sheaves on quiver varieties (via the work of
Varagnolo-Vasserot \cite{VarVas}). Rouquier \cite{Rouq2KM-pp} gives the same presentation, without the use of planar diagrams. The Karoubi envelope of this collection of
semisimple perverse sheaves contains all perverse sheaves, and thereby categorifies the positive half of the quantum group by work of Lusztig. However, the work of
Khovanov-Lauda-Rouquier gives an explicit version of this categorification, and allows for a direct and more general proof of categorification-related results. \end{example}

\begin{example} The category of representations of a complex semisimple lie algebra $\gg$ is semisimple, but still difficult to describe as a monoidal category. However,
in some cases the subcategory of tensor products of fundamental representations does admit a nice description via planar diagrams. For $\sl_2$, this is the Temperley-Lieb algebra,
given its diagrammatic presentation by Kauffman \cite{Kauf}. For rank 2 lie algebras, the corresponding diagrams are the spiders of Kuperberg \cite{Kupe}. Recently in \cite{CKM}, Cautis,
Kamnitzer, and Morrison have extended this presentation to type $A_n$ for $n>2$. \end{example}

The next step is to translate these successes into results about the interesting category $\Kar(\CC)$. To study any particular object in $\Kar(\CC)$, one must be able to realize it as
the image of some idempotent in $\CC$. One is led to the following question: given any indecomposable object in $\Kar(\CC)$, can one find an object $M \in \CC$ and an idempotent $e \in
\End(M)$ giving rise to it? Usually the object $M$ is obvious from the context, but the idempotent is difficult to compute.

\begin{example} In the Temperley-Lieb algebra, the idempotents in question are known as Jones-Wenzl projectors, and one has explicit recursive formulas to find them. For
$\mathfrak{sl}_n$ the idempotents are called clasps, and there are as yet no formulas for $n>3$. \end{example}

\begin{example} In the Khovanov-Lauda-Rouquier categorification of the positive half of quantum $\sl_2$, the idempotents are easy to find using the technology of Frobenius
extensions, see \cite{KhoLau09}. For $\sl_3$ the idempotents were uncovered by the wizardry of Sto\v{s}i\'{c} \cite{StosicSL3}. Beyond that nothing is known. \end{example}

Given a diagrammatic presentation of a category $\CC$, and an idempotent pair $(M,e)$ as above, it is easy to provide a diagrammatic presentation for the category $\CC(M,e)$ obtained by
formally adjoining the image of $e$ to $\CC$. We discuss this procedure in more depth in \S\ref{subsec-thickening}. On the level of morphisms, one adds a projection map $p \co M \to
(M,e)$ and an inclusion map $i \co (M,e) \to M$, with the obvious relations $ip = e$ and $pi = \1_{(M,e)}$. Finding $e$ is sufficient to provide a presentation for $\CC(M,e)$. Iterating
this procedure, one can find a presentation of any \emph{partial idempotent completion} for which one can describe all the idempotents added. Unfortunately, finding idempotents is very
difficult in general.

There may be other interesting consequences and formulas involving these new maps: for example, $(M,e)$ may also be a summand of some other object $M'$, and the inclusion map $(M,e) \to
M'$ may require a computation.

\begin{example} For Lauda's categorification \cite{LauSL2} of the entire quantum $\mathfrak{sl}_2$, a diagrammatic presentation of the Karoubi envelope was given by
Khovanov-Lauda-Mackaay-Sto\v{s}i\'{c} \cite{KLMS}. They refer to their new diagrammatics as a ``thick calculus," because the (images of the) new idempotents are represented by thick
lines. Their calculus also includes a variety of interesting formulas, such as the Sto\v{s}i\'{c} formula. \end{example}

More generally, one can think of replacing the diagrammatics for $\CC$ with the diagrammatics for a partial idempotent completion as \emph{thickening} the calculus; the calculus is not
as thick as possible until one adjoins every indecomposable in $\Kar(\CC)$. One hopes to provide useful formulas to aid computation in a partial idempotent completion, like the
Sto\v{s}i\'{c} formula.

Let us return to our original context, where $\CC = \BSBim$ and $\Kar(\CC) = \SBim$. Given a reduced expression $w = s_{i_1} \cdots s_{i_d}$ one knows that $B_w$ is a summand inside
$B_{i_1} \ot \cdots \ot B_{i_d}$, but finding a formula for this idempotent is an incredibly interesting and extremely difficult problem, for which a complete solution is currently out
of reach. One should expect that the idempotents may become arbitrarily complex for arbitrary $w \in W$, but that they might be computable for certain classes of $w \in W$.

\begin{remark} Since the original writing of this paper, the author and Williamson \cite{EWHodge} have proven the Soergel conjecture, which states that indecomposable Soergel bimodules
descend to the Kazhdan-Lusztig basis, when the category is defined over a field of characteristic zero. However, in finite characteristic the sizes of the indecomposable bimodules will
change, and so will their images in the Grothendieck group. Finding the idempotents explicitly will tell one which primes need to be inverted for the indecomposable bimodule to have its
``generic" size, and can help answer several questions in modular representation theory. As a motivating example, we point the reader to recent work of Williamson
\cite{WillCounter}, who constructs idempotents requiring certain Fibonacci numbers to be invertible, and uses this to disprove the Lusztig conjecture. \end{remark}

In this paper, we compute the idempotents mentioned above for reduced expressions of the longest element $w_J$ in a parabolic subgroup (in type $A$). In fact, given different reduced
expressions for $w_J$, we compute the corresponding map $B_{i_1} \ot \cdots \ot B_{i_d} \to B_{i_1'} \ot \cdots \ot B_{i_d'}$ which projects to the common summand $B_J$. This allows one
to give a concise and reasonably elegant diagrammatic presentation for the partial idempotent completion $\fooBim$, a thickening of the original calculus of \cite{EKho}.

%
\subsection{Manin-Schechtman theory}
\label{subsec-mstheory}

To produce the idempotent corresponding to a parabolic subset $J$, we utilize Manin-Schechtman theory. Manin and Schechtman in \cite{ManSch} provide a beautiful and detailed study of a
certain $n$-category associated with the symmetric group $S_n$. This $n$-category produces a collection of posets known as
the \emph{higher Bruhat orders}, because the most basic such poset is the symmetric group itself with its usual (weak left) Bruhat order.

To describe the next most basic poset, consider the set of all reduced expressions for the longest element $w_0 \in S_n$. This set can be given the structure of a graph by placing an
edge between two reduced expressions if they are related by a single braid relation. Manin and Schechtman equip this graph with a specific \emph{semi-orientation}: edges corresponding
to the commutation $su=us$ of two commuting simple reflections are unoriented, and called \emph{equivalences}, while edges corresponding to the braid relation $sts=tst$ are oriented.
They prove that their directed graph has a unique sink and a unique source up to equivalence. The induced order on equivalence classes of reduced expressions is the first higher Bruhat
order.

For the reader's edification, here is a brief description of the higher Bruhat order. Fix the usual total order on the set $X = \{1, \ldots, n\}$. Let $P_k(X) = \{ I \subset X \textrm{
such that } I \textrm{ has size } k\}$, which inherits a lexicographic order. An element $w \in S_n$ can be interpreted as an order on $P_1(X) \cong X$, and its inversion set can be
thought of as those $I \in P_2(X)$ such that the induced order on $P_1(I)$ is antilexicographic. The edges in the (weak left) Bruhat graph are induced by the following operation: find
$I \in P_2(X)$ such that $P_1(I)$ is an interval in the order, and flip the elements in $P_1(I)$ from lexicographic to antilexicographic. This is the same as multiplying $w$ by a simple
reflection on the left. Meanwhile, a reduced expression of $w$ is an order on its inversion set inside $P_2(X)$, saying in which order the inversions were created by simple reflections.
For example, the reduced expression $sts$ adds the inversions $\{(12), (13), (23)\}$ in lexicographic order, while $tst$ adds them in antilexicographic order. Now, the higher inversion
set of a reduced expression can be thought of as those $I \in P_3(X)$ such that the induced order on $P_2(I)$ is antilexicographic. The higher Bruhat order is induced by the following
operation: find $I \in P_3(X)$ such that $P_2(I)$ is an interval in the order, and flip $P_2(I)$ from lexicographic to antilexicographic. There are many interesting subtleties here
(such as the equivalence relation) which we sweep under the rug. Ultimately, we will not need the details of their beautiful construction in this paper.

Let $S_m \subset S_n$ be the parabolic subgroup which permutes the subset $\{i+1, \ldots, i+m\} \subset \{1, \ldots, n\}$. Then there are numerous embeddings of the reduced expression
graph for the longest element $w_{0,m}$ of $S_m$ into the reduced expression graph for the longest element $w_{0,n}$ of $S_n$. There is one embedding for every way to extend $w_{0,m}$
to $w_{0,n}$ by adding simple reflections on the left and right. For example, the graph for $w_{0,3}$ is a single oriented edge, and every oriented edge in the reduced expression graph
of $w_{0,n}$ comes from some parabolic embedding $S_3 \subset S_n$. The Manin-Schechtman order is \emph{parabolic compatible} in the sense that these embeddings all preserve orientation
(when $\{i+1, \ldots, i+m\}$ is also given its usual total order).

Another property of the Manin-Schechtman order is that it is \emph{monoidal} or \emph{local}. Whenever a reduced expression has two braid relations which can be applied in disjoint
parts of the expression, there is a corresponding square inside the reduced expression graph, which we call a \emph{disjoint square}. The monoidal property states that a disjoint square
has a parallel orientation. In other words, whether a given braid relation is to be applied ``forwards" or ``backwards" is independent of the application of braid relations to distant
parts of the expression. That the Manin-Schechtman order is monoidal is to be expected, given that this order is part of a higher $n$-category.

To any path in a reduced expression graph, we can associate a morphism between the Bott-Samelson bimodules corresponding to the start and end expressions. Equivalences are sent to
isomorphisms. However, oriented edges are not sent to isomorphisms. In this paper we focus on the following two properties of the Manin-Schechtman orientation. \begin{itemize} \item The
orientation is \emph{consistent with Bott-Samelson bimodules}, or \emph{BS-consistent}. For oriented paths, the associated morphism does not depend on the oriented path chosen, only on
the start and end of the path! The same holds for reverse-oriented paths. \item The orientation on the reduced expression graph for $w_0$ is \emph{idempotent-magical}. The morphism
associated to a path which goes in oriented fashion from source to sink and then in reverse-oriented fashion from sink to source is in fact the idempotent projecting to $B_{w_0}$! This
property may fail for other elements $w \in S_n$! \end{itemize}

The proof that the Manin-Schechtman orientation is BS-consistent goes as follows. For $S_4$, BS-consistency is precisely the most interesting relation in the diagrammatic calculus for
Bott-Samelson bimodules \cite{EKho}, the \emph{Zamolodchikov relation}. Meanwhile, it is proven in \cite{ManSch} that the cycles in a reduced expression graph (ignoring cycles which
disappear when the non-oriented edges are contracted) are generated by parabolic embeddings of $S_4$ and by disjoint squares. Thus, BS-consistency follows from being both monoidal and
parabolic compatible. The complete proof (dealing with the contracted cycles as well) is given in \S\ref{subsec-expressions}.

The Manin-Schechtman orientation is not the only BS-consistent orientation. (We restrict our attention to semi-orientations with a unique source and sink modulo equivalence.) For the
longest element of $S_4$, there are two such orientations (and their reversals). For $S_5$, there are four. Most orientations will not be BS-consistent. As noted in the
previous paragraph, whether an orientation is BS-consistent or not is a combinatorial question, pertaining to the disjoint squares and parabolic embeddings of $S_4$ inside the graph.
The combinatorics are rather interesting, and deserve further study.

The idempotent-magical property is very special and quite surprising. This property is also shared by the other BS-consistent orientation for $S_4$, and it is unknown
whether or not it holds for the other BS-consistent orientations in general. There is currently no understanding for what makes an idempotent-magical orientation special, or why any
orientation should have this property to begin with.

The proof that the Manin-Schechtman orientation is idempotent-magical comprises the bulk of this paper, and it is our main technical result. As far as the author is aware, this is the
first genuine application of Manin-Schechtman's orientation (and not just the rough structure of the graph) to representation theory. We use a concrete description of certain oriented
paths, and perform very explicit computations to verify the result. The Manin-Schechtman orientation is rather convenient for this. More combinatorial work would be required to prove
the result (in the same fashion) for other orientations, though this is theoretically possible.

\begin{remark} There is a Manin-Schechtman orientation on the reduced expression graphs of arbitrary elements of $S_n$, not just the longest element. This orientation is always
BS-consistent, by the argument above, but is rarely idempotent-magical. For example, there are elements which have only a single equivalence class of reduced expressions, but whose
Bott-Samelson bimodules are not indecomposable. \end{remark}

A similar situation seems to occur in type $B$: I conjecture that one can place a (natural) orientation on the reduced expression graph of the longest element which is
parabolic-compatible, monoidal, BS-consistent, and idempotent-magical. This has led the author to conjecture the existence of higher Bruhat orders associated to type $B$ and possibly to
other wreath products, and to encourage this study more seriously. Some early work in this direction has recently appeared in \cite{SheSuh}.

However, the theory appears to break down beyond these cases. A computational result from \cite{EWGr4sb} shows that the reduced expression graph of the longest element of $H_3$ does not
admit a BS-consistent orientation with a unique source and sink. An unpublished result of the author (essentially, an extremely long exercise) shows that the reduced expression graph of
the longest element of $D_4$ also does not admit a BS-consistent, monoidal orientation! The existence of BS-consistent and idempotent-magical orientations should imply something
about the geometry of the flag variety, though what this says about the flag variety in types $A$ and $B$ which fails in type $D$ is a complete mystery. Much more study is required.

Nicolas Libedinsky independently studied a similar question in \cite{LibNB}, where he looks at morphisms induced by paths in the expression graph for an arbitrary element $w \in W$, but
in the context not of the symmetric group but of \emph{extra large} Coxeter groups ($m_{s,t}>3$ for any simple reflections $s,t$). For these Coxeter groups the expression graph is quite
simple. Libedinsky shows that the morphism corresponding to any path which hits every reduced expression is an idempotent, and that these idempotents pick out a single isomorphism class
of direct summand. However, this summand is typically not indecomposable.

%
\subsection{Singular Soergel bimodules and induced Hecke modules}
\label{subsec-inducedandssbim}

The entire story of Soergel bimodules should be viewed in the larger context of singular Soergel bimodules, as introduced by Geordie Williamson in his thesis \cite{WillSingular}.
Singular Soergel bimodules form a 2-category $\mmB$: the objects are rings $R^J$ for each parabolic subset $J \subset \Gamma$, and the Hom categories are some full subcategories of
$(R^J,R^K)$-bimodules. Soergel bimodules form the endomorphism category of the object $R=R^{\emptyset}$ inside the 2-category of singular Soergel bimodules. The Grothendieck category of
$\mmB$ is isomorphic to the \emph{Hecke algebroid}, an ``idempotented" version of the Hecke algebra. Note that the objects of the Hecke algebroid are also parametrized by parabolic
subsets. These results are due to Williamson \cite{WillSingular}, and more details can be found there.

The \emph{trivial module} of the Hecke algebra is a particular one-dimensional representation. Let $T_J$ denote the induction to $\HB$ of the trivial module of the sub-Hecke-algebra
$\HB_J$ attached to a parabolic subgroup. It is not difficult to show that $T_J$ is isomorphic to $\Hom(J,\emptyset)$ inside the Hecke algebroid, viewed as a module over
$\Hom(\emptyset, \emptyset) = \HB$. Thus, it is categorified by the singular Soergel bimodule category $\Hom_{\mmB}(R^J,R)$.

This paper provides a diagrammatic presentation of $\JBim$, the full subcategory of $J$-\emph{singular Bott-Samelson bimodules}, which has $\Hom_{\mmB}(R^J,R)$ as its idempotent
completion. A $J$-singular Bott-Samelson bimodule is just a usual Bott-Samelson $R$-bimodule with scalars restricted to $R^J$ on the right. There is a (non-full) faithful functor from
$\Hom(R^J,R)$ to $\Hom(R,R)=\SBim$ given by inducing on the right from $R^J$ back up to $R$. Using the idempotent which picks out $B_J$, we describe the image of this functor
diagrammatically, which allows us to prove that our presentation of $\JBim$ is correct. However, the diagrammatic category for $\JBim$ is itself quite simple, and can be enjoyed in the
absence of any knowledge of the idempotent or its properties, or any knowledge of singular Soergel bimodules.

Finding a diagrammatic presentation of (a Morita-equivalent sub-2-category of) the entire 2-category of singular Soergel bimodules is a more difficult question, and is
work in progress between the author and Williamson. The functor between the description given here of the small fragment $\JBim$ and the description in progress of the entire 2-category
is straightforward.

Having a categorification of induced trivial modules is a first step towards the categorification of Hecke representation theory in general. Much about the representation theory of the
Hecke algebra can be understood using induced trivial and sign modules, and the Hom spaces between them. It is one of the author's goals to categorify this whole picture.

One should also point out that induced modules have been categorified before in much more generality, in the context of category $\mc{O}$, by Stroppel and Mazorchuk \cite{MazStr}. It is
very likely that the category of this paper should be related to their categorification of the induced trivial module by applying Soergel's functor.

\vspace{0.06in}

{\bf Structure of the paper.}

Chapter \ref{sec-background} contains background material. In \S\ref{subsec-hecke} we discuss the Hecke algebra. In \S\ref{subsec-soergel} we elaborate on the
Soergel-Williamson categorification, which was partially described in the introduction. In \S\ref{subsec-soergeldiagrammatics} we recall the diagrammatic category defined in
\cite{EKho}, which is equivalent to the category of Bott-Samelson bimodules. Finally, in \S\ref{subsec-thickening} we discuss in detail how partial idempotent completions work, and
outline what needs to be computed to thicken our calculus.

In Chapter \ref{sec-calcs} we perform these computations, and find the idempotent for $B_J$. In \S\ref{subsec-expressions} we discuss reduced expression graphs and the
Manin-Schechtman semi-orientation, and prove that this orientation is BS-consistent. In \S\ref{subsec-proofoutline} we outline the proof that the orientation is idempotent-magical, and
prove it modulo some technical lemmas. The rest of the chapter is a series of awful, terrible computations. In \S\ref{subsec-longest} we begin the process of proving these
technical lemmas, setting notation in place for certain reduced expressions and paths in the reduced expression graph. In \S\ref{subsec-thicktrivalent} we describe the extremely
important morphism which will become the ``thick trivalent vertex.'' In \S\ref{subsec-nastiness} we prove the technical lemmas, and in \S\ref{subsec-projectors} we finish the proof that
the idempotent we construct has the desired properties.

In Chapter \ref{sec-augmented} we use the results of \S\ref{sec-calcs} to provide diagrammatics for $\fooBim$. It should be accessible without reading \S\ref{sec-calcs}, which the
reader uninterested in the specifics is welcome to skip entirely. We augment the diagrammatics by adding new diagrams for interesting morphisms in $\fooBim$, such as the thick trivalent
vertices, and discuss some relations involving these new diagrams.

After all this work, it is quite easy in Chapter \ref{sec-induced} to provide diagrammatics for $\JBim$, the categorification of the induced trivial module $T_J$.

The appendix contains a list of notation.

\vspace{0.06in}

{\bf Acknowledgments.}

The author was supported by NSF grants DMS-524460 and DMS-524124 and DMS-1103862, and would like to thank Mikhail Khovanov for his suggestions, and Geordie Williamson for many useful discussions. The author also thanks the anonymous referee for many valuable comments, and for pointing out the important missing case \eqref{a6z}.

\section{Background}
\label{sec-background}

We expect the reader to be familiar with \cite{EKho}. It will be enough to have glanced through Chapters 2 and 3 of that paper; aside from that, we use only once the concept of
one-color reduction, found in Chapter 4. Alternatively (perhaps preferably, thanks to the advent of color) one should read \cite{ETemperley} chapters 2.1, 2.3 and 2.4. We change the
notation for the names of the categories, change the convention for grading shifts, and change the variable $t$ to $v$ in order to be consistent with more recent papers and the papers
of Soergel (wherein, for those who understand, $t=q^{\frac{1}{2}}$ and $v=q^{-\frac{1}{2}}$; in \cite{EKho} we used $t$ for $q^{-\frac{1}{2}}$ in some preliminary versions).

We will not discuss singular Soergel bimodules in this paper. Nonetheless, this additional context may be useful to the reader: we recommend the first three pages of the introduction to
\cite{WillSingular}. We change notation in that we use the letter $b$ to denote Kazhdan-Lusztig elements, instead of $\underline{H}$.

The results below which are attributed to Soergel may be found in \cite{Soer07}. Some of these results appeared in \cite{Soer90, Soer92} first, but \cite{Soer07} has the advantage of
being purely algebraic. The results below which are attributed to Williamson may be found in \cite{WillSingular}. More information on any of the topics in the first two sections can be
found in those papers.

Let us fix some notation pertaining to the symmetric group. Additional notation will be introduced in the various sections of this chapter. We have provided a list of notations in the
appendix.

Let $I=\{1,\ldots,n\}$ index the vertices of the Dynkin diagram $A_n$. Elements of $I$ will be called \emph{indices} or \emph{colors}. The terms \emph{distant} and \emph{adjacent} refer
to the relative position of two indices in the Dynkin diagram, not in any word or picture. Let $W=S_{n+1}$ be the Coxeter group, with simple reflections $s_i$, $i \in I$. Let $w_0$
denote the longest element. For a \emph{parabolic subset} $J \subset I$, let $W_J$ denote its parabolic subgroup. We let $w_J$ be the longest element of $W_J$, and $d_J$ its length. We
say that an index $i$ is \emph{distant} from a parabolic subset $J$ if it is distant from all the indices in $J$, and we say two parabolic subsets are \emph{distant} if all the indices
in one are distant from the indices in the other.

We let $\qJ$ denote the Hilbert polynomial of $W_J$, defined by $v^{-d_J}\sum_{w \in W_J}v^{2l(w)}$. It was denoted $\pi(J)$ in \cite{WillSingular}. This Hilbert polynomial is a product
of quantum numbers. For instance, the Hilbert polynomial of $S_n$ is $\left[n\right]!$, quantum $n$ factorial. Recall that $[n] = v^{-n+1} + v^{-n+3} + \ldots + v^{n-3} + v^{n-1}$, and
that $[n]! = [n][n-1]\cdots[1]$.

%
\subsection{The Hecke algebra}
\label{subsec-hecke}

The Hecke algebra $\HB$ has a presentation as an algebra over $\Zvv$ with generators $b_i$, $i \in I$ and the \emph{Hecke relations}
\begin{eqnarray}
b_i^2 & = & (v + v^{-1}) b_i  \label{bisq}\\
b_ib_j & = & b_jb_i \ \textrm{ for distant } i, j \label{bidistantbj}\\
b_ib_jb_i + b_j & = & b_jb_ib_j + b_i \textrm{ for adjacent } i, j \label{biadjacentbj}.
\end{eqnarray}
The subalgebra $\HB_J$ is generated by $b_j$, $j \in J$. It can also be described by generators and relations in the same way. Note that $\HB$ is free over $\Zvv$.

There is a basis for $\HB$ known as the \emph{Kazhdan-Lusztig basis}, with one element $b_w$ for each $w \in W$. We will not write down this basis explicitly in terms of the generators,
nor is it easy to do so. However, we note that when $w=w_J$ is the longest element of a parabolic subgroup, the corresponding basis element $b_J \define b_{w_J}$ does have a simple
presentation in another basis known as the \emph{standard basis}. When $J=\{i\}$ is a singleton, $b_J=b_i$. When $J$ is empty, $b_\emptyset=1$.

We will not need any features of the standard basis in this paper, beyond the implications for the Kazhdan-Lusztig basis which we mention in this section.

We also note that $b_J$ is contained inside the subalgebra $\HB_J$. This immediately implies two generalizations of \eqref{bidistantbj}.
\begin{equation} b_i b_J = b_J b_i \textrm{ for any } i \textrm{ distant from } J \label{bJcommutebi}. \end{equation}
\begin{equation} b_J b_K = b_K b_J \textrm{ for distant } J,K \label{bJcommutebK}. \end{equation}
Working with standard bases, it is also very easy to prove the following, which explains \eqref{bJcommutebK}.
\begin{equation} b_J b_K = b_{J \coprod K}  \textrm{ for distant } J,K \label{bJdistantbK}. \end{equation}

The Hecke algebra has a \emph{trivial representation} $T$, a left $\HB$-module. It is free of rank 1 as a $\Zvv$-module, and $b_i$ acts by multiplication by $[2] = v+v^{-1}$. The
element $b_J$ acts on $T$ by multiplication by $[J]$, a fact which is obvious when the Hecke algebra is described in terms of the standard basis, see \cite[Lemma 2.2.3]{WillSingular}.
Meanwhile, one knows (see \cite[Theorem 6.6]{LuszUnequal14}) that $b_i b_w = (v+v^{-1}) b_w$ whenever $w$ has $s_i$ in its left descent set. This implies that the $\Zvv$-span of $b_K$
is a realization of the trivial representation of the subalgebra $\HB_J$, whenever $J \subset K$. From this we have the following generalizations of \eqref{bisq}.
\begin{equation} b_i b_J = b_J b_i = (v + v^{-1}) b_J \textrm{ for any } i\in J \label{bJproperty} \end{equation}
\begin{equation} b_J b_K = b_K b_J = \qJ b_K \textrm{ whenever } J \subset K \label{bJbK} \end{equation}

Conversely, any $x \in \HB$ for which $b_j x = [2] x$ for all $j \in J$ also satisfies $b_J x = \qJ x$, and thus must live in the right ideal of $b_J$ (at least up to scalar). If $x$ is
in $\HB_J$ as well, then $x$ is a multiple of $b_J$.

The trivial representation of $\HB$ itself is realized inside $\HB$ as the span of $b_{w_0}$, which is also an ideal. A less obvious fact is that the left ideal of $b_J$ is a
realization of the induction $T_J \define \Ind_{\HB_J}^\HB T$ from the trivial module of $\HB_J$. This also becomes more obvious in the standard basis, see \cite{WillSingular}.

%

%
\subsection{The Soergel-Williamson categorification}
\label{subsec-soergel}
%

Let $R=\Bbbk[f_1,\ldots,f_n]$ be the coordinate ring of the geometric representation of $W$, where $f_i$ are the simple roots and $\Bbbk$ is a field of characteristic $\ne 2$. The ring
$R$ is graded with $\deg(f_i)=2$. If $M = \oplus M^i$ is a graded $R$-module then the grading shift convention will be $M(i)^j = M^{i+j}$. All $R$-modules in this paper will be graded.

In order to categorify the Hecke algebra within the category of $R$-bimodules, we may wish to find $R$-bimodules $B_i$, $i \in I$, which satisfy

\begin{eqnarray}
B_i \otimes B_i & \cong & B_i(1) \oplus B_i(-1) \label{Bisq} \\
B_i \otimes B_j & \cong & B_j \otimes B_i \label{BidistantBj} \textrm{ for distant } i, j \\
B_i \otimes B_j \otimes B_i \oplus B_j & \cong & B_j \otimes B_i \otimes B_j \oplus B_i \label{BiadjacentBj} \textrm{ for adjacent } i, j. \end{eqnarray}

\noindent One could then define a map from $\HB$ to the split Grothendieck ring of $R$-bimodules by sending $b_i$ to $[B_i]$. In order to categorify various aspects of the Hecke
algebroid, one may also seek $R$-bimodules $B_J$ for each parabolic set $J$, which satisfy

\begin{equation} B_J \ot B_i \cong B_i \ot B_J \cong B_J(1) \oplus B_J(-1) \textrm{ whenever } i \in J \label{BJBi}.\end{equation}
\begin{equation} B_J \ot B_K \cong B_K \ot B_J \cong \qJ B_K \textrm{ whenever } J \subset K \label{BJBKinside} \end{equation}
\begin{equation} B_J \ot B_K \cong B_K \ot B_J \cong B_{J \coprod K} \textrm{ for distant } J,K \label{BJdistantBK}.\end{equation}

\noindent Here, we use the shorthand that $\qJ B_K$ indicates a direct sum of many copies of $B_K$ with the appropriate degree shifts. The isomorphism (\ref{BJBi}) could be rewritten
$B_J \ot B_i \cong \qtwo B_J$, for instance. The map from $\HB$ to the Grothendieck ring should send $b_J$ to $[B_J]$.

In fact, these modules exist, and we described them briefly in the introduction. Let us recall the construction again. For a more thorough study of this material, see Williamson's
thesis \cite{WillSingular}.

Consider the 2-category (usually called $\Bim$ in the literature) where the objects are graded rings, and the morphism category between two rings $\Hom(S',S)$ is the category of graded
$(S,S')$-bimodules. Composition of 1-morphisms is given by the tensor product of bimodules. When we write $\HOM(X,Y)$ for $X,Y$ two objects in a graded category, we refer to the graded
vector space which is $\oplus_{n \in Z} \Hom(X,Y(n))$. Note that when $X,Y$ are graded $(S,S')$-bimodules for two commutative rings $S,S'$, the graded vector space $\HOM(X,Y)$ is also a
graded $(S,S')$-bimodule. We will now describe several 2-subcategories of $\Bim$.

The ring $R$ has a natural action of $W$, and so for each parabolic subset $J$ we have a subring $R^J \define R^{W_J}$ of invariants. When $J=\{i\}$ is a single element subset, we
abbreviate the invariant ring as $R^i$. Tensoring on the left with $R$, viewed as an $(R^J,R)$-bimodule, is the same as the restriction functor from $R$-modules to $R^J$-modules.
Tensoring on the left with $R$, viewed as an $(R, R^J)$-bimodule, is the same as the induction functor from $R^J$-modules to $R$-modules.


Let $B_i \define R \ot_{R^i} R(1)$, which restricts from $R$ down to $R^i$, and then induces back to $R$. For a sequence $\ii = i_1 i_2 \ldots i_{d(\ii)}$ of indices of length $d(\ii)$,
let $B_{\ii} \define B_{i_1} \ot_R \cdots \ot_R B_{i_{d(\ii)}}$. These $B_{\ii}$ are called \emph{Bott-Samelson bimodules}, and together with their grading shifts, form a full monoidal
graded subcategory $\BSBim$ of $\Hom_{\Bim}(R,R)$. Let us mention now that we will write $s_{\ii}$ for the product $s_{i_1} \cdots s_{i_{d(\ii)}}$, an element of $W$.

Let $B_J \define R \ot_{R^J} R(d_J)$, which restricts from $R$ down to $R^J$, and then induces back to $R$. It is clear that the objects $B_J$ are indecomposable as $R$-bimodules, being
generated by their unique term $1 \ot 1$ in minimal degree. For a sequence $\JJ = J_1 J_2 \ldots J_{d(\ii)}$ of parabolic subsets, we let $B_{\JJ} \define B_{J_1} \ot_R \cdots \ot_R
B_{J_{d(\JJ)}}$. These $B_{\JJ}$ are called \emph{generalized Bott-Samelson bimodules}, and together with their grading shifts, also form a full monoidal graded subcategory $\fooBim$ of
$\Hom_{\Bim}(R,R)$.


The category of \emph{Soergel bimodules} $\SBim$ is obtained by taking all direct sums and summands of objects in $\BSBim$. The objects of $\fooBim$ are, in fact, Soergel bimodules. Since Soergel bimodules form an idempotent-closed subcategory of a bimodule category, it has the Krull-Schmidt property and its Grothendieck group will be generated freely by
indecomposables. Note that none of the categories mentioned are abelian; \emph{Grothendieck group} always refers to the additive or split Grothendieck group.

%

It is not hard to see that the bimodules $B_i$ and $B_J$ satisfy the isomorphisms \eqref{Bisq} through \eqref{BJdistantBK} above, and we can even be explicit. The isomorphism
\eqref{Bisq} can be deduced from the fact that, as an $R^i$-bimodule, $R \cong R^i \oplus R^i(-2)$. This isomorphism is given explicitly using the \emph{Demazure operator} $\partial_i
\colon R \to R^i$, where $\partial_i(P)=\frac{P - s_i(P)}{f_i}$. This operator is $R^i$-linear, has degree -2, and sends $f_i \mapsto 2$ and $R^i \subset R$ to zero. The two projection
operators from $R$ to $R^i$ are $P \mapsto \partial_i(P)$ and $P \mapsto \partial_i(Pf_i)$, of degrees -2 and 0 respectively; the inclusion operators are $Q \mapsto \frac{1}{2}f_i Q$
and $Q \mapsto \frac{1}{2}Q$, of degrees +2 and 0 respectively. We invite the reader to figure out the projections and inclusions for \eqref{Bisq}, or to look them up in \cite{EKho}. The
Demazure operator also provides an easy way to define the projection and inclusion operators for the splitting in \eqref{BJBi}.

In similar fashion, the isomorphism \eqref{BJBKinside} comes from the fact that, as an $R^J$-bimodule, $R \cong \qJ R^J$ up to an overall grading shift. The isomorphism is given
explicitly by choosing dual bases for $R$ over $R^J$ with respect to the Demazure operator $\pa_{w_J}$. More details will be found later, in the discussion before
\eqref{thickthickidemp}.

The isomorphism \eqref{BiadjacentBj} was studied closely and made explicit in \cite{EKho}.

Consider the isomorphism of \eqref{BJdistantBK}, which is \[R \ot_{R^J} R \ot_{R^K} R (d_J+d_K) = B_J \ot B_K \cong B_{J \coprod K} = R \ot_{R^{J \coprod K}} R (d_J+d_K)\] where
$J,K$ are distant. Note that any polynomial in $R$ can be decomposed into polynomials symmetric in either $W_J$ or $W_K$, so that the left side is spanned by elements $f \ot 1 \ot g$
for $f,g \in R$. The isomorphism sends $f \ot 1 \ot g \mapsto f \ot g$.

The following theorem is due to Soergel.

\begin{thm} (See \cite[Proposition 5.7]{Soer07} or \cite[Theorem 4.1.5]{WillSingular}) The Grothendieck ring of $\SBim$ is isomorphic to $\HB$, where the isomorphism is defined by sending
$b_i$ to $[B_i]$ and $v$ to $[R(1)]$. Under this isomorphism, $b_J$ is sent to $[B_J]$. Any Soergel bimodule is free as a left (resp. right) $R$-module. \end{thm}


%

The Hecke algebra is equipped with a canonical $\Zvv$-linear \emph{trace} map $\epsilon \colon \HB \to \Zvv$, which picks out the coefficient of $1$ in the standard basis of $\HB$. It
satisfies $\epsilon(b_{\ii})=v^{d(\ii)}$ when $\ii$ has no repeated indices, and $\epsilon(b_J)=v^{d_J}$. The Hecke algebra also has an antilinear
antiinvolution $\omega$, satisfying $\omega(v^a b_{\ii})=v^{-a}b_{\omega(\ii)}$ where $\omega(\ii)$ is the sequence run in reverse. Together, these induce a pairing on $\HB$ via
$(x,y)=\epsilon(y\omega(x))$.

\begin{prop} (See \cite{Soer07} Theorem 5.15) This canonical pairing is induced by the categorification $\SBim$, in that the graded rank of $\HOM_{\SBim}(X,Y)$ as
a free left (or right) $R$-module is precisely $([X],[Y])$.  \end{prop}


Putting together the previous theorems, we can state the following useful lemma, which tells us when an idempotent in $\End(B_{\ii})$ will pick out $B_J$.

\begin{lemma} Let $X$ be a summand of $B_{\ii}$ for some reduced expression $\ii$ of $w_J$. Suppose that $X$ satisfies $X \ot B_i \cong X(1) \oplus X(-1)$ for all $i \in J$. Then the
class of $X$ in $\HB$ is a scalar multiple of $b_J$, and this scalar is the graded rank of the free $R$-module $\HOM(R,X)$, divided by $v^{d_J}$. In particular, if $\HOM(R,X)$ and
$\HOM(R,B_J)$ have the same graded rank, then $X \cong B_J$. \label{isitBJ} \end{lemma}

\begin{proof} Whatever $[X]$ is, it is contained in $\HB_J$ and satisfies $[X]b_i=(v + v^{-1}) [X]$, so $[X]$ must actually be a scalar multiple of $b_J$ (as
discussed in \S\ref{subsec-hecke}). Since $\HOM(R,X)$ has graded rank $\epsilon([X])$ and $\epsilon(b_J)=v^{d_J}$, we can determine the scalar from this graded rank by dividing by
$v^{d_J}$. If $[X]=b_J$ then by computing the graded rank of $\END(X)$ we know that $X$ is indecomposable, and it must therefore be isomorphic to $B_J$. \end{proof}

There is one more category to describe, in order to categorify the induced trivial module $T_J$ in a combinatorial way. Let $\JBim$, the category of \emph{$J$-singular Bott-Samelson
bimodules}, denote the full subcategory of $(R,R^J)$-bimodules given by objects $\Res B_{\ii}$, where the Bott-Samelson bimodule $B_{\ii}$ is restricted to become an $R^J$-module on the
right. These are special examples of singular Soergel bimodules, as described in the introduction. Williamson has generalized the results of Soergel above to singular Soergel bimodules
\cite{WillSingular}. To avoid needing to discuss singular Soergel bimodules in this paper, we only sketch a proof the following claim.

\begin{claim} The idempotent completion of $\JBim$ is the category of singular Soergel bimodules inside $(R,R^J)$-bimodules. In particular, $\JBim$ categorifies $T_J$.
\label{idempotentcompletionofsingularBS} \end{claim}

\begin{proof} Using the classification theorem of indecomposable singular Soergel bimodules (see \cite[Theorem 5.4.2]{WillSingular}) we know there is one indecomposable for each coset
of $W_J$ in $W$. Let $s_{\ii}$ be a reduced expression for the minimal element $w$ of that coset, and consider the restriction of $B_{\ii}$. Using the support filtration (see
\cite{WillSingular}), it is clear that this restriction has the indecomposable corresponding to that coset appearing as a summand with multiplicity 1. (Thanks to Williamson for this
quick proof.)

That this is a categorification of $T_J$ comes from \cite[Theorem 4.1.5]{WillSingular}. \end{proof}

The remaining facts about $\JBim$ that we will use can be quoted directly from \cite[Theorems 4.1.5 and 5.2.2]{WillSingular}. 

\begin{claim} Morphism spaces between two objects in $\JBim$ are free as left $R$-modules and right $R^J$-modules. There is a formula for the graded ranks of $\HOM$ spaces in $\JBim$, analogous to the Proposition above for Soergel bimodules.

The functor which induces from $R^J$ to $R$ on the right is fully faithful after base change. That is, for $X, Y \in \JBim$ we have \[\HOM_{(R,R)}(X \ot_{R^J} R, Y \ot_{R^J} R) \cong \HOM_{(R,R^J)}(X,Y) \ot_{R^J} R. \]  \end{claim}

We will provide diagrammatics for the category $\JBim$ in Chapter \ref{sec-induced}.

%
\subsection{Soergel diagrammatics}
\label{subsec-soergeldiagrammatics}
%
%
%

In \cite{EKho}, the author and M. Khovanov give a diagrammatic presentation of a category $\DC$ by generators and relations. A functor $\FC \co \DC \to \BSBim$ was constructed, and it was shown that $\FC$ is an equivalence when the base field $\Bbbk$ is of characteristic not equal to $2$.

Technically, morphisms in $\DC$ are graded vector spaces, which describe the space $\HOM(X,Y)$ between two Bott-Samelson bimodules. We write Hom spaces in $\DC$ as $\HOM_{\DC}(X,Y)$
accordingly. Obtaining a graded category, whose morphisms all have degree $0$, is a trivial process. We will also ignore the difference between $\DC$ and its additive closure, speaking
freely about direct sums of objects in $\DC$.

What follows is a brief summary of sections 2.3 and 2.4 of \cite{ETemperley}. If this is the reader's first encounter with Soergel diagrammatics, we recommend reading those sections instead, as a
better introduction. We will assume that many of the ideas found there are known to the reader, including biadjointness, dot forcing rules, idempotent decompositions, and one-color
reductions. One can also see \cite{EKho} for a less pretty version, but a version where the equivalence $\FC$ between $\DC$ and $\BSBim$ is explicitly defined.

An object in the diagrammatic version of $\BSBim$ is given by a sequence of indices $\ii$, which is visualized as $d$ points on the real line $\RM$, labelled or ``colored'' by the
indices in order from left to right. Morphisms from $\ii$ to $\jj$ are planar graphs in $\RM \times [0,1]$, with each edge colored by an index,
with bottom boundary $\ii$ and top boundary $\jj$. The allowed vertices are univalent vertices (degree +1), trivalent vertices joining 3 edges of the same color (degree -1), 4-valent
vertices joining edges of alternating distant colors (degree 0), and 6-valent vertices joining edges of alternating adjacent colors (degree 0).

We will occasionally use a shorthand to represent double dots (two univalent vertices connected by an edge). We identify a double dot colored $i$ with the polynomial $f_i\in R$, and to
a linear combination of disjoint unions of double dots in the same region of a graph, we associate the appropriate linear combination of products of $f_i$. For any polynomial $f\in R$,
a square box with a polynomial $f$ in a region will represent the corresponding linear combination of graphs with double dots. We have a bimodule action of $R$ on morphisms by placing
boxes (i.e. double dots) in the leftmost or rightmost regions of a graph. The functor $\FC$ respects this $R$-bimodule action.

\plabel{$f_i^2f_j$}{48}{12}
For instance, $\ig{1.5}{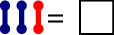}$.

In the following relations, blue represents a generic index.

\begin{equation} \label{assoc1} \ig{1}{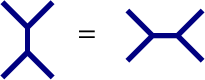} \end{equation}	
\begin{equation} \label{unit} \ig{1}{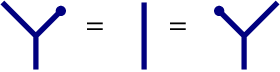} \end{equation}
\begin{equation} \label{needle} \ig{1}{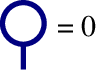} \end{equation}
\begin{equation} \label{dotslidesame} \ig{1}{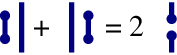} \end{equation}

We have the following implication of (\ref{dotslidesame}):

\begin{equation} \label{iidecomp} \ig{1}{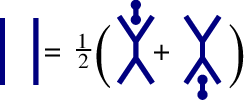}. \end{equation}

In the following relations, the two colors are distant.

\begin{equation} \label{R2} \ig{.8}{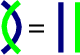} \end{equation}
\begin{equation} \label{distslidedot} \ig{1}{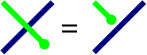} \end{equation}
\begin{equation} \label{distslide3} \ig{.8}{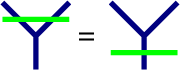} \end{equation}
\begin{equation} \label{dotslidefar} \ig{.8}{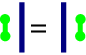} \end{equation}

In this relation, two colors are adjacent, and both distant to the third color.

\begin{equation} \label{distslide6} \ig{.8}{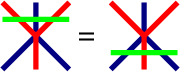} \end{equation}

In this relation, all three colors are mutually distant.

\begin{equation} \label{distslide4} \ig{1}{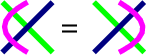} \end{equation}

\begin{remark} Relations (\ref{R2}) thru (\ref{distslide4}) indicate that any part of the graph colored $i$ and any part of the graph colored $j$ ``do not interact'' for $i$ and $j$
distant. That is, one may visualize sliding the $j$-colored part past the $i$-colored part, and it will not change the morphism. We call this the \emph{distant sliding property}.
\end{remark}

In the following relations, the two colors are adjacent.

\begin{equation} \label{dot6} \ig{.8}{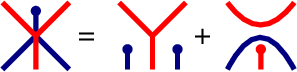} \end{equation}
\begin{equation} \label{ipidecomp} \ig{.8}{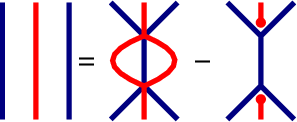} \end{equation}
\begin{equation} \label{assoc2} \ig{1}{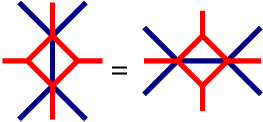} \end{equation}
\begin{equation} \label{dotslidenear} \ig{1}{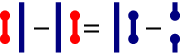} \end{equation}

We have the following implication of the above:

\begin{equation} \label{doubleasspartial} \ig{1.2}{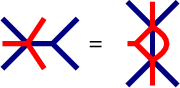} \end{equation}

In this final relation, the colors have the same adjacency as $\{1,2,3\}$.

\begin{equation} \label{assoc3} \ig{1}{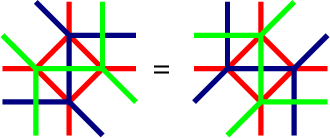} \end{equation}

\begin{thm}[Main Theorem of \cite{EKho}] There is a functor $\FC$ from $\DC$ to $\BSBim$, which is an equivalence of categories. Thus, the indecomposable objects in $\Kar(\DC)$ are parametrized by $w \in W$. \end{thm}

We will abusively call these indecomposable objects $B_w$, even though that is their image under the functor $\FC$, because denoting them $w$ would be even more confusing notation.

You can start paying attention again NOW.

We do not recall the definition of $\FC$ because it is largely irrelevant. Let $1 \ot 1 \ot \ldots \ot 1 \in B_{\ii}$ be called a \emph{1-tensor}. It will be significant that the
6-valent vertex and the 4-valent vertex both send 1-tensors to 1-tensors. So does the dot, positioned so that it represents a map from $B_i \to R$, and the trivalent vertex, positioned
so that it represents a map from $B_i \to B_i \ot B_i$. When the trivalent vertex is positioned so that it represents a map from $B_i \ot B_i \to B_i$, the corresponding map of
bimodules will simply apply the Demazure operator $\partial_i$ to the middle term in $R \ot_{R^i} R \ot_{R^i} R$.

The relation \eqref{ipidecomp} is sent under $\FC$ to the first direct sum decomposition below, while flipping the colors yields the second. Here, blue is $i$, red is $i+1$, and
$J=\{i,i+1\}$.

\begin{eqnarray} B_i \otimes B_{i+1} \otimes B_i = B_J \oplus B_i\\ B_{i+1} \otimes B_i \otimes B_{i+1} = B_J \oplus B_{i+1}. \end{eqnarray}

That is, the identity $1_{i(i+1)i}$ is decomposed into orthogonal idempotents. The first idempotent, which we call a \emph{doubled 6-valent vertex}, is the projection from $B_i \ot
B_{i+1} \ot B_i$ to its summand $B_J$. The 6-valent vertex itself is the projection from $B_i \ot B_{i+1} \ot B_i$ to $B_J$ and then the inclusion into $B_{i+1} \ot B_i \ot B_{i+1}$.

We call the following map, which is the projection from $i(i+1)i$ to the ``wrong" summand, by the name \emph{aborted 6-valent vertex}. 
\begin{equation} \label{whatisaborted} \ig{1}{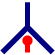} \end{equation} Because
projections to different summands are orthogonal, we have the following key equation, a simple consequence of \eqref{dot6}, \eqref{assoc1} and \eqref{needle}:
\begin{equation} \ig{1}{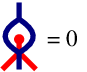} \label{dot6inside} \end{equation}
Several times in this paper we will use the fact that the morphism
\begin{equation} \ig{1}{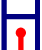} \label{itfactors} \end{equation}
factors through the aborted 6-valent vertex. This follows from \eqref{assoc1}.

%
\subsection{Thickening}
\label{subsec-thickening}

%

In this section, we discuss in detail some general facts about partial idempotent completions. We will eventually use the observations of this chapter to enhance the diagrammatics $\DC$ for $\BSBim$ above into a diagrammatic calculus $\fooDC$ for $\fooBim$.

\begin{defn} Let $\CC$ be a full subcategory of an ambient module category. By assuming this, we guarantee that the idempotent completion is fairly nice (Krull-Schmidt); in particular,
it is also embedded in the same ambient category. If $\mc{S}$ is a set of objects in the idempotent completion of $\CC$, we let $\CC(\mc{S})$ be the full subcategory of the ambient
module category whose objects are those of $\CC$ as well as $\mc{S}$. We call this a \emph{partial idempotent completion} or a \emph{thickening} of $\CC$. When $\mc{S}$ consists of a
single object $M$, we denote the thickening by $\CC(M)$. \end{defn}

Let us assume that we have a description of all morphisms in $\CC$, by generators and relations, and that $M$ is a module in the idempotent completion. We may pick out $M$ by using a
particular idempotent $\phi_X$ inside some object $X \in \CC$ of which $M$ is a summand.

\begin{claim} \label{simplepartialidcomp} To obtain a presentation of $\CC(M)$ by generators and relations, we may take as generators the generators of $\CC$ along with two new maps $p_X
\colon X \to M$ and $i_X \colon M \to X$, and as relations the relations of $\CC$ along with $i_X p_X = \phi_X$ and $p_X i_X = \1_M$. \end{claim}

\begin{proof} This is a tautological fact about idempotent completions, which we spell out this once. By definition, $\Hom_{\CC(M)}(M,Y)=\Hom_{\CC}(X,Y) \phi_X$, that is, those morphisms
$X \to Y$ which are unchanged under precomposition by $\phi_X$. The map $i_X$ corresponds to $\phi_X \in \End(X)$ itself, and in the thickening, all maps from $M \to Y$ will
clearly be compositions of $i_X$ with some map $X \to Y$. Similar dual statements may be made about $p_X$ and maps to $M$. Therefore $i_X,p_X$ are the only new generators required. Any
relations among maps factoring through $M$ arise from relations for maps factoring through $X$ (by definition of morphisms in $\CC(M)$) and these can all be deduced using the relations
stated above. \end{proof}

While this is sufficient to describe $\CC(M)$, the result is not necessarily an intuitive description. The object $M$ may have a variety of interesting maps to various objects in $\CC$,
whose properties could be deduced solely from the properties of $\phi_X$, but which are not obvious \emph{a priori}. For instance, it is not even obvious which other objects $Y$ might
have $M$ as a summand.

One thing we can do is augment our presentation (or diagrammatics) by adding new symbols for certain maps which can be constructed out of maps in $\CC$ and the new maps $p_X, i_X$. That
is, we add a new generator paired with a new relation which defines the new generator in terms of the existent morphisms. We can then deduce some relations which hold among these new
symbols, by checking them in $\CC(M)$. In doing so, one need not worry about the eternal questions one faces when given a presentation: do we have enough relations? Do we have too many?
We already have a presentation, and we are merely adding new symbols and relations as a more intuitive shorthand.

One useful such augmentation will be to produce inclusions and projections for other objects $Y$ of which $M$ is a direct summand, using only the maps $i_X,p_X$ and maps in $\CC$.

\begin{claim} Given a summand $M$ of $X$ defined by idempotent $\phi_X$, $M$ will also be a summand of $Y$ if and only if there exist maps $\phi_{X,Y} \colon X \to Y$ and $\phi_{Y,X}
\colon Y \to X$ such that $\phi_{X,Y} \phi_X = \phi_{X,Y}$, $\phi_X \phi_{Y,X} = \phi_{Y,X}$, and $\phi_{Y,X} \phi_{X,Y} = \phi_X$. \end{claim}

\begin{proof} We let the inclusion map $i_Y \colon M \to Y$ be $i_Y \define \phi_{X,Y} i_X$, and similarly we define the projection $p_Y \define p_X \phi_{Y,X}$. Conversely, given
inclusion and projection maps to $Y$, we let $\phi_{X,Y} = i_Y p_X$ and $\phi_{Y,X} = i_X p_Y$. The reader may verify that this works. \end{proof}

Thus, if we know what the transition maps $\phi_{X,Y}$ are explicitly, then we may augment our presentation by adding new maps $i_Y$ and $p_Y$, and relations as in the claim above and
its proof. Having such an augmentation will make it more obvious that $M$ is a summand of both $Y$ and $X$. The point is that the data of the transition maps is data which is entirely defined in terms of morphisms in $\CC$, not in any completion thereof.

We now generalize this to the form we will use, and leave proofs to the reader. In the claim below, one should think of $\phi_{X,X}$ as the idempotent $\phi_X$.

\begin{claim} \label{partialidempotentcompletion} Suppose that we have a nonempty collection $\{X_\alpha\}$ of objects in $\CC$ for which $M$ is a summand (that is, with each object we
fix an inclusion and projection map defining $M$ as a summand). Let $\phi_{\alpha,\beta}$ be the map $X_\alpha \to X_\beta$ given by the composition $X_\alpha \to M \to X_\beta$ of a
projection map with an inclusion map. Note that this implies $\phi_{\beta,\gamma} \phi_{\alpha,\beta} = \phi_{\alpha,\gamma}$. The maps $\phi_{\alpha,\beta}$ are morphisms in $\CC$, so
let us suppose that we know how to describe these maps explicitly. We may obtain a presentation of $\CC(M)$ as follows. The generators will consist of those generators of $\CC$ as well
as new maps $p_\alpha \colon X_\alpha \to M$ and $i_\alpha \colon M \to X_\alpha$. The relations will consist of those relations in $\CC$ as well as the new relations $i_\beta p_\alpha =
\phi_{\alpha,\beta}$ and $p_\alpha i_\alpha = \1_M$. \end{claim}

\begin{claim} A collection of functions $\phi_{\alpha,\beta} \colon X_\alpha \to X_\beta$ will define a mutual summand if and only if $\phi_{\beta,\gamma} \phi_{\alpha,\beta} =
\phi_{\alpha,\gamma}$ for all $\alpha,\beta,\gamma$. \end{claim}

\begin{defn} \label{defn:consistentfamily} We call a collection of morphisms $\phi_{\alpha,\beta}$ satisfying $\phi_{\beta,\gamma} \phi_{\alpha,\beta} = \phi_{\alpha,\gamma}$ a \emph{consistent family of projectors}.
\end{defn}

\begin{remark} \label{rmk:consistentisoms} A similar concept is a \emph{consistent family of isomorphisms}, which occurs when all $\phi_{\alpha,\beta}$ are isomorphisms. This is the data required to state that a
family of objects are canonically isomorphic. In fact, the idempotent $\phi_{\alpha, \alpha}$ will produce an object of the Karoubi envelope for each $\alpha$, and the maps
$\phi_{\alpha, \beta}$ will descend to a consistent family of isomorphisms between these various objects. \end{remark}

This concludes the discussion of adding numerous projections and inclusions to the presentation of $\CC(M)$. Of course, we may want to augment our presentation still further, to help
describe other features of the new objects.

\begin{remark} Given a consistent family of projectors, a map in $\CC(M)$ from $M$ to any object $Z$ can be specified by one of the following equivalent pieces of data \begin{itemize}
\item for a single $X$ of which $M$ is a summand, a map $f$ from $X \to Z$ such that $f \phi_{X,X} = f$, \item a ``consistent" family of maps $f_{\alpha} \colon X_{\alpha} \to Z$ such
that $f_{\alpha} \phi_{\beta,\alpha} = f_{\beta}$. \end{itemize} We will usually define a map by using one member of the consistent family, but once the map is defined, we will freely
use any member of the family (and will often investigate what certain other members of the family look like).

Note that we can also let $g$ be an arbitrary map $X \to Z$, without assuming that $g \phi_{X,X} = g$, and let $f = g \phi_{X,X}$. Then $f$ will yield a map from $M$ to $Z$. In this description, multiple different maps $g$ can give rise to the same map $f$, so the description is not unique.
\label{familyofmaps} \end{remark}

The goal of the next chapter is to present $\fooBim$ by generators and relations. We know that, for any parabolic subgroup $J$, $B_J$ will be a summand of $B_{\ii}$ where $\ii$ is any
reduced expression for $w_J$. For instance, $B_{s_1s_2s_1s_3s_2s_1}$ is a summand of $B_{s_1} \ot B_{s_2} \ot B_{s_1} \ot B_{s_3} \ot B_{s_2} \ot B_{s_1}$. The set of reduced
expressions for $w_J$ gives a family of objects in $\BSBim$ for which we wish to have inclusion and projection maps, which means that we will need to find an explicit description of a
consistent family of projectors $\phi_{\ii,\jj} \colon B_{\ii} \to B_{\jj}$ for any two reduced expressions of $w_J$. 

\begin{prop} \label{prop:isitBJ} Suppose that one has a family of maps $\phi_J = \{ \phi_{\ii,\jj} \}$ for each pair of reduced expressions of $w_J$, which satisfy the following three
properties. \begin{itemize} \item The family $\phi_J$ is a consistent family of projectors, picking out a summand $X$. \item The summand $X$ satisfies $X \ot i \cong X(1) \oplus X(-1)$
for each $i \in J$. \item The space $\HOM_{\DC}(X,\emptyset)$ is a cyclic $R$-module, generated in degree $+d_J$. \end{itemize} Then $X$ is indecomposable, and is sent by the functor
$\FC$ to the Soergel bimodule $B_J$. \end{prop}

\begin{proof} This follows immediately from Lemma \ref{isitBJ}. \end{proof}

Finding these maps $\phi_J$ and proving these three properties is the job of the next chapter.

\section{Expression graphs and the idempotent}
\label{sec-calcs}

%
\subsection{Expression graphs and path morphisms}
\label{subsec-expressions}
%
%

Most of the terminology of this section is ad hoc, but there does not seem to be any standard terminology in the literature.

\begin{defn} Let $w$ be an element of $S_n$. Let $\tilde{\Gamma_w}$, the \emph{expanded expression graph}, be the set of reduced expressions for $w$. We give $\tilde{\Gamma_w}$
the structure of an undirected graph by placing an edge between $x$ and $y$ if and only if they are related by a single application of either $s_is_js_i=s_js_is_j$ for $i,j$ adjacent,
or $s_is_j=s_js_i$ for $i,j$ distant. We may distinguish between these two different kinds of edges, calling the former \emph{adjacent edges} and the latter \emph{distant edges}.
Distant edges will be drawn as dashed lines in the examples below.

We may shorten reduced expressions to a sequence of indices, such as $121$ for $s_1 s_2 s_1$.

We may place an orientation on the adjacent edges of the expanded expression graph, using the \emph{lexicographic partial order}, so that arrows always go from $i(i+1)i$ to
$(i+1)i(i+1)$. The distant edges remain unoriented. When we speak of an \emph{oriented path} in $\tilde{\Gamma}$, we refer to a path which may follow distant edges freely, but can only
follow adjacent edges along the orientation. A \emph{reverse oriented path} is an oriented path backwards. When we say \emph{path} with no specification, we refer to any path, which may
follow the adjacent edges in either direction. \end{defn}

\begin{example} The expanded expression graph for $21232$ in $S_4$. \igc{1}{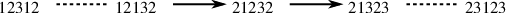} \end{example}

\begin{defn} Let $\Gamma_w$, the \emph{(conflated) expression graph}, be the graph which is the quotient of $\tilde{\Gamma}$ by all distant edges; that is, one identifies any two
vertices connected by a distant edge, and removes the distant edges. \end{defn}

The following examples will include some important definitions.

\begin{example} The expanded expression graph for $135$ in $S_6$. \igc{1}{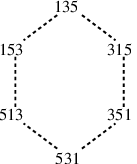} There are two distinct paths from $135$ to $531$, which form a hexagon. We refer to any cycle of this
form appearing in any expanded expression graph (which will occur whenever $\ldots ijk \ldots$ occurs inside a larger word, for $i,j,k$ all mutually distant) as a \emph{distant hexagon}. Note that the conflated expression graph is a point.
\end{example}

\begin{example} The expanded expression graph for $1214$ in $S_5$. \igc{1}{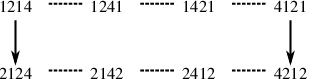} There are two distinct (oriented) paths from $1214$ to $4212$, which form an octagon. As above, we
refer to any cycle of this form in any expanded expression graph as a \emph{distant octagon}. Note that the conflated expression graph is a single edge. \end{example}

\begin{example} The expanded expression graph for the longest element $121321$ in $S_4$. \igc{1}{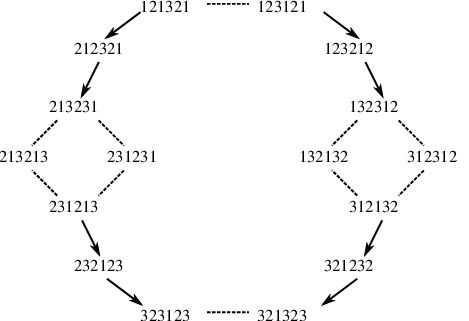} There are different kinds of cycles appearing here. For instance, a square
is formed between $213231$ and $231213$, because there are two \emph{disjoint} distant moves which can be applied, and one can apply them in either order. Any square of this kind in any
graph we call a \emph{disjoint square}. A disjoint square can involve distant or adjacent edges, with parallel edges having the same type. For example, there is a disjoint square or adjacent edges from $121343$ to $212434$.

Ignoring the ambiguities created by the two disjoint squares in this picture, or considering the conflated expression graph, there are two oriented paths from $121321$ to $323123$. We
refer to any cycle of this form in any expanded expression graph as a \emph{Zamolodchikov cycle}. \end{example}

Expression graphs were studied deeply in Manin-Schechtman \cite{ManSch}. In that paper they define an $n$-category, where one can at various stages view expressions as objects, distant
and adjacent edges as morphisms between objects, the cycles above as 2-morphisms between morphisms, and on to even higher structure. We will use only a few of their results here.

\begin{remark} To the casual reader, it may not be obvious how Manin-Schechtman discusses expression graphs at all in \cite{ManSch}, or how the results there translate into the results
here. We recommend closely reading \S 2 of that paper just for the example of $I=\{1,\ldots,n\}$ and $k \le 4$, and skimming \S 3, always keeping the examples of $S_3$ (see Chapter 3
Example 6) and $S_4$ in mind. The key results will be Theorem 3 and Lemma/Corollary 8 from \S 2, which we will translate below. \end{remark}

\begin{prop} \label{cyclesingraph} For any $w \in S_n$, the graph $\tilde{\Gamma_w}$ is connected. The 4 types of cycles mentioned above (disjoint squares, distant hexagons, distant
octagons, and Zamolodchikov cycles) generate $H^1(\tilde{\Gamma_w})$ topologically. Viewing these cycles as transformations on oriented paths (switching locally from one oriented path in a
cycle to the other), any two oriented paths with the same source and target can be reached from each other by these cycles.

The same statements are true for $\Gamma_w$, restricting oneself only to disjoint squares and Zamolodchikov cycles (for the remaining cycles are trivial in $\Gamma_w$). \end{prop}

\begin{proof} See \cite[\S 2, Corollary 8]{ManSch}. In fact, the disjoint squares and Zamolodchikov cycles in $\Gamma_w$ essentially give rise to the two different kinds of edges in the
next higher Bruhat order. \end{proof}

Now we see how to combine the theory of expression graphs with morphisms in $\DC$.

\begin{defn} To a vertex $\xb \in \tilde{\Gamma_w}$ we may associate an object $B_\xb \in \DC$, which is simply the sequence of indices corresponding to $\xb$. (Although $B_{\xb}$ looks
more like an object in $\BSBim$, we felt this made it easier to distinguish between the vertex in $\tilde{\Gamma}$ and the corresponding object in $\DC$.) To a path $\xb \fancyto \yb$
in $\tilde{\Gamma_w}$ we may associate a morphism $B_\xb \to B_\yb$ in $\DC$, by assigning the 4-valent vertex to a distant edge, and the 6-valent vertex to an adjacent edge. We call
this the \emph{path morphism} associated to the path. To a length $0$ path at vertex $B_\xb$ we associate the identity morphism of $B_\xb$. Note that reversing the direction of a path
will corresponding to placing the path morphism upside-down. \end{defn}

We will consistently use bold letters for vertices in $\tilde{\Gamma_w}$ or $\Gamma_w$. The notation $B_{\xb}$ is just like the notation $B_{\ii}$, and in fact $\xb$ is a special kind
of sequence $\ii$.

\begin{prop} \label{prop:BSconsistent} Let $f$ and $g$ represent two oriented (resp. reverse oriented) paths from $\xb$ to $\yb$. Then their path morphisms are equal. In the notation of the introduction, this orientation is BS-consistent. \end{prop}

\begin{proof} Since oriented paths $f$ and $g$ are related by the various cycle transformations above (in Proposition \ref{cyclesingraph}) we need only show that the path morphism is
unchanged under these cycle transformations. Because $\DC$ is a monoidal category, the morphisms associated to either path in a disjoint square are equal. Relation \eqref{distslide4} is
a restatement of the fact that the morphisms associated to either path in the distant hexagon are equal; similarly for \eqref{distslide6} and the distant octagon. Relation
\eqref{assoc3} is a restatement of the fact that the morphisms associated to either \emph{oriented} path in the Zamolodchikov cycle are the same. Flipping the pictures upside-down
yields the statement for reverse oriented paths. There is one additional kind of loop which is possible for oriented paths: following a distant edge from $\xb$ to $\yb$ and then right
back to $\xb$. But this is the identity by \eqref{R2}. \end{proof}

\begin{remark} If $\xb$ and $\yb$ are connected by a distant edge, then the loop $B_\xb \to B_\yb \to B_\xb$ is the identity of $B_\xb$, while if the edge is an adjacent edge then the
loop is \emph{not} the identity, but is instead the doubled 6-valent vertex, an idempotent of $B_\xb$. In this sense adjacent edges and distant edges are unsurprisingly different.
\end{remark}

\begin{remark} \label{orientationmatters} In the Zamolodchikov cycle for the longest element of $S_4$, some of the non-oriented paths are non-equal (beyond the trivial example of the
previous remark).

For instance, consider the two possible (necessarily unoriented) path morphisms from $212321$ to $321232$ which don't recross an oriented edge. These are pictured below as morphisms
fitting in a circle with boundary $212321232123$, and where blue, red, and green represent $1$, $2$, $3$ respectively.
\begin{equation} \label{nonorientednonequal} \ig{.8}{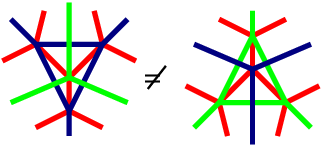} \end{equation} To show that they are non-equal, we
attach an aborted 6-valent vertex to the top of both sides: the right is clearly zero by \eqref{dot6inside}, while the left is nonzero. \igc{.8}{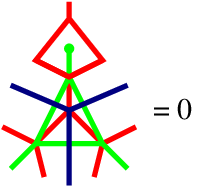}
\igc{.8}{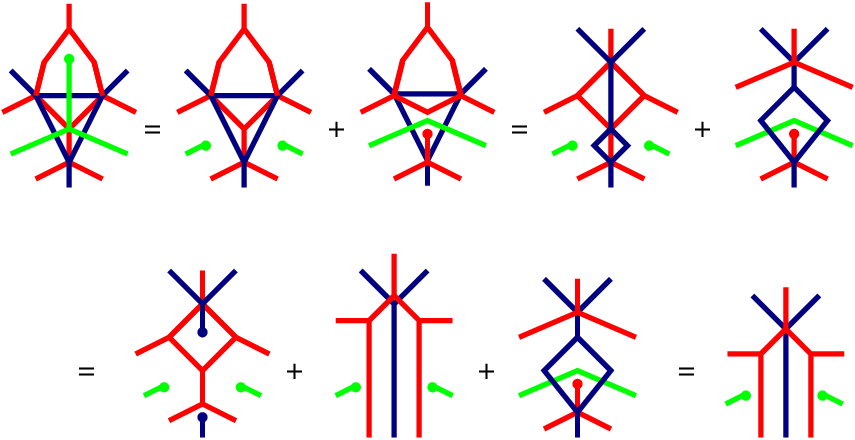} In this calculation, the first equality comes from \eqref{dot6}, the second from \eqref{assoc2} and \eqref{doubleasspartial}, the third from
\eqref{ipidecomp}, and the fourth from \eqref{dot6inside}. \end{remark}

\begin{cor} Consider a set of vertices $\{\xb_\alpha\}$ which are connected using only distant edges, and let $\phi_{\alpha,\beta}$ be the (unique) path morphism (using only distant
edges) from $B_{\xb_\alpha}$ to $B_{\xb_\beta}$. Then $\phi$ forms a consistent family of isomorphisms. \end{cor}

\begin{proof} Because the paths use only distant edges, they are oriented, and their compositions are also oriented. The previous proposition therefore immediately implies that
$\phi_{\beta,\gamma} \phi_{\alpha,\beta} = \phi_{\alpha,\gamma}$. They are isomorphisms since $\phi_{\alpha,\alpha}$ is the identity map. \end{proof}

The 4-valent vertices/distant edges are all isomorphisms and do not play a significant role. Henceforth, whenever we speak of the \emph{length} of a path in $\tilde{\Gamma_w}$ we mean
the number of adjacent edges in the path (i.e. the length in $\Gamma_w$).

\begin{defn} We abuse notation henceforth in order to talk about path morphisms of $\Gamma$ instead of $\tilde{\Gamma}$. To a vertex $\xb \in \Gamma$ we may associate a class of
objects, one for each vertex in $\tilde{\Gamma}$ lifting $\xb$, with a fixed family of transition isomorphisms between them. To a path $\xb \fancyto \yb$ in $\Gamma$, we may associate
its \emph{path morphism}, by which refer to the family of morphisms from some lift of $\xb$ to some lift of $\yb$ given by some lift of the path, or equivalently, to any specific
morphism in this family. See Remark \ref{familyofmaps}. \label{pathsinGamma} \end{defn}

%
\subsection{Statement, outline, and preliminaries}
\label{subsec-proofoutline}
%
%

We now restrict our attention to the graph $\Gamma = \Gamma_{w_0}$ for the longest element of $S_{n+1}$.

\begin{notation} Suppose that $\xb \to \yb$ is an edge in $\Gamma$. We write $\xb \downoneto \yb$ to denote the unique length 1 oriented path from $\xb$ to $\yb$, which is just the
edge. We write $\yb \uponeto \xb$ for the reverse-oriented length 1 path from $\yb$ to $\xb$.

Meanwhile, for $\xb,\yb \in \Gamma$, we write $\xb \downto \yb$ for some oriented path from $\xb$ to $\yb$ of unspecified length, assuming that one exists. We write $\yb \upto \xb$ for
some reverse-oriented path. We write $\psi_{\xb \downto \yb}$ and $\psi_{\yb \upto \xb}$ for the corresponding path morphisms, which do not depend on the choice of oriented path by
Proposition \ref{prop:BSconsistent}.

For any path $V$, we write $\overline{V}$ for the reversed path. For any diagram $D$ in $\DC$, we write $\overline{D}$ for the same diagram flipped upside-down. If $D$ is the path
morphism of $V$, then $\overline{D}$ is the path morphism of $\overline{V}$. \end{notation}

\begin{prop} There is a unique source $\sb$ in $\Gamma$, and a unique sink $\tb$. Let $m$ be the length of the shortest (not necessarily oriented) path from $\sb$ to $\tb$. Then every
vertex lies on some \emph{oriented} path $\sb \downto \tb$ of length $m$, and every oriented path $\xb \downto \yb$ can be extended to a length $m$ path $\sb \downto \xb \downto \yb
\downto \tb$. \end{prop}

\begin{proof} See \cite[\S 2, Theorem 3]{ManSch}. In fact, for the longest element of $S_{n+1}$, one has $m = {n+1 \choose 3}$. \end{proof}

It is easy to see what $\sb$ and $\tb$ are: we give a representative vertex in $\tilde{\Gamma}$ for the vertex in $\Gamma$. \begin{itemize}
	\item $n=1$: $\sb=1$, $\tb=1$.
	\item $n=2$: $\sb=1\ul{21}$, $\tb=2\ul{12}$.
	\item $n=3$: $\sb=121\ul{321}$, $\tb = 323\ul{123}$.
	\item $n=4$: $\sb=121321\ul{4321}$, $\tb=434234\ul{1234}$.
	\item $n=5$: $\sb=1213214321\ul{54321}$, $\tb=5453452345\ul{12345}$. \end{itemize}
At each stage, for $\sb$, we add a new sequence $n,n-1,\ldots,1$. The reader should be able to see the pattern.

\begin{defn} For $\xb, \yb \in \Gamma$, let $\phi_{\xb,\yb} = \psi_{\sb \downto \yb} \circ \psi_{\tb \upto \sb} \circ \psi_{\xb \downto \tb}$. In other words, $\phi_{\xb,\yb}$ is
corresponds to any path which goes from $\xb$ down to the bottom, up to the top, and then down to $\yb$. Let $\chi_{\xb,\yb} = \psi_{\tb \upto \yb} \circ \psi_{\sb \downto \tb} \circ
\psi_{\xb \upto \sb}$ correspond to the any path which goes from $\xb$ up to the top, down to the bottom, and then up to $\yb$. \end{defn}

Note that we have described a family of morphisms for vertices in $\Gamma$ instead of $\tilde{\Gamma}$, but the two concepts are functionally equivalent via Definition
\ref{pathsinGamma}.

Our main theorem, constructing the idempotent for the longest element, is the following.

\begin{thm} \label{MainThm} One has $\phi_{\xb, \yb} = \chi_{\xb, \yb}$ for all $\xb, \yb \in \Gamma$. The family $\phi$ is a consistent family of projectors. Its image is the
indecomposable object $B_{w_0}$ corresponding to the longest element.

If instead we had worked with $\Gamma_{w_J}$ for the longest element of $W_J$, where $J$ is a connected parabolic subset (so that $W_J \cong S_m$ for some $m$), we would obtain the
indecomposable object $B_J$ corresponding to $w_J$. \end{thm}

In the remainder of this section, we give an outline of the proof of Theorem \ref{MainThm}. We also discuss the part of the proof which just manipulates paths in $\Gamma$, without
needing to perform any calculations involving path morphisms.

\begin{defn} Let $Z = \psi_{\sb \downto \tb}$ denote the unique oriented path morphism from source to sink. Let $\overline{Z} = \psi_{\tb \upto \sb}$
denote the unique reverse-oriented path morphism from sink to source. \end{defn}

Note that $\phi_{\tb, \sb} = \overline{Z}$, $\chi_{\sb, \tb} = Z$, and $\phi_{\sb, \sb} = \chi_{\sb, \sb} = \overline{Z} \circ Z$. If $\chi$ is to be a consistent family of projectors,
then we must have \begin{equation} \label{ZZZ} Z \overline{Z} Z = Z. \end{equation} In fact, this equality suffices.

\begin{lemma} Suppose that \eqref{ZZZ} is true. Then $\phi_{\yb,\zb} \circ \phi_{\xb,\yb} = \phi_{\xb,\zb}$ and $\chi_{\yb,\zb} \circ \chi_{\xb,\yb} = \chi_{\xb,\zb}$ for all $\xb, \yb, \zb \in \Gamma$. \end{lemma}

It is easier to think about these equalities when one just describes the path, rather than composing the maps $\psi$. This is because one can read the path in order, rather than
backwards (one of the unfortunate features of function notation). So, we will confuse a path morphism with its path.

\begin{proof} The composition $\phi_{\yb, \zb} \circ \phi_{\xb, \yb}$ comes from the path \[ \xb \downto \tb \upto \sb \downto \yb \downto \tb \upto \sb \downto \zb.\] In the middle,
one has the composition $\tb \upto \sb \downto \tb \upto \sb$ which is just $\overline{Z} Z \overline{Z}$. Using \eqref{ZZZ} to replace this with $\overline{Z}$, we obtain
$\phi_{\xb,\zb}$. \end{proof}

How would one prove \eqref{ZZZ}? One method would be to prove the following equality between path morphisms for each edge $\xb \downoneto \yb$.
\begin{equation} \label{extradontmatter}  \sb \downto \tb \upto \yb \uponeto \xb \downoneto \yb = \sb \downto \tb \upto \yb. \end{equation}
In other words, if we go from source to sink, make our way back up to $\yb$, and then apply a little loop, that loop could be deleted.

\begin{lemma} \label{lem:ZZZ} Let $V$ be any particular oriented path between source and sink, and suppose that \eqref{extradontmatter} holds for each edge $\xb \downoneto \yb$ in $V$.
Then $Z \overline{Z} Z = Z$. \end{lemma}

\begin{proof} The path yielding $Z \overline{Z} Z$ can be thought of applying $Z$, going up $V$, then coming down $V$. Let $\sb = \vb_0 \downoneto \vb_1 \downoneto \cdots \downoneto
\vb_m = \tb$ be the path $V$. Letting $\xb = \sb$ and $\yb = \vb_1$, $\xb \downoneto \yb$ is the first edge in $V$, so we can use \eqref{extradontmatter} to replace $\sb \downto \tb
\upto \sb \downto \tb$ with $\sb \downto \tb \upto \vb_1 \downto \tb$. Now letting $\xb = \vb_1$ and $\yb = \vb_2$, eliminate the second edge of $V$, replacing $\sb \downto \tb \upto
\vb_1 \downto \tb$ with $\sb \downto \tb \upto \vb_2 \downto \tb$. Repeating this argument, we will eventually obtain $\sb \downto \tb$, which is $Z$. \end{proof}

Similarly, we can use \eqref{extradontmatter} to show that $\chi = \phi$.

\begin{lemma} \label{phiischi} Let $V$ be any particular oriented path between source and sink, and suppose that \eqref{extradontmatter} holds for each edge in $V$. Then $\phi_{\xb,
\yb} = \chi_{\xb, \yb}$ for any two vertices along the path $V$. \end{lemma}

\begin{proof} By the same argument used in the proof of Lemma \ref{lem:ZZZ}, we know that for any $\xb$ in $V$ one has an equality of path morphisms
\begin{equation} \label{aorna} \sb \downto \tb \upto \sb \downto \xb = \sb \downto \tb \upto \xb. \end{equation}
Thus
\begin{align*} \phi_{\xb, \yb} & = \xb \downto \tb \upto \sb \downto \yb = \xb \downto \tb \upto \sb \downto \tb \upto \sb \downto \yb = \\
	& = \xb \upto \sb \downto \tb \upto \sb \downto \yb = \xb \upto \sb \downto \tb \upto \yb. \end{align*}
First we replaced $\overline{Z}$ with $\overline{Z}Z\overline{Z}$, then applied the upside-down version of  \eqref{aorna} to $\xb$, and then applied \eqref{aorna} to $\yb$. \end{proof}

The bulk of the computations in this chapter go into proving \eqref{extradontmatter} for just one path $V$, a path that we choose very carefully for its nice combinatorial properties!
It would be nice to have a more path-independent proof of these results.

Let us quickly discuss how this works. The equality \eqref{extradontmatter} really wants to replace the doubled 6-valent vertex $\yb \uponeto \xb \downoneto \yb$ with the identity map
of $\yb$. But these are two different morphisms, and the difference between them is the other idempotent in \eqref{ipidecomp}, which factors through the aborted 6-valent vertex
\eqref{whatisaborted}. Thus \eqref{extradontmatter} is implied by
\begin{equation} \label{abortedwithZ} A_{\yb \uponeto \xb} \psi_{\tb \upto \yb} Z = 0. \end{equation}
where $A_{\yb \uponeto \xb}$ is the aborted 6-valent vertex, instead of the 6-valent vertex, an \emph{aborted edge}. We think of the composition $A_{\yb \uponeto \xb} \psi_{\tb \upto \yb}$ as being an aborted version of $\overline{V}$. So, said quickly, \eqref{abortedwithZ} is the statement that $Z$ is orthogonal to aborted versions of $\overline{Z}$.

This is what we prove, painstakingly, throughout the next several sections, for a special path $V$. Using a combinatorial description of $V$, we describe every possible abortion of $V$,
and check this orthogonality property. The path $V$ that we use is inductively defined and behaves reasonably well for the inclusion of the longest element of $S_n$ inside the longest
element of $S_{n+1}$. We isolate how the aborted 6-valent vertex behaves with respect to the ``inductive'' part of the part $V$, which will be the path $FR_1$ defined in Definition
\ref{defFRi}. 

So, having proven \eqref{extradontmatter}, we know that $\phi$ is a consistent family of projectors. We prove in \S\ref{subsec-projectors} that the image of $\phi_J$ (coming from the
graph $\Gamma_{w_J}$) is the indecomposable $B_J$ corresponding to $w_J$. The proof follows the lines of Proposition \ref{prop:isitBJ}. Let $X$ temporarily denote the image of $\phi$.
We construct a morphism $a_i$ in \S\ref{subsec-thicktrivalent}, one for each $i \in J$, and compute that it descends to a map from $X \ot i$ to $X$ of degree $-1$. Using $a_i$, we
construct the inclusions and projections in a decomposition $X \ot i \cong X(1) \oplus X(-1)$. Then, using orthogonality results like \eqref{abortedwithZ}, we prove that the Hom space
$\HOM(X,\emptyset)$ is at most rank $1$, generated by a specific morphism $\xi_J$ in degree $d_J$. We check that $\xi_J$ is nonzero by applying the functor to bimodules. Put together,
these facts and Proposition \ref{prop:isitBJ} imply that $X \cong B_J$.

We also prove a number of properties of the maps $a_i$, which will simplify the thick calculus. We introduce $a_i$ before proving that $\phi$ is a consistent family of projectors
because it will be useful in the inductive proof of \eqref{abortedwithZ}.

Once we know that the image of $\phi_J$ in the Grothendieck group is $b_J$, we can then prove that everything which was true for our specific path $V$ also holds for a general path.

\begin{lemma} The family $\phi_J$ is orthogonal to any aborted 6-valent vertex. \end{lemma}

\begin{proof} The following is a simple exercise in the Hecke algebra which requires some knowledge of the standard basis and the Deodhar formula. For an arbitrary sequence $\ii$ of length $k$, rewriting $b_{i_1} b_{i_2} \cdots b_{i_k}$ in the standard basis $\{H_w\}$, the coefficient of $H_w$ is zero when $\ell(w) > k$, and otherwise has minimal degree greater than or equal to $-(k - \ell(w))$. In particular, in the standard pairing, $(b_w, b_{i_1} b_{i_2} \cdots b_{i_k})$ has no degree $0$ component if $\ell(w) > k$ and $b_w$ is smooth (all the Kazhdan-Lusztig polynomials are $1$).

Now let $\ii$ be the target of an aborted 6-valent vertex. That is, $\ii$ has length $d_J - 2$, and is a reduced expression for $w_J$ with some $s_i s_j s_i$ replaced by
$s_i$, for $i, j \in J$ adjacent. Then there are no degree $0$ maps from $B_J$ to $\ii$, and the aborted 6-valent vertex must be orthogonal to the projection to $B_J$. \end{proof}

\begin{prop} \label{phiunchanged} The equality \eqref{extradontmatter} holds for all edges. In addition, for any edge from $\xb$ to $\yb$ in $\tilde{\Gamma}$, the composition of
$\phi_{\zb, \xb}$ with this edge is $\phi_{\zb, \yb}$. \end{prop}

\begin{proof} Using \eqref{ipidecomp}, the difference between $\yb \uponeto \xb \downoneto \yb$ and the identity of $\yb$ factors through an aborted 6-valent vertex, and is orthogonal
to $\phi_J$. This proves \eqref{extradontmatter}. Now consider an edge from $\xb$ to $\yb$ in $\tilde{\Gamma}$. If it is a distant edge or an oriented edge, adding this edge to the path
for $\phi_{\zb, \xb}$ yields precisely the path for $\phi_{\zb, \yb}$. If it is a reverse oriented edge, we use \eqref{extradontmatter} (or its flipped version) to deduce the desired
equality. \end{proof}

In particular, this implies that $\chi = \phi$ everywhere using Lemma \ref{phiischi}.

Now we begin the long and nasty proof of \eqref{abortedwithZ} with a study of various vertices and paths between them in $\tilde{\Gamma}$.

%
\subsection{The longest element}
\label{subsec-longest}
%

We fix notation for several useful vertices in $\tilde{\Gamma}$ lying over $\sb$ and $\tb$. We will use a subscript $J$ when describing the corresponding construction in $\tilde{\Gamma}_{w_J}$ for a connected parabolic subset $J$.

\begin{notation} (Example: $n=5$) We write $\sb^R$ for the lexicographically minimal expression, which was given at the beginning of \S\ref{subsec-proofoutline}, e.g. $1\ 21\ 321\ 4321\
54321$.  Note that $\sb^R=\sb^R_K\  54321$ for $K=\{1,2,3,4\}$. We write $\sb^L$ for the horizontal flip of the default sequence, also thought of as the ``left-facing" default sequence,
$\sb^L = 12345\ 1234\ 123\ 12\ 1 = 12345 \sb^L_K$. \end{notation}

\begin{remark} Here and elsewhere, we will be illustrating various proofs and definitions by using examples. This is a dangerous practice, but we solemnly promise that the general case
truly will be evident to the reader after a brief study of the examples (or so we hope!). We will be drawing path morphisms soon enough, and it is incredibly inconvenient to draw a
general case; while the general case could be made explicit symbolically, the advantage of graphical calculus is lost when one passes to long strings of symbols. The examples we give
will be more enlightening. \label{examplesareenough} \end{remark}

\begin{notation} Now we look for expressions interpolating between $\sb^R$ and $\sb^L$. Let $i \in I$, and let $M_i=\{1,2,\ldots,i\}$. Then we write $\sb^R_i$ for the
sequence that begins as $\sb^L$ until $i+1$, i.e. with $12345\ 1234\ \ldots\ 12\ldots (i+1)$ and then concludes with $\sb^R_{M_i}$. Note that $\sb^R_n = \sb^R$. Explicitly,
\begin{itemize} \item $\sb^R_5=\sb^R = 1\ 21\ 321\ 4321\ 54321$, \item $\sb^R_4 = 12345\ 1\ 21\ 321\ 4321$, \item $\sb^R_3 = 12345\ 1234\ 1\ 21\ 321$, \item $\sb^R_{2} = 12345\ 1234\
123\ 1\ 21 = 12345\ 1234\ 123\ 12\ 1 = \sb^R_{1} = \sb^L$. \end{itemize} Flipping all of these horizontally, we get $\sb^L_i = \sb^L_{I,i}$. \end{notation}

We have put spacing in our reduced expressions just to make it easier for the reader to parse, but the spaces have no mathematical significance.

To get from $\sb^R_5$ to $\sb^R_4$, simply find the first instance of each index $12345$ (which is the start of each segment separated by the spaces), and commute them to the left as
far as possible, until the sequence begins with $12345$. That this is possible can be proven inductively. Note that $\sb^R_4 = 12345 \sb^R_K = 12345 \sb^R_{K,4}$ for $K = \{1,2,3,4\}$.
The transformation from $\sb^R_4$ to $\sb^R_3$ does not affect the initial sequence $12345$, and is just the inductively defined operation from $\sb^R_{K,4}$ to $\sb^R_{K,3}$ on the
remainder. Thus one can travel from any $\sb^R_i$ to any $\sb^R_j$ using only distant edges in $\tilde{\Gamma}$, which justifies that they all lie over $\sb$.

The salient feature of $\sb^R_i$ is that, reading from the right, it begins with an expression for the longest element of $M_i = \{1,2,\ldots,i\}$. We will see that when $B_{\sb}$
interacts with other colors on the right, and those other colors are in $M_i$, then $\sb^R_i$ is the most useful representative of the isomorphism class, because the parts on the left
($123451234$ when $i=3$) will not play an important role. This will become clear in examples to come.

\begin{notation} Similarly, let $\tb^R = 5\ 45\ 345\ 2345\ 12345 = \tb^R_1$ be the lexicographically maximal expression, also given at the beginning of \S\ref{subsec-proofoutline}. Let
$\tb^R_2= 54321\ 5\ 45\ 345\ 2345$, and so forth, with $\tb^L_i$ being the horizontal flips. Here, we think of $M_i$ as being $\{i,i+1,\ldots,n\}$. \end{notation}

Now we introduce several paths in $\tilde{\Gamma}$. For the rest of the paper, we use the colors blue, red, green, purple, and black to represent the indices $1,2,3,4,5$ respectively.

\igc{1}{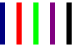}

\begin{defn} First, we introduce a useful path on subexpressions. For $i<j$ let $F=F_{i,j}$ denote the \emph{flip} path, defined by example here: \[ F_{1,4} = 12\ul{343}21 \downoneto
1243421 = 41\ul{232}14 \downoneto 4132314 = 43\ul{121}34 \downoneto 4321234,\] \[F_{3,5}=34543 \downoneto 35453 = 53435 \downoneto 54345.\] In general, $F_{i,j}$ starts at the sequence
going from $i$ up to $j$ and back down to $i$, and repeatedly applies an adjacent edge to the middle until the final expression is reached. We will usually just call the flip map $F$
when the indices are understood, and also use the same notation $F$ for the opposite path to this one. The flip path $F$ can also be viewed as the unique oriented path from source to
sink in $\Gamma_v$ for $v$ the appropriate permutation. Here is the path morphism for $F_{1,5}$. \igc{1.2}{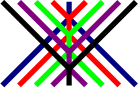} This diagram is exactly the path described above, up to ``monoidal
operations,'' that is, up to disjoint squares in $\tilde{\Gamma}$. \end{defn}

\begin{defn} Next, we inductively define a specific path $V = V^R$ from $\sb^R \downto \tb^R$. When $n=1$, $\sb=\tb$ and the path is trivial.
Suppose we have defined $V_J$ for all connected $J$.

(Example: $n=5$) Let us define $V$ by example. Let $K=\{1,2,3,4\}$. Then $s^R = s^R_K 54321$. To obtain $V$, first we apply $V_K$ to get $t^R_K 54321 = 4\ 34\ 234\ 1234\ 54321$,
and then we apply a sequence of flip maps of decreasing size $4\ 34\ 234\ \ul{123454321} \downto 4\ 34\ \ul{2345432}\ 12345 \downto 4\ \ul{34543}\ 2345\ 12345 \downto \ul{454}\ 345\
2345\ 12345 \downto 5\ 45\ 345\ 2345\ 12345 = t^R$.

(Example: $n=4$) Here is the picture of $V$: \begin{equation} \label{Vdefn} \ig{1}{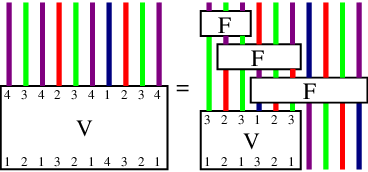} \end{equation}  \end{defn}

The path $V$ will be one of our preferred paths from $\sb$ to $\tb$, and thus a preferred realization of the morphism $Z$. Before proving facts about $Z$, we will need to describe
several other important paths.

\begin{defn} \label{defFRi} (Example: $n=4$) We introduce some paths in $\Gamma$ which use compositions of flip maps to bring certain indices to the far right. We think of them as paths in $\Gamma$ because we will be flexible as to which equivalence class of reduced expression we use for the source and sink.

Note that the expression $\tb^R = 4\ 34\ 234\ 1234$ has a single instance of the index $1$. To bring $1$ to the far right, we can apply the following sequence of (reverse) flip maps:
$\ul{4\ 34}\ 234\ 1234 \upto 3\ \ul{43\ 234}\ 1234 \upto 3\ 23\ \ul{432\ 1234} \upto 3\ 23\ 123\ 4321$. We denote this sequence of flips by $\overline{FR}_{\tb,1}$, and let us denote
the sequence in reverse by $FR_{\tb,1}$. In \eqref{Vdefn} above, the path $V$ ends in a sequence of flips, which is exactly $FR_{\tb,1}$.

The notation in $\overline{FR}_{\tb,1}$ is chosen because it takes index $1$ from $\tb$ and Flips it to the Far Right. We let $FR_{\tb,1}$ have no overline because
it is an oriented map, while $\overline{FR}_{\tb,1}$ is reverse-oriented. We will not introduce notation for the source of $FR_{\tb,1}$, which is some vertex in $\Gamma$.

Now let us define $\overline{FR}_{\tb,2}$, which brings the index $2$ to the far right with a minimal sequence of flips. We begin with $\tb^R_2 = 4321\ 434234$, and then apply the flip
maps $4321\ \ul{434}234 \upto 4321\ 3\ul{43234} \upto 4321\ 323432$. Similarly, we define $FR_{\tb,2}$ to be the reverse map, which is pictured below.

\igc{1}{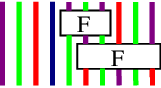}

To define $\overline{FR}_{\tb,i}$, which brings the index $i$ to the far right, we begin with $\tb^R_i$, which by definition ends in $\tb^R_{M_i}$ where $M_i = \{i, \ldots, n\}$. Then,
to the subsequence $\tb^R_{M_i}$, we apply the flip sequence $\overline{FR}_{\tb_{M_i}, i}$ which has been inductively defined. Thus our final two examples for $n=4$ are
$\overline{FR}_{\tb, 3} \co 4321432\ \ul{434} \upto 4321432\ 343$, and the identity map $\overline{FR}_{\tb, 4} \co 432143243\ 4 \to 432143243\ 4$. The sources of the map $\overline{FR}_{\tb,i}$ are various different expressions in $\tilde{\Gamma}$ which lift $\tb$, but since we view $\overline{FR}_{\tb,i}$ as a path in $\Gamma$, the all have the same source $\tb$.

In similar fashion, we can define $FR_{\sb,i}$, which brings the index $i$ to the far right, starting from $\sb$. Now it is $FR_{\sb,4}$ which is the most complicated, and $FR_{\sb,1}$
which is the identity map. For example, $FR_{\sb,4}$ begins with $\sb^R = 1213214321$ and applies the flips $\ul{121}3214321 \downto 2\ul{12321}4321 \downto 232\ul{1234321} \downto
2324321234$. Similarly, $FR_{\sb,3}$ begins with $\sb^R_3 = 1234\ 121321$ and applies the flips $1234\ \ul{121}321 \downto 1234\ 2\ul{12321} \downto 1234\ 232123$. As before, we define
$\overline{FR}_{\sb,i}$ to be the reverse path, and note that the version without the overline is an oriented path.

Finally, we can also define $FL_{\tb,i}$ and $FL_{\sb,i}$, which Flip the index $i$ to the Far Left. These are the horizontal reflections of the above. \end{defn}
	
\begin{claim} Fix any $i \in I$. Let $FR_{\sb,i} \co \sb \downto \xb$ and $FR_{\tb,i} \co \yb \downto \tb$ be the oriented paths defined above, so that $x$ and $y$ are both expressions with $i$ on the far right. Then there is an oriented path $x \downto y$ which does not alter the rightmost index $i$. \label{rewritezforiclaim}
\end{claim}

This claim allows one to factor the path morphism $Z \co \sb \downto \tb$ as $\sb \downto \xb \downto \yb \downto \tb$. This is another useful path denoted $V_i$, which is adapted to the index $i$. Note that $V = V_1$.

\begin{example} (Example: $n=4$) When $i=2$, we have \igc{1}{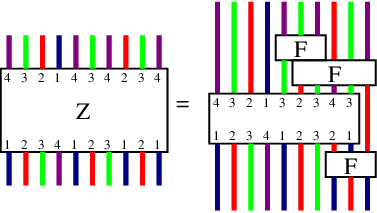} where the unlabeled box represents some unspecified oriented path. We have chosen the source and target of this morphism to be $\sb^R_2$ and $\tb^R_2$ respectively. The flip below the mystery box is $FR_{\sb, 2}$ and the flip sequence above the box is $FR_{\tb, 2}$. \end{example}

\begin{proof} Let $\hat{\xb}$ denote the expression $\xb$ without the final index $i$. Technically, this is an equivalence class of expressions for $w_0 s_i$, living in $\Gamma_{w_0 s_i}$. Define $\hat{\yb} \in \Gamma_{w_0 s_i}$ similarly. We need to show that there is an oriented path $\hat{\xb} \downto \hat{\yb}$ in $\Gamma_{w_0 s_i}$.

First, consider the special case where $i=1$. Then $\xb = \sb$, and $\hat{\xb}$ is the (unique) source of the expression graph for $w_0 s_1$. Thus there is an oriented path $\hat{\xb} \downto \hat{\yb}$. Similarly, when $i = n$, $\hat{\yb}$ is the unique sink for $w_0 s_n$.

Now, consider the example when $n=7$ and $i=3$. Then
\[ \xb = 1234567\ 123456\ 12345\ 1234\ \xb', \quad \quad \yb = 7654321\ 765432\ \yb'.\] 
Here, $\xb'$ is the target of the flip map $FR_3$ from $\sb_{\{1,2,3\}}$, and $\yb'$ is the source of the flip map $FR_3$ to $\tb_{\{3,4,5,6,7\}}$. We know by the special cases above that there are oriented maps $\xb' \downto \tb_{\{1,2,3\}}$ and $\sb_{\{3,4,5,6,7\}} \downto \yb'$ which do not involve the final index $3$. So, replacing $\xb'$ with $\tb_{\{1,2,3\}}$ and $\yb'$ with $\sb_{\{3,4,5,6,7\}}$, it is enough to show that there is an oriented path of the form
\begin{equation}\label{somethingorother} 1234567\ 123456\ 12345\ 1234\ (321\ 32\ 3) \quad \downto \quad  7654321\ 765432\ (34567\ 3456\ 345\ 34\ 3)\end{equation}
which does not involve the final index $3$. 

Let us apply a sequence of flip maps:
\begin{align*} 1234567\ 123456\ 12345\ \ul{1234\ 321}\ 32\ 3 \downto 1234567\ 123456\ \ul{12345\ 4321}\ 234\ 32\ 3 \downto \ldots \\ \downto 7654321\ 234567\ 23456\ 2345\ 234\ (32\ 3).\end{align*}
The result begins with $7654321$ as does our desired target. Removing $7654321$ from both source and target, the remainder is actually of exactly the form \eqref{somethingorother}, except for the parabolic subset $\{2,3,4,5,6,7\}$ instead of $\{1,2,3,4,5,6,7\}$. Now, induction finishes the proof.
\end{proof}

%
\subsection{Thick trivalent vertices}
\label{subsec-thicktrivalent}
%
%

\noindent When we have fixed a parabolic subset $J \subset I$ and wish to express an arbitrary index $i \in J$, we use either the color teal or brown.

\igc{1}{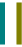}

Because $B_J \ot B_i \cong B_J(1) \oplus B_J(-1)$ for $i \in J$, there should be some projection/inclusion maps between $B_J \ot B_i$ and $B_J$ of degree $\pm 1$. When $J=\{i\}$ so that
$B_J = B_i$, we know that the projection map of degree $-1$ is drawn as a trivalent vertex (as is the inclusion map of degree $-1$, since they are related under biadjunction). Whatever
they are for $B_J$, we should think of the projection (resp. inclusion) map of degree $-1$ as a \emph{thick trivalent vertex}.

Once we prove that $\phi_J$ is a consistent family of projectors, Remark \ref{familyofmaps} states that we can specify a map $B_J \ot B_i \to B_J$ by giving a map from $B_{\xb} \ot B_i
\to B_{\xb}$, for some reduced expression $\xb$ for $w_J$, which is preserved under pre-composition with $\phi_{\xb, \xb} \ot 1_{i}$ and post-composition with $\phi_{\xb, \xb}$.
Alternatively, we can provide any map $B_{\xb} \ot B_i \to B_{\xb}$, and then pre- and post-compose it as above. The map we provide will be named $a^R_{\xb, i}$, where the $R$ indicates
that the new $i$-colored strand is on the Right. Pre-composing $a$ with $\phi_{\xb, \xb} \ot 1_{i}$ and post-composing with $\phi_{\xb, \xb}$, we obtain the map which descends to the
thick trivalent vertex.

For example, the map $a^R_{\tb^R_J,i}$ has the following form, when $J = \{1,2,3,4\}$. \igc{1}{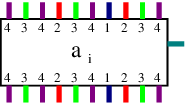} Teal represents the arbitrary index $i \in J$. We have rotated the picture so that $i$ appears as a side boundary instead of a bottom boundary.

Similarly, since $B_i \ot B_J$ also splits into two copies of $B_J$, there should be a thick trivalent vertex on the left as well. The horizontal flip of $a^R_{\tb^R_J,i}$ above would be $a^L_{\tb^L_J,i}$, which creates an $i$-labeled strand on the left, and uses the expression $\tb^L_J$. For these
maps, $\xb$ and $i$ can be determined from the colors on the boundary, and $R$ or $L$ determined by where the additional teal line sticks out, so we will often call it simply $a_i$ or
$a$. When the parabolic subset $J$ is not written, it is assumed to be the entire set $I$.

We have not yet proven that $\phi_J$ is a consistent family of projectors, but we will at least define $a^R_{\xb,i}$ for all reduced expressions $\xb$ lifting $\sb$ and $\tb$, and check
that it commutes appropriately with $\phi_{\sb, \sb}$ and $\phi_{\sb, \tb}$. In fact, the basic principle is this: when $\xb$ has $i$ on the right, the thick trivalent vertex comes from
an ordinary trivalent vertex (composed with $\phi$). When it does not, they are intertwined by the path morphism which brings $i$ to the right, namely $FR_i$. This is confirmed in
\eqref{triflippushiter} and \eqref{aflippush} below.

\begin{defn} (Example: $n=4$) Let the morphism $a^R_{\tb^R,i}$ be the morphism described in the following examples. It has top and bottom boundary $\tb^R$ and side
boundary $i$.
\begin{equation} \label{aRtRi} \ig{1}{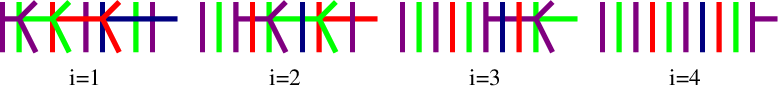} \end{equation}
Read this morphism from right to left. When $i=n$, the result is just a trivalent vertex. Otherwise, the $i$-strand comes from the right, crosses left over any strands with labels $\ge i+2$, and then meets the strand labeled $i+1$ in a 6-valent vertex. Then, the strand labeled $i+1$ crosses left over any strands with labels $\le i-1$. What remains on bottom and top is the reduced expression $\tb^R_{\{2,3, \ldots, n\}}$, with $i+1$ coming in from the right. Then, the process is repeated inductively for the $i+1$ strand.

Any expression $\xb$ for $w_0$ which lives above $\tb$ is connected to any other by a sequence of distant edges, and these distant edges form a compatible family of isomorphisms, as previously discussed. Thus we may define $a^R_{\xb,i}$ to be the morphism intertwined with $a^R_{\tb^R,i}$ by these isomorphisms. We let $a^R_{\tb,i}$ denote this family. To give an example, here is the morphism $a^R_{\tb^R_i,i}$, which has top and bottom boundary $\tb^R_i$ instead.
\begin{equation} \label{aRtRii} \ig{1}{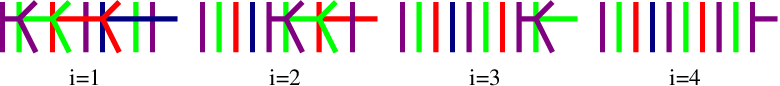} \end{equation}
This version of $a_i$ is possibly more intuitive, as it has fewer needless 4-valent vertices. When $i=1$, $\tb^R_1 = \tb^R$ and the diagram is as complicated as before. Otherwise, it decomposes as identity maps next to $a^R_{\tb^R_{J,i},i}$ for $J = M_i = \{i,\ldots,n\}$.

Similarly, one can define $a_i=a^R_{\sb^R,i}$ as follows.
\begin{equation} \label{aRsRi} \ig{1}{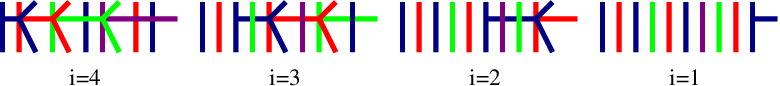} \end{equation} One can use this to define $a^R_{\xb,i}$ for any expression lifting $\sb$. \end{defn}

Now it is time to gather some results about the various maps $a_i,FR_i,Z$, in an effort to prove that $a_i$ interacts nicely with $\phi_{\sb, \tb}$ and $\phi_{\sb, \sb}$. The key
relation, which leads to all the others, is the following:

\begin{claim} (Example: $n=5$) We have the following equality. \begin{equation} \ig{1}{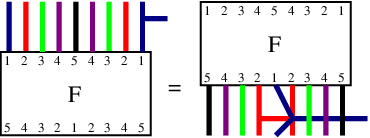}. \label{triflippush} \end{equation} \end{claim}

\begin{proof} Writing out the flip map in full, this is an immediate application of relation (\ref{doubleasspartial}). Explicitly, \igc{1}{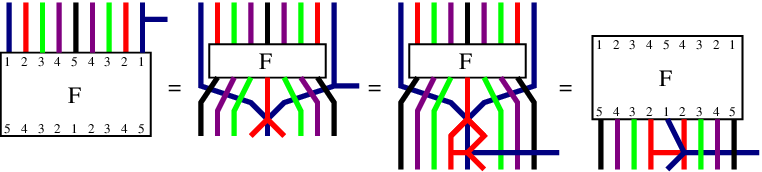} The box labelled $F$ in the
intermediate calculations is the flip map for $\{2,3,4,5\}$. \end{proof}

We now demonstrate some of this nice behavior for $a_1$.

\begin{claim} We have the following equalities, dealing with $a_1$. The first two are examples for $n=4$.

\begin{equation} \ig{1}{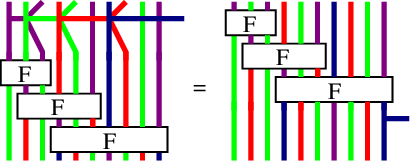}. \label{triflippushiter} \end{equation}
Symbolically, this is $a_1 = a^R_{\tb^R,i}$ put above $FR_{\tb,1}$.

\begin{equation} \ig{1}{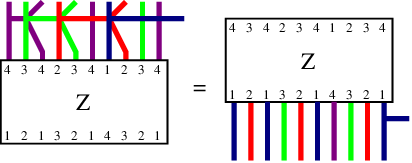}. \label{trizpush} \end{equation}
Once again, the diagram above $Z$ on the LHS is $a_1$. The trivalent vertex below $Z$ on the RHS also happens to be $a_1 = a^R_{\sb^R,1}$.

This last equation is not motivated at the moment, but will be used to deal with ``abortion terms'' in a future calculation. This is the $n=5$ example.
\begin{equation} \ig{1}{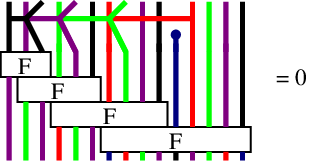}. \label{whatkills} \end{equation}
The upper left corner looks like \eqref{triflippushiter} but for $J = \{2,3,4,5\}$. The bottom is the flip sequence $FR_{\tb,1}$ for $\{1,2,3,4,5\}$.
\end{claim}

\begin{proof} The equality \eqref{triflippushiter} is easily seen to be an iterated use of \eqref{triflippush}, starting with the leftmost flip and moving to the right.
	
To get \eqref{trizpush}, remember that we can always express $Z$ as
\begin{equation} \ig{1}{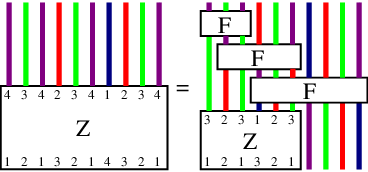} \label{rewritez} \end{equation}
using the path $V$, see \eqref{Vdefn}. Placing $Z_{\{1,2,3\}}$ below \eqref{triflippushiter} yields \eqref{trizpush}.

Applying \eqref{triflippushiter} to all but the last flip in \eqref{whatkills}, the diagram factors through the following picture as it enters the final flip: \igc{1}{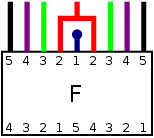}
This factoring was observed in \eqref{itfactors}. However, the first thing that happens in a flip is a 6-valent vertex, which kills the aborted 6-valent vertex by \eqref{dot6inside} \end{proof}

It may be useful for the reader to work through the examples of $n=2,3$ explicitly, writing out flip maps and $Z$ in full and checking these results. The general computations are no more
difficult.

The claim above could be thought of as a suite of results about $a_1$. Now we generalize these results to $a_i$ for any $i$.

\begin{claim} We have the following equalities, for any $i \in I$. Each expression has an arbitrary expression lifting $\tb$ as the target. The map $a_i$ represents either $a^R_{\tb,i}$ or $a^R_{\sb,i}$. The map $FR_i$ is $FR_{\tb,i}$. As usual, $\overline{Z}$ represents $Z$ upside-down.

 \begin{equation} \ig{1}{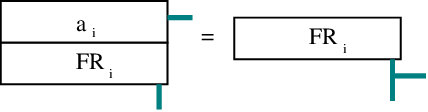}. \label{aflippush} \end{equation}

 \begin{equation} \ig{1}{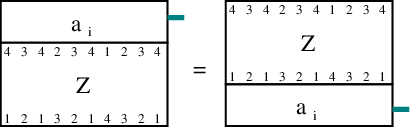}. \label{azpush} \end{equation}

 \begin{equation} \ig{1}{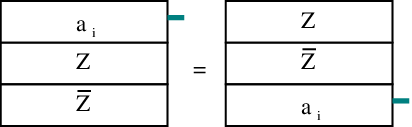}. \label{azzcommute} \end{equation}

 \begin{equation} \ig{1}{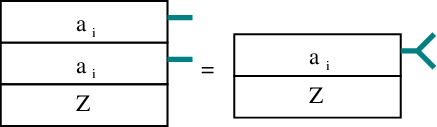}. \label{asquaredz} \end{equation}
\end{claim}

In particular, \eqref{azpush} says that our maps $a_{\tb,i}$ and $a_{\sb,i}$ are intertwined by $Z$, and \eqref{azzcommute} says that $a_{\tb,i}$ commutes with
$\phi_{\tb,\tb}$.

\begin{proof} To prove \eqref{aflippush}, let us choose $\tb^R_i$ to be our target, and use the description of $a_i$ given in \eqref{aRtRii}. If $i=1$ then \eqref{aflippush} is
\eqref{triflippushiter}. Otherwise, both $a_i$ and $FR_i$ are expressed as a tensor product of an identity morphism, and the corresponding maps for the parabolic subset $M_i = \{i,
\ldots, n\}$. This reduces the equality to \eqref{triflippushiter} for $M_i$.

To prove \eqref{azpush}, we simply use the path $V_i$ of Claim \ref{rewritezforiclaim} and apply \eqref{aflippush} twice. We provide an example, for $n=4$ and $i=2$. The target is chosen to be $\tb^R_i$, and the source to be $\sb^R_i$. The flip sequences above the mystery box is $FR_{\tb,i}$, and the flip sequence below it is $FR_{\sb,i}$. \igc{1}{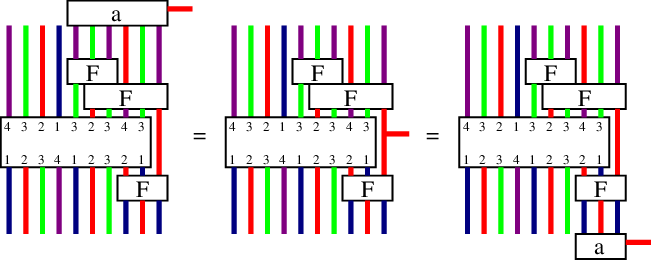}
Applying \eqref{azpush} twice, we immediately get \eqref{azzcommute}.

To obtain \eqref{asquaredz}, we again use the path $V_i$, and proceed as follows. \igc{1}{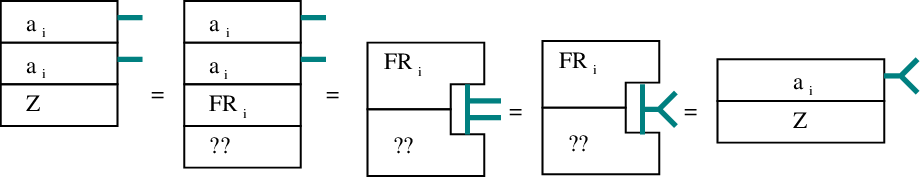}
\end{proof}

\begin{cor} (Example: $n=4$). We have the following equality. \igc{1}{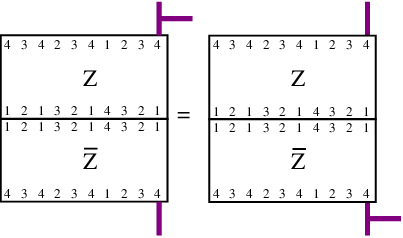} \label{tricommutes} \end{cor}

\begin{proof} This is a special case of \eqref{azzcommute}, when $i=n$. \end{proof}
	

We now show that $Z$ is orthogonal to certain aborted versions of $a_i$. This is the first step towards a proof of \eqref{abortedwithZ}, and will also be used to describe some additional properties of $a_i$ below.

\begin{prop} (Example: $n=5$) The maps of Figure \ref{whatkillsfigure} are all zero. \label{whatkillsgeneral} \end{prop}

\begin{figure}  \igc{.9}{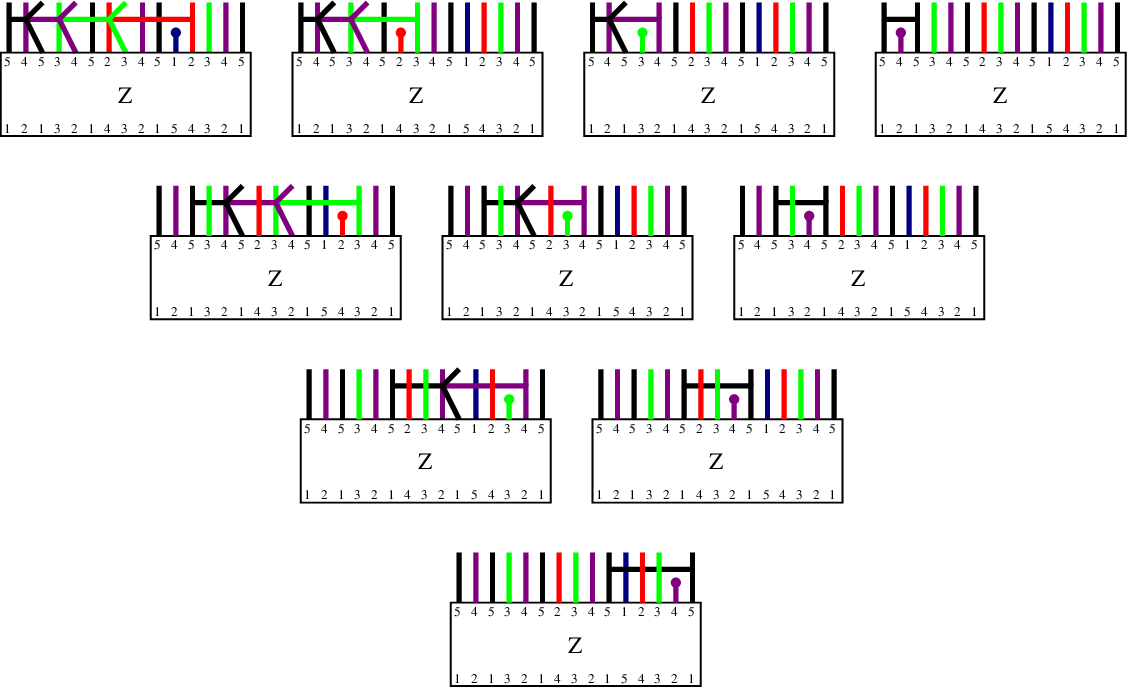} \caption{The morphisms in Proposition \ref{whatkillsgeneral}} \label{whatkillsfigure} \end{figure}

\begin{remark} To the reader familiar with light leaves from \cite{EWGr4sb} or \cite{LibLL}, this proposition (with the factorization discussed near \eqref{itfactors}) states that $Z$
is orthogonal to certain light leaves of degree zero. It is possible that many statements in this chapter can be streamlined by introducing light leaves, though it is
not clear that the proofs would be. \end{remark}

\begin{proof} This is a generalization of \eqref{whatkills}, and follows quickly from it. The first picture of the first row is precisely that of \eqref{whatkills}, while the second
picture is what \eqref{whatkills} would be for $K=\{2,3,4,5\}$ with additional lines $12345$ added on the right. Choosing a path $\sb \downto \tb$ which ends by applying $V_K$ on the
left, the map is zero by the $K$ case of \eqref{whatkills}. The third picture in the first row is zero by the $K=\{3,4,5\}$ case, and so forth.

For the second row, as for $a_2$, we slide $54321$ to the left and ignore it (that is, we were better off using $\tb^R_2$ instead of $\tb^R$). Then, letting $K=\{2,3,4,5\}$, we are left
with what would be the first row for (\ref{whatkillsgeneral}) for $K$. We leave a more explicit version to the reader. Similarly, for the third row, we slide $543215432$ to the left,
and what remains is the first row for $K=\{3,4,5\}$. The pattern is now clear. \end{proof}

We conclude this section with some further properties of $a_i$, which will simplify our thick calculus in \S\ref{sec-augmented}. The first one states that \eqref{asquaredz} holds true,
even without $Z$ below. The proof of \eqref{asquaredz} used the existence of $Z$ and the path $V_i$ to simplify the proof drastically, while the proof below is an annoying computation.
The same statement is true for the other statements here: they are easier to prove when $Z$ is placed below them, and we leave these easier proofs as exercises for the reader. In fact,
we will only ever care about the maps $a_i$ when they are being post- or pre-composed with $Z$ or $\overline{Z}$. Nonetheless, the properties below mostly hold in the absence of $Z$, so
they might as well be proven there.

\begin{prop} \label{aproperties} The following equations hold, as endomorphisms of (some lift of) $\sb_J$.

For any $i \in J$: \begin{equation} \ig{1}{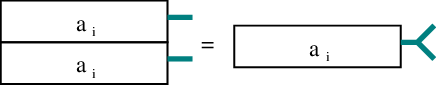} \label{asquared} \end{equation}

When teal is $i \in J$ and brown is $i+1 \in J$: \begin{equation} \ig{1}{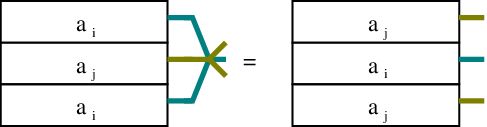} \label{a6} \end{equation}

When teal is $i \in J$ and brown is $i+1 \in J$ (this \emph{does} require $Z$ below): \begin{equation} \ig{1}{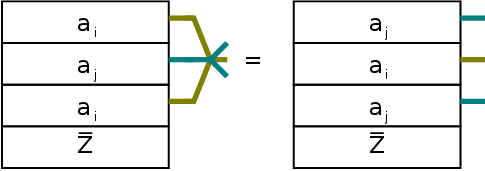} \label{a6z} \end{equation}

For any $i,j \in J$ distant: \begin{equation} \ig{1}{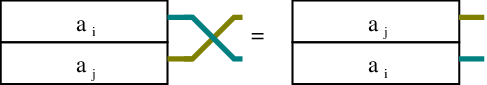} \label{a4} \end{equation}

For any $i,j \in J$, with no restrictions ($i=j$ is possible): \begin{equation} \ig{1}{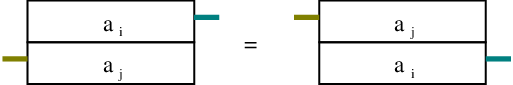} \label{aopp} \end{equation}

For any $i \in J$ (this \emph{does} require $Z$ below): \begin{equation} \ig{1}{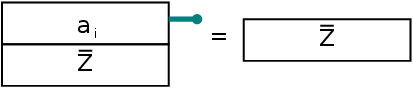} \label{adotz} \end{equation}

These equations (modifying $Z$ and $\overline{Z}$ appropriately) also hold as endomorphisms of (some lift of) $\tb_J$. One should swap \eqref{a6} and \eqref{a6z}.
\end{prop}

We prove these results for a lift of $\sb$. The $\tb$ case is analogous.

\begin{proof}[Proof of \eqref{asquared}] 
This is an application of one-color associativity \eqref{assoc1}, and then repeated uses of two-color associativity \eqref{doubleasspartial}, as
illustrated in the following example of 3 colors. The dotted rectangle surrounds that which is to be changed, and the solid rectangle what it becomes; the final step is just a distant
slide. \igc{1.1}{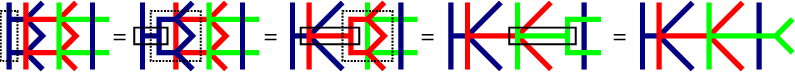}

Clearly, any calculation for $a_3$ reduces to this calculation, regardless of the number of total colors in $J$, because if we choose the expression $\sb^R_{J,3}$ as the top and bottom
of the diargam, then this calculation (tensored on the left with identity maps) will suffice. The calculation clearly generalizes to any number of colors. \end{proof}

\begin{proof}[Proof of \eqref{a6}] This is an application of one-color associativity \eqref{assoc1}, then two-color
associativity \eqref{assoc2}, then repeated uses of three-color associativity \eqref{assoc3}, as illustrated in the following example of 4 colors. \igc{1.1}{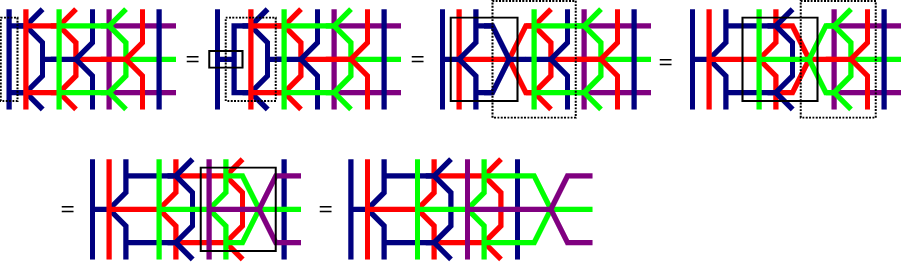}
Using the same tricks as the previous proof, we can bootstrap this calculation to deal with any case when teal is $i$ and brown is $i+1$. \end{proof}

In this proof, we repeatedly apply the equality \eqref{assoc3}. If teal and brown were swapped, however, then a similar proof would be forced to use the non-equality \eqref{nonorientednonequal} instead! Unfortunately, when teal and brown are swapped, \eqref{a6} is just false, and one only has the weaker result \eqref{a6z}.

Note, however, that by applying the Dynkin diagram automorphism to the colors above, one obtains the analogous result for $\tb$ and where teal and brown are swapped.

\begin{proof}[Proof of \eqref{a6z}] Use \eqref{azpush} to slide each $a_i$ past $\overline{Z}$, so that they now have top and bottom $\tb$. Then apply the $\tb$ version of \eqref{a6}, and use \eqref{azpush} to slide them back. \end{proof}

\begin{proof}[Proof of \eqref{a4}] Regardless of which distant colors are chosen, this is a simple application of the distant sliding rules. The reader will observe that any interesting
features of $a_i$ will be on strands which are uninteresting in $a_j$, and vice versa (some strands may be uninteresting for both). This should be clear from the following example.
\igc{1.1}{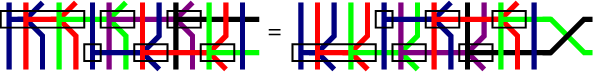} \end{proof}

\begin{proof}[Proof of \eqref{aopp}] First, it is important to note that the map $a^R_i$ going to the right and the map $a^L_j$ going to the left require \emph{different} choices of
representative of $s$, either $s^R$ or $s^L$. To make sense of the equation as, say, endomorphisms of $s^R$, then one must think of $a^L_j$ as first composing by the distant crossings
that bring $s^R$ to $s^L$, applying $a_j$ as normal, and then applying the distant crossings back. As such, this equation can be a pain to write out in general.

Let us do the hardest case: when $i=j=n$ inside $I=\{1,\ldots,n\}$. We will demonstrate the example when $n=5$. Instead of viewing it as an endomorphism of $s^R$ or $s^L$, we view it as
an endomorphism of $121321\ul{4354}23121$ (for the general case, the interesting part in the middle will read $(n-1)(n-2)(n)(n-1)$). Then the reader can observe that the equality may be
written suggestively as \igc{1.1}{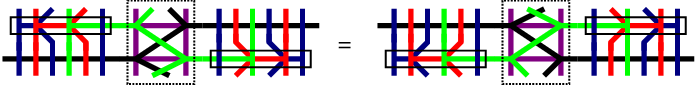} This equality follows immediately from three-color associativity (\ref{assoc3}) applied to the dashed box, and distant sliding rules applied
to the rest. It should be obvious how this example generalizes to any $n \ge 3$. The cases of $n=1,2$ are precisely one-color and two-color associativity.

We leave to the reader the proof that all other cases (for $i<n$ or $j<n$) follow from $i=j=m$ and $J=\{1,\ldots,m\}$ for some $m<n$. \end{proof}

\begin{proof}[Proof of \eqref{adotz}] To prove this, one writes out the map $a$ with a dot, and resolves the dot using \eqref{dot6} into a sum of two diagrams. One of the diagrams will
die when composed with $Z$, thanks to Proposition \ref{whatkillsgeneral}. The other diagram has a dot attached to another 6-valent vertex, so resolves again using (\ref{dot6}) into a
sum of two diagrams. Again, one of the two will die thanks to Proposition \ref{whatkillsgeneral}, and so forth. The single surviving diagram is the identity map. The proof of Lemma
\ref{killerwrapgeneral} has a similar argument which is made pictorial. \end{proof}

\begin{remark} There is a direct proof of \eqref{a6z} analogous to the proof of \eqref{adotz}. One adds a new teal-brown 6-valent vertex to both sides of \eqref{a6}. The doubled
6-valent vertex on one side can be replaced with the identity, modulo an error term has an aborted 6-valent vertex instead. Resolving this term as in the proof of \eqref{adotz}, one
gets zero.

Similarly, there is an indirect proof of \eqref{adotz} similar to the proof of \eqref{a6z} and \eqref{asquaredz}, which slides $a_i$ past flip maps until it becomes a trivalent vertex, and the dot just pulls in by \eqref{unit}. \end{remark}

%
\subsection{Analyzing the path $V$}
\label{subsec-nastiness}
%
%

The goal of this section is to prove \eqref{abortedwithZ} along our chosen path $V$, thereby proving \eqref{extradontmatter} and \eqref{ZZZ}. Let us restate the goal.

\begin{prop} \label{abortedV} Let $Q_\xb$ be the following morphism with bottom boundary $\sb$: one follows the path $V$, but at some vertex $\xb$ along the way, instead of following an
adjacent edge $\xb \downoneto \yb$, one does an aborted 6-valent vertex instead. Then $Q_\xb \overline{Z} = 0$. \end{prop}

\begin{proof} We prove this statement by induction on the number of colors. The base case of two colors follows from \eqref{dot6inside}. So consider the example when $n=5$,
$I=\{1,2,3,4,5\}$ and $K=\{1,2,3,4\}$. Recall that the path $V$ begins with $V_K$, and continues with the flip sequence $FR_{\tb,1}$. If the vertex $\xb$ is chosen so that it is part of
$\overline{V}_K$, then $Q_{\xb}$ is actually an aborted version of $V_K$, and $Q_\xb \overline{Z}$ factors through $Q_{\xb} \overline{Z}_K$. Therefore the product is zero by induction.
So we may assume that the vertex $\xb$ is within the sequence of flips $FR_{\tb,1}$. Therefore, $Q_{\xb} \overline{Z}$ looks like this. \igc{1}{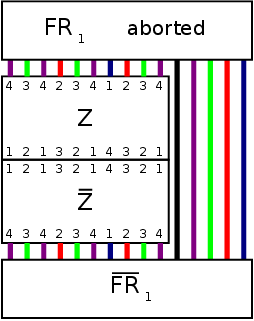}

Lemma \ref{FRiabort} below will show that the aborted version of $FR_{\tb,1}$ composed with $Z_K$ is a linear combination of certain diagrams, each of which has a dot on one of the
strands $54321$ which bypass $Z_K$. Actually, the dot is on one of $5432$, and cannot be on the final $1$ strand. Lemma \ref{FRidot} will show that when this dot is pulled downwards
into $\overline{FR}_{\tb,i}$, it resolves into a linear combination of certain diagrams (actually, the same class of diagrams, upside-down). Finally, Lemma \ref{killerwrapgeneral} will
imply that combining any of the diagrams from Lemma \ref{FRiabort} on top with any of the diagrams from Lemma \ref{FRidot} on bottom will yield zero. \end{proof}

It remains to describe this class of diagrams which appears, and prove the lemmas.

\begin{lemma} (Example: $n=5$) Let $c$ be one of the pictures in the table of figure \ref{killerwrapgeneralfig}, and let $b$ be another picture in the same row. For each row, the lines
on the right are $54321$ except with one index conspicuously absent. The following map is zero, where $z=z_K$ for $K=\{1,2,3,4\}$ and $\overline{b}$ is $b$ upside-down.
\igc{1}{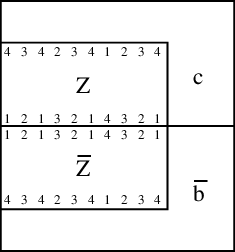} \label{killerwrapgeneral} \end{lemma}

\begin{figure} \igc{.9}{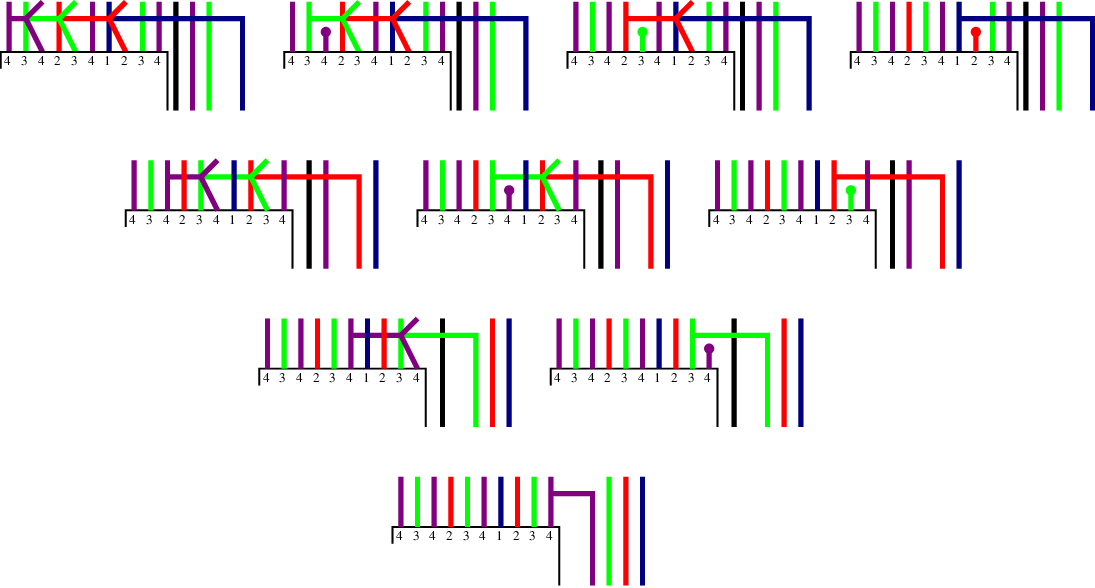} \caption{Morphisms appearing in Lemmas \ref{killerwrapgeneral} and \ref{FRidot}} \label{killerwrapgeneralfig} \end{figure}

\begin{proof} Note that the leftmost picture in each row is precisely $a_i$ for some $i$. If both $b$ and $c$ are chosen to be the leftmost picture in a row, then we can use \eqref{azzcommute} to bring the two copies of $a_i$ together, use \eqref{asquaredz} to combine them, and then use \eqref{needle} to show the result is zero. Now we handle the other diagrams in each row by proving that they are obtained from the leftmost by adding a dot on top.

Consider the top row; other rows are entirely analogous. Consider what happens when we place a dot on the leftmost picture in various places. Place a dot on the 3rd to left strand on top
(the purple strand) and resolve using relation \eqref{dot6}. You get two diagrams, one of which is the second diagram on the first row, the other of which factors through a map which
vanishes when hit with $Z_K$, thanks to Proposition \ref{whatkillsgeneral}. If you place a dot instead on the 5th strand (green) and resolve, and ignore diagrams which vanish due to
Proposition \ref{whatkillsgeneral}, then one is left with the third diagram on the first row. Placing a dot on the 8th strand (red) will give the final diagram on that row. \end{proof}

\begin{example} Placing a dot on the 5th strand and resolving: \igc{1}{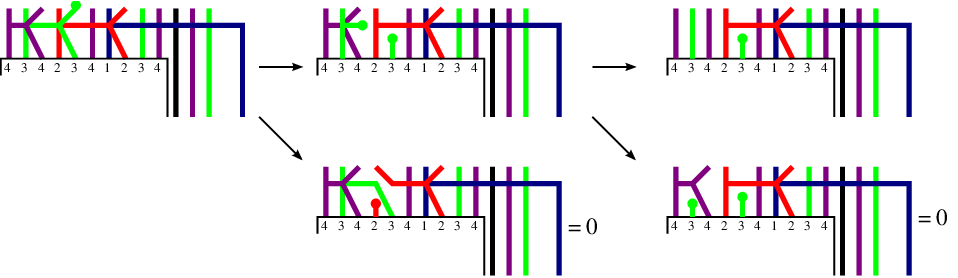} \end{example}

\begin{lemma} (Example: $n=5$) Suppose that we place a dot below $FR_1$ on one of the last strands $54321$, but not on the final strand $1$. If we place a dot on strand $2$, then the
resulting morphism can be rewritten as a linear combination of morphisms which factor through the morphisms on the first row of Figure \ref{killerwrapgeneralfig}. If we place a dot on
strand $3$, then it can be rewritten to factor through the second row in that figure. Etcetera. \label{FRidot} \end{lemma}

\begin{proof} We shall prove the statement placing a dot on strand $2$, and let the reader prove the remainder as an exercise (it follows by the same arguments). The calculation is
straightforward but annoying. We are trying to understand the morphism \igc{1}{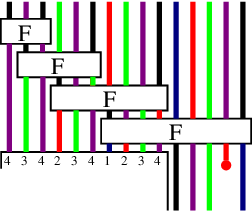} The first step is to resolve the rightmost flip as below, using \eqref{dot6} twice.
\igc{1}{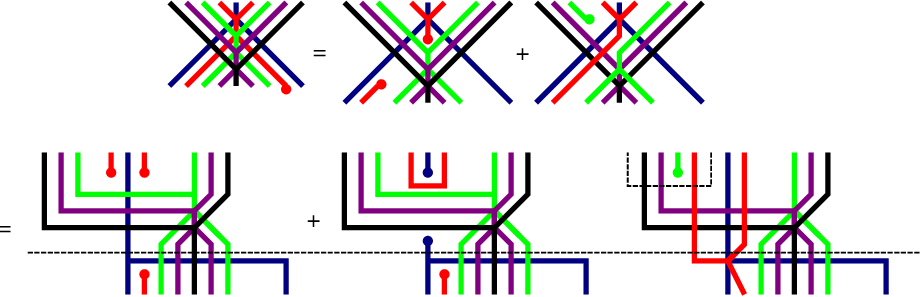} Having written the final three terms suggestively, we see that the regions below the dotted line are familiar. Note that these diagrams have bottom $1234\
543\hat{2}1$ with a conspicuously absent $2$. Let us compare the part below the dotted line with the part of each diagram in the first row of Figure \ref{killerwrapgeneralfig} which
only involves $1234\ 543\hat{2}1$. In the first two pictures, the region below the dotted line is precisely the rightmost entry in the first row. Thus we need only deal with the third
picture, where the region below the dotted line agrees with the other entries in the first row.

In the third picture, the region enclosed by the dotted box will be passed as input to the next flip map in the sequence. The dot coming into the next flip is green, not red, but it is
still the 2nd to right input (which will remain true as we iterate this procedure). Therefore we may resolve the next flip map in the same way, to yield two diagrams which factor
through the 2nd to right entry of the first row of Figure \ref{killerwrapgeneralfig}, and a final diagram which has a dot entering the next flip in the 2nd to right input. The third
diagram of this iteration is \igc{1}{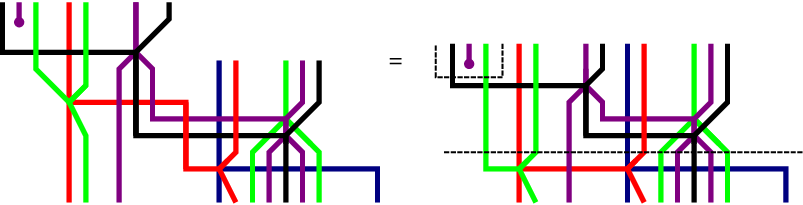} Once again, the region below the dotted line agrees with the $234\ 1234\ 543\hat{2}1$ part of the remaining diagrams on the first row
of Figure \ref{killerwrapgeneralfig}. We may then repeat the argument.

Where the argument stops is when the 2nd to right input of a flip is also the central input of a flip. Equivalently, this is when the dot coming into the next flip is black. (If our
original dot were on strand $3$, we would stop when the 3rd to right input was also the central input of a flip, which is the penultimate flip map.) This black dot enters a black-purple
6-valent vertex. Resolving this using \eqref{dot6}, we obtain two diagrams, both of which factor through a purple trivalent vertex. This, finally, is the purple trivalent vertex appearing in the first entry of the relevant row of Figure \ref{killerwrapgeneralfig}. Both diagrams factor through this first entry. \end{proof}

\begin{lemma} Suppose that we ``abort" $FR_1$. That is, we follow the path $FR_1$ from $\tb$ up through $\Gamma$, but at some vertex $\xb$ along the way, instead of following the next
adjacent edge $\xb \uponeto \yb$, we do an aborted 6-valent vertex instead. The resulting morphism then factors through the morphisms in Figure \ref{killerwrapgeneralfig}, except
instead of having a conspicuously absent index on the right, we have a boundary dot instead. \label{FRiabort} \end{lemma}

\begin{example} Let $n=5$ and let us abort $FR_1$ at the top of the bottommost flip $F_{1,5}$. We ignore the sequence of lines $434234$ appearing on the left of the morphism. Then using
\eqref{dot6} repeatedly we resolve as follows. \begin{equation} \label{hardproofgah} \ig{1}{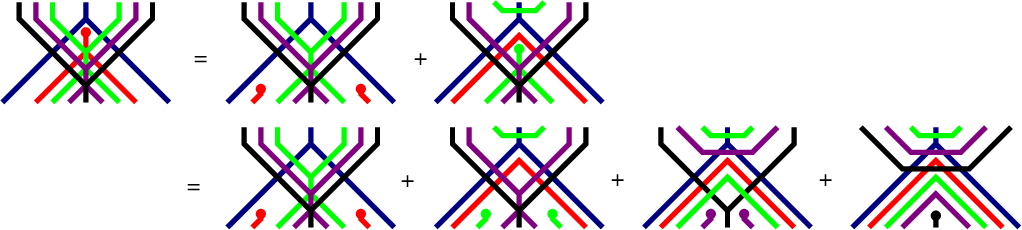} \end{equation} There are four terms in the end, corresponding to the five
indices in $54321$ but not counting the final $1$. Each term has a dot somewhere in the final string $54321$; remove that dot to leave a conspicuously absent strand instead. Each term
also has a trivalent vertex or cap (the blue trivalent in the first term, the red cap in the second, the green cap in the third, the purple cap in the fourth) which surrounds the
remaining dot, and surrounding a bunch of junk that it slides through. Note that a cap factors through a trivalent vertex, via \eqref{unit}. Thus the result factors through the
rightmost entry of the appropriate row of Figure \ref{killerwrapgeneralfig}.

Suppose we had aborted $FR_1$ inside $F_{1,5}$, but the red-green 6-valent instead of the blue-red 6-valent. This yields a subdiagram of one of the intermediate terms pictured in
\eqref{hardproofgah} (the rightmost diagram on the top row, replacing the red cap with a red trivalent vertex via \eqref{unit}). In general, aborting $F_{1,5}$ on top, at the blue-red
6-valent vertex, is the hardest; aborting $F_{1,5}$ at any of the other 6-valent vertices and resolving will give subdiagrams of some of the four terms obtained in \eqref{hardproofgah}.
These subdiagrams all factor through the rightmost entry of the appropriate row of Figure \ref{killerwrapgeneralfig}. \end{example}

\begin{proof} When we abort $FR_1$ at the very first edge (the bottommost 6-valent vertex in $F_{1,5}$) we get the bottom entry of Figure \ref{killerwrapgeneralfig}. Now we use
induction on which edge we abort at. As can be seen in the example above, when one aborts and then resolves the dot using \eqref{dot6}, there are two terms, one of which factors through
a previous abortion, and the other of which will provide an interesting new term. This is true also if we abort the first vertex in a new flip, for we shall have \begin{equation}
\ig{1}{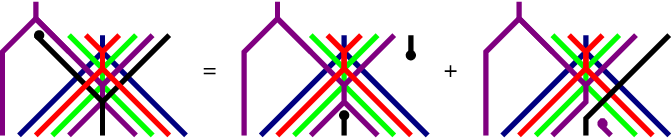} \label{blargh} \end{equation} where the first term is a previous abortion. At each stage, therefore, we will ignore previous abortions by induction. In doing so,
we only look at one term in each resolution of \eqref{dot6}. So we can deterministically replace our aborted copy of $FR_1$ with another diagram where no dot touches a 6-valent vertex.
There are only two operations one needs to perform: replacing the LHS of \eqref{hardproofgah} with only the first term on the RHS, and replacing the LHS of \eqref{blargh} with only the
second term on the RHS.

It is a simple exercise to confirm that this yields precisely the diagrams in Figure \ref{killerwrapgeneralfig}, with a boundary dot instead of a conspicuously absent strand.

For the example $n=5$, the path $FR_1$ has length $10$, so there are $10$ possible abortions, which we order by the edge aborted. Resolving these modulo previous abortions, we have 10
different diagrams. These factor precisely through the 10 diagrams in Figure \ref{killerwrapgeneralfig}, in the following order: the rightmost column from bottom to top, then the next
rightmost column from bottom to top, and so on. \end{proof}

This concludes the proof of Proposition \ref{abortedV}. Using the arguments of \S\ref{subsec-proofoutline}, we now know that $\phi_J$ is a consistent family of projectors.

%
\subsection{Analyzing the image of the projectors}
\label{subsec-projectors}
%
%
%

The family $\phi_J$ picks out some mutual summand of every reduced expression for $w_J$, living in the Karoubi envelope of $\DC$, which we temporarily call $X$. We wish to show that
this summand is precisely $B_J$ (we have not even shown yet that $X$ is nonzero).

\begin{prop} The $R$-bimodule $\HOM_{\DC}(X,\emptyset)$ is a free left (or right) $R$-module of rank $1$, generated by the map $\xi_J$ of degree $d_J$ which consists of including
$X$ into $B_\xb$ for some reduced expression $\xb$ of $w_J$, and then placing dots on each strand of $B_\xb$. This map is independent of the choice of $\xb$. \label{zHomtoR} \end{prop}

\begin{proof} Let $\xi_{\tb}$ denote the map $B_{\tb} \to \emptyset$ which puts a dot on every strand. This will descend to a non-zero map from $B_J$ if and only if $\xi_{\tb} Z$ is nonzero, since $Z$ is a member of our family $\phi$.
	
\begin{example} (Example: $n=4$) The map $\xi$: \igc{1}{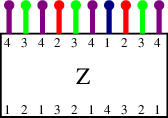} \end{example}

Consider the functor from diagrammatics to $R$-bimodules defined in \cite{EKho}. It is easy to see that $Z$ preserves 1-tensors, since the same is true of the 6-valent and 4-valent
vertices (we recalled this in \S\ref{subsec-soergeldiagrammatics}). The collection of dots sends $1 \ot 1 \ot \ldots \ot 1$ to $1$. Therefore the composition $\xi_{\tb} Z$ is nonzero,
sending the 1-tensor to $1$.

We know that all Hom spaces in $\SBim$ are free as left (or right) $R$-modules, so it is enough to show that every morphism from $B_{\tb}$ to $\emptyset$, when precomposed with $Z$,
reduces to $\xi_{\tb} Z$ with some polynomial. If we do this, we will show that the space of morphisms is rank 1, and therefore is determined by the image of the 1-tensor from
$B_{\sb}$. For another $\xb \in \Gamma$, letting $\xi_{\xb}$ be the collection of dots, the corresponding morphism from $B_J$ is $\xi_\xb \phi_{\sb,\xb}$. But $\phi_{s_J,x}$ also
preserves 1-tensors, so the overall map sends the 1-tensor to $1$. Thus, the map is independent of the choice of $\xb$, pending the proof that morphisms are rank 1.

(Example: $n=4$) We now use the one-color reduction results of \cite{EKho} (and remind the reader to reread this part of \cite{EKho} for some terminology). Consider an arbitrary
morphism from $B_{\tb}$ to $\emptyset$. First simplify the color $1$, which is an extremal color (in the Dynkin diagram). Since it appears only once in $\tb$, we can reduce the
$1$-colored subgraph to a boundary dot. This boundary dot will not interfere with further simplifications of the diagram. Now consider the color $2$, which is an extremal color in the
remainder of the morphism. There are exactly two instances of $2$ on the boundary, so either they are connected by a line, or they both end in boundary dots. However, if they are
connected in a line then the morphism must factor through one of the pictures which vanishes due to Proposition \ref{whatkillsgeneral} (only the last column is needed). So both
instances of $2$ end in boundary dots, and they will not interfere with further simplification. Now consider the color $3$, which is an extremal color in the remainder of the morphism,
so it must form a disjoint union of simple trees. There are three instances of $3$ on the boundary; call them instance $1,2,3$. If instance $m$ is connected to instance $m+1$ in the
tree, then the morphism factors through a vanishing picture in Proposition \ref{whatkillsgeneral}. If this is not the case, but instance $m$ is connected to instance $m+2$ (in this
case, the only possibility is instance $1$ connected to instance $3$) then instance $2$ ends in a dot. However, we know that \igc{1}{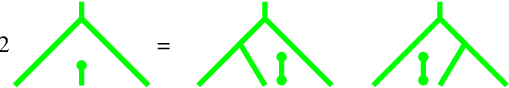} thanks to \eqref{dotslidesame}.
Therefore, such a morphism factors through one where $m$ and $m+1$ are connected, for some $m$, and vanishes by the above arguments. Thus none of the instances are connected, and they
must all end in boundary dots. (If one does not believe that morphisms must be free $\Z$-modules and is worried about $2$-torsion, one can calculate that this morphism is zero more
directly, without using the \eqref{dotslidesame} trick.) Similarly, considering the color $4$, no two instances of $4$ on the boundary may be connected, since such a morphism will
factor through a morphism where  instances $m$ and $m+1$ are connected, and this morphism will vanish thanks to Proposition \ref{whatkillsgeneral}. Therefore, any morphism which
pairs against $Z$ to be nonzero can be assumed to be $\xi_{\tb}$. \end{proof}

\begin{prop} For any $i \in J$, $X \ot i \cong X(1) \oplus X(-1)$. \label{BJBidecomp} \end{prop}

\begin{proof} Note that \igc{1}{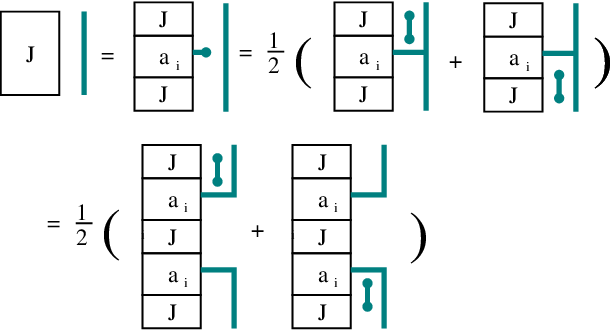} In this calculation, the box labeled $J$ just represents some appropriate projector for $X$. This calculation, proven using Proposition
\ref{aproperties}, decomposes the identity of $X \ot i$ into two orthogonal idempotents, in a similar fashion to \eqref{iidecomp}. Checking that they are orthogonal idempotents is also
easy with Proposition \ref{aproperties}. \end{proof}

Putting these propositions together, we have now proven Theorem \ref{MainThm}, as outlined in \S\ref{subsec-proofoutline}.

\section{Thick calculus}
\label{sec-augmented}
%
%
%

Recall some notation: when $J$ is a connected parabolic subset, $W_J$ is its parabolic subgroup, $w_J$ is the longest element of $W_J$, and $\tilde{\Gamma}_{w_J}$ is the set of reduced
expressions for $w_J$. The length of $w_J$ is $d_J$.

In the previous chapter we constructed a family of maps $\phi_J = \{ \phi_{\xb,\yb} \}$ in $\DC$ associated to pairs $\xb,\yb \in \tilde{\Gamma}_{w_J}$, when $J$ is connected. We proved
in Theorem \ref{MainThm} that these maps formed a compatible system of projectors, in the sense of Definition \ref{defn:consistentfamily}. Thus they pick out a single summand in the
Karoubi envelope of $\DC$, a summand we abusively call $J$. We may now use Claim \ref{partialidempotentcompletion} to extend the graphical calculus for $\DC$ into a graphical calculus
for $\fooDC$, the partial idempotent completion which adds the new object $J$ for each connected parabolic subset. We have also shown in Theorem \ref{MainThm} that the
functor from $\DC$ to $\BSBim$, extended to the Karoubi envelopes, sends $J$ to $B_J$. Thus $\fooDC$ and $\fooBim$ are equivalent.

\begin{defn} Let $\fooDC$ be the graded monoidal category presented diagrammatically as follows. Objects are sequences $\JJ = J_1 J_2 \ldots J_d$ of connected subsets of
$I=\{1,\ldots,n\}$. When $J=\{j\}$ is a single element, we write the element $j$ instead of writing $J$, and identify it with an object in $\DC$. We draw the identity of $J$ as follows: \igc{.6}{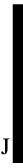}

The generating morphisms are the usual generators of $\DC$, in addition to \emph{$J$-inclusions} and \emph{$J$-projections}. The $J$-inclusion is a map from $J$ to $\xb$ where $\xb$
is any reduced expression for $w_J$. The $J$-projection is a map in the other direction. Both have degree $0$. \igc{1}{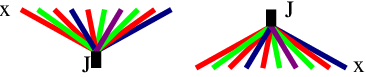}

The relations are those relations of $\DC$ as well as

\psfrag{psphiyx}{$\displaystyle \phi_{\yb,\xb}$}
\begin{equation} \ig{1}{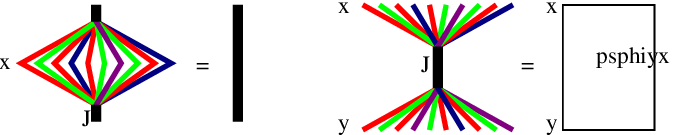} \label{projinc}. \end{equation}
\end{defn}

\begin{thm} This diagrammatic category is equivalent to the partial idempotent completion of $\DC$ by the images of $\phi_J$. The functor $\FC$ from $\DC$ to $\BSBim$ extends to a functor from $\fooDC$ to $\fooBim$, which is an equivalence when $\FC$ is an equivalence. \label{defnfooworks} \end{thm}

\begin{proof} As discussed above, this follows from Claim \ref{partialidempotentcompletion} and Theorem \ref{MainThm}. \end{proof}

Let us mention where the functor from diagrams to bimodules sends the $J$-inclusions and $J$-projections. These are degree zero maps, determined uniquely up to an invertible scalar. We
choose the scalars so that the $J$-projection map sends $1 \ot 1 \ot \ldots \ot 1 \in B_{\xb}$ to $1 \ot 1 \in B_J$ for any $\xb$, and so that the $J$-inclusion map sends $1 \ot 1 \in
B_J$ to $1 \ot 1 \ot \ldots \ot 1 \in B_{\xb}$. Since the transition maps $\phi_{\xb,\yb}$ send 1-tensors to 1-tensors, this system of scalars is consistent and the functor is
well-defined.

Note that the image of $B_J$ in $B_\xb$ is spanned by tensors of the form $f \ot 1 \ot 1 \ot \ldots \ot 1 \ot g$, which we call \emph{extremal tensors}. In other words, the image is
generated as a bimodule by the 1-tensor. Since any transition map $\phi_{\xb,\yb}$ factors through $B_J$, it must send arbitrary tensors to extremal tensors. To determine what the
$J$-projection map will do to an arbitrary tensor $f_1 \ot f_2 \ot \ldots \ot f_d$, we may apply any transition map $\phi_{\xb,\yb}$ (which is explicit, albeit annoying to
compute) to obtain an extremal tensor, and then map the extremal tensor to $B_J$ by the rule $f \ot 1 \ot \ldots \ot 1 \ot g \mapsto f \ot g$. However, we have not produced anything
resembling a closed formula for the $J$-projection.

We have proven in Proposition \ref{phiunchanged} that, composing $\phi_{\xb, \yb}$ with a 6-valent or 4-valent vertex $\yb \to \zb$ produces $\phi_{\xb,\zb}$. The implications in $\fooDC$ are the following two relations.
\begin{equation} \ig{1}{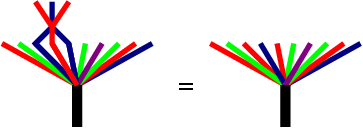} \label{suckin6} \end{equation}
\begin{equation} \ig{1}{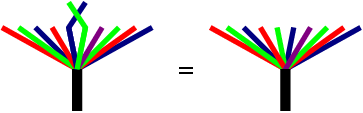} \label{suckin4} \end{equation}
Recall that $\phi_{\xb,\yb}$ is constructed only out of 4-valent and 6-valent vertices, so that it also ``sucks in" to the $J$-inclusion, a fact which we also know from \eqref{projinc}. Another implication is that a $J$-inclusion is orthogonal to an aborted 6-valent vertex, or any of the diagrams in Proposition \ref{whatkillsgeneral} which kill the morphism $Z$.

As discussed in \S\ref{subsec-thickening}, the diagrammatic calculus for $\fooDC$ can be improved by introducing new pictures which represent maps that can already be obtained from the
generators above. This ``augmentation'' of the diagrammatics makes many statements more intuitive. We now introduce new pictures one at a time with their defining relation, and discuss
their properties.

Any relations between these new pictures must be checked in $\DC$ using the relations already mentioned. This can be done explicitly, using the calculations of \S\ref{sec-calcs} or
analogous computations. Some of these checks are huge pains in the neck though. Thankfully, Theorem \ref{defnfooworks} implies that we already know the graded dimension of Hom spaces
between objects in $\fooBim$, and we have a functor to $R$-bimodules. Thus, if we have two different pictures which represent a map of a certain degree, and the dimension of the Hom
space in that degree is 1, then they are equal up to a scalar, and the scalar may often be quickly checked by looking at what the map does to a 1-tensor. This saves a great deal of work.

When calculating these graded dimensions, just remember that $\epsilon(b_J)=v^{d_J}$, $b_J^2 = \qJ b_J$, and $b_J b_i = \qtwo b_J$ when $i \in J$. For instance, $\END(B_J)$ has graded
dimension $v^{d_J} \qJ$, so it has a 1-dimensional space of degree $0$ morphisms, generated by the identity map.

\begin{defn} The \emph{thick cup} and \emph{thick cap} express the biadjointness of $J$ with itself. \begin{equation} \ig{.8}{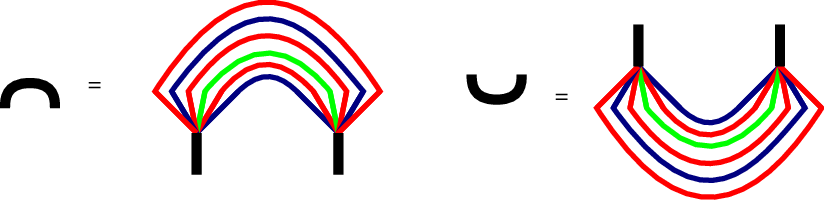} \label{defthickcupcap}
\end{equation} \end{defn}

Thanks to the self-biadjointness of $B_i$, we know that $B_{\ii}$ is biadjoint to $B_{\omega(\ii)}$ (recall that $\omega$ reverses the order of a sequence). Consequently, $B_v$ is
biadjoint to $B_{v^{-1}}$ for an arbitrary element $v \in W$. The longest element $w_J$ of a parabolic subgroup is an involution, so $B_J$ should be self-biadjoint. The thick cups and
caps realize this biadjunction. Note that the thick cups and caps are well-defined because for any reduced expression $\xb$ of $w_J$, $\omega(\xb)$ is also a reduced expression for
$w_J$. It is straightforward to check the biadjointness relation directly.

\begin{equation} \ig{.8}{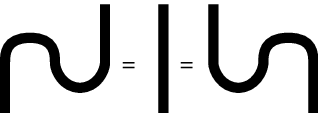} \label{thickbiadjoint} \end{equation}

\begin{defn} The \emph{thick dot} is obtained by choosing a reduced expression $\xb$ for $w_J$ and composing $J \to \xb \to \emptyset$, where the latter map consists of a dot on
every strand. \begin{equation} \ig{1}{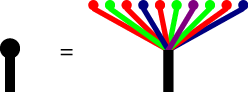} \label{defthickdot} \end{equation} \end{defn}

This was the map called $\xi_J$ in Proposition \ref{zHomtoR}. From that proposition we have:

\begin{claim} The map above is nonzero, and independent of the choice of $x$, so it is well defined. It is the generator of $\Hom_{\fooDC}(J,\emptyset)$ as an $R$-bimodule. \end{claim}

The definition of an upside-down thick dot is the same, only flipped upside-down. The cyclicity of thick dots is quite clear.

\begin{equation} \ig{.8}{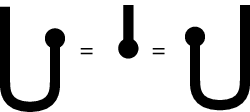} \label{thickdotcyclic} \end{equation}

\begin{defn} The \emph{thick crossing} exists only when $J$ and $K$ are \emph{distant}, that is, every index inside them is distant. It agrees with the usual 4-valent crossing when
$J=\{j\}$ and $K=\{k\}$. \begin{equation} \ig{1}{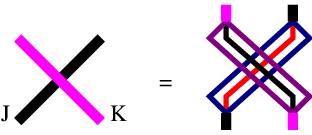} \label{defthickcrossing} \end{equation} \end{defn}

This crossing gives the isomorphism $J \ot K \cong K \ot J$ for $J,K$ distant. There will be distant sliding rules, which we will present in full after all the new pictures have
been presented. Once again, it is obvious that the thick crossing is cyclic.

\begin{equation} \ig{.8}{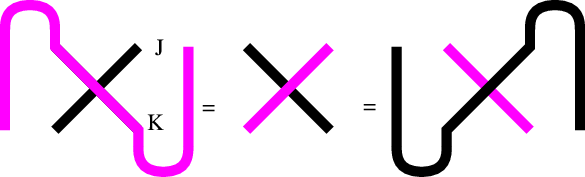} \label{thickcrossingcyclic} \end{equation}

\begin{defn} The \emph{thick trivalent vertex} exists only when $i \in J$. It agrees with the usual trivalent vertex when $J$ is a singleton. Thick trivalent vertices may be
\emph{right-facing} as in the picture below, or may be \emph{left-facing} (sending the extra index $i$ off to the left). For the definition of $a_i$, see \S\ref{subsec-thicktrivalent}. \begin{equation} \ig{1}{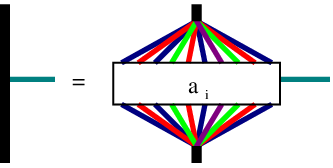} \label{defthicktri}
\end{equation} \end{defn}

The utility of these maps has already been seen in the proofs of the previous section. These relations rephrase Proposition \ref{aproperties}.

\begin{equation} \ig{.8}{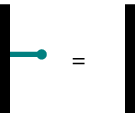} \label{thicktridot} \end{equation}
\begin{equation} \ig{.8}{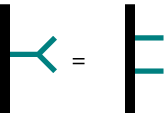} \label{thicktriass} \end{equation}
\begin{equation} \ig{.8}{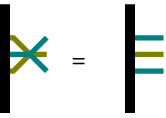} \label{thicktri6} \end{equation}
\begin{equation} \ig{.8}{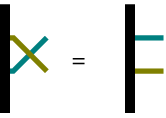} \label{thicktri4} \end{equation}
\begin{equation} \ig{.8}{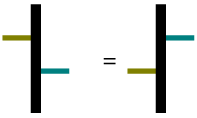} \label{thicktriopp} \end{equation}
Note that \eqref{thicktri6} only makes sense when the teal and brown indices are adjacent, and \eqref{thicktri4} only when they are distant.

In addition, a quick comparison of $a_i$ when expressed on the right for $s^R_J$ and on the left for $t^L_J$ will convince the reader of the cyclicity of the thick trivalent vertex.

\begin{equation} \ig{.8}{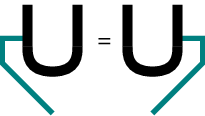} \label{thicktricyclic} \end{equation}

Rephrasing the isomorphism $B_J \ot B_i \cong \qtwo B_J$ in terms of thick trivalent vertices is easy given from Proposition \ref{BJBidecomp}. We leave it to the reader to prove that
the thick trivalent vertex, viewed as a morphism $J \ot i \to J$, is the unique morphism of degree $-1$ up to scalar, and after applying the functor to bimodules, it sends $f \ot g \ot
h \mapsto f \ot \partial_i(g)h$. (Hint: use a reduced expression for $w_J$ ending in $i$.)

Now, we wish to give a diagrammatic relation corresponding to the isomorphism $B_J \ot B_J \cong \qJ B_J$, and so we will need a number of projection and inclusion maps between $B_J \ot
B_J$ and $B_J$ of various different degrees. We construct these in a fashion entirely akin to the isomorphism $B_i \ot B_i \cong \qtwo B_i$. First, we find the projection/inclusion
of minimal degree, and draw it as a trivalent vertex.

\begin{defn} The \emph{very thick trivalent vertex} is contructed as follows.  Rotate the $J$-inclusion for $\xb$ by 90 degrees, and then connect the output sequence $\xb$ to another $J$-colored strand by a sequence of thick trivalent vertices. There are $d_J$ thick trivalent vertices, so this morphism has degree $-d_J$.
\begin{equation} \ig{1}{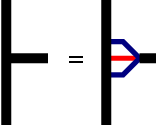} \label{defverythicktri} \end{equation} \end{defn}

This map does not depend on the choice of $\xb$. If we combine \eqref{thicktri6} with \eqref{suckin6}, or \eqref{thicktri4} with \eqref{suckin4}, we can alter $\xb$ by applying any
braid relation.

Moreover, it is true, but not immediately obvious, that the morphism is cyclic. We show this via the functor to bimodules. Recall that there is a unique morphism of degree $-d_J$ from
$B_J \to B_J \ot B_J$ up to scalar, and similarly from $B_J \ot B_J \to B_J$. Consider the map $\ig{1}{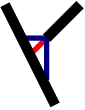}$. It is easy to observe that this map sends the 1-tensor to the
1-tensor, and so does its horizontal reflection, so the map is equal to its horizontal reflection.

\begin{defn} Let $\ii$ be a reduced expression of $w_J$. Then $\partial_{\ii} \define \partial_{i_1} \ldots \partial_{i_{d_J}}$, the composition of Demazure operators, is independent of
the choice of reduced expression. The result is a degree $-2d_J$ map from $R \to R^J$ which is $R^J$-linear. It is denoted $\partial_J$ and is also called a \emph{Demazure operator}.
\end{defn}

Consider $\igv{1}{verythicktritest}$. This map sends $f \ot g \ot h \mapsto f \ot \partial_J(g) h$, and so does its horizontal reflection, so the two maps are equal. These two horizontal reflection equalities are sufficient to prove the cyclicity of the very thick trivalent vertex.

\begin{equation} \ig{1}{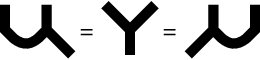} \label{verythicktricyclic} \end{equation}

Before we go any further, we digress to display the \emph{thick distant sliding property}. Everything other than $J$ is assumed to be distant from $J$. The corresponding calculations in
$\DC$ are all trivial, coming from the usual distant sliding property.

\begin{equation} \ig{1}{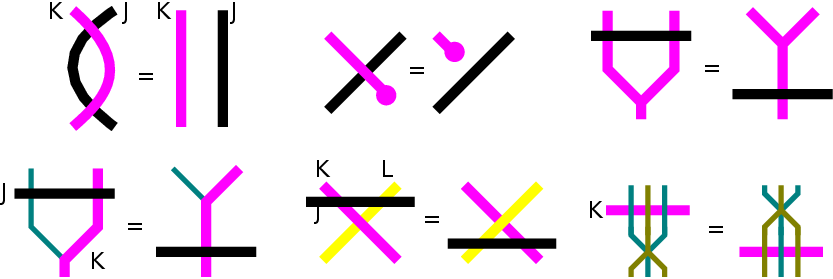} \label{thickdistantslidingproperty} \end{equation}

Now we investigate the very thick trivalent vertex and the thick trivalent dot more thoroughly. The first statement is that they give $J$ the structure of a Frobenius algebra object in
$\fooDC$. In addition to the cyclicity properties already mentioned, this requires only two more relations.

\begin{equation} \ig{.8}{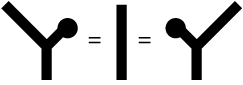} \label{thickunit} \end{equation}
\begin{equation} \ig{.8}{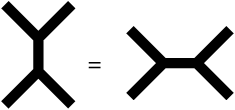} \label{thickassoc} \end{equation}

These are easy to check. For both relations, the equalities must be true up to scalar by a calculation of the graded dimension of Hom spaces, and the scalar is $1$ because both sides
send a 1-tensor to a 1-tensor (this is true of the appropriate twist of (\ref{thickassoc})). Here is another ``unit" relation, checked in a similar way:

\begin{equation} \ig{.8}{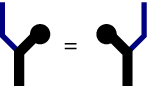} \label{funkyunit}. \end{equation}

Before we go on, it will be useful to make a brief aside about Frobenius algebras and Demazure operators. We can see that $B_J$ is a Frobenius algebra object in $\fooBim$, but this is
because $R$ is a (symmetric) Frobenius algebra over $R^J$. In particular, the Demazure operator $\partial_J$ is the nondegenerate trace, so that there is some set of polynomials
$\{g_r\}$ which is a basis for $R$ as an $R^J$-module (free of graded rank $\qJ$), and some other basis $\{g^*_r\}$ as well, for which $\partial_J(g_r g^*_q) = \delta_{r,q}$. These
polynomials are indexed by $r \in W_J$, and $g_r$ has degree $2l(r)$, while $g^*_r$ has degree $2(d_J - l(r))$. We may as well assume that $g_1 = g^*_{w_J} = 1$.

\begin{claim} The element $\beta=\sum_{W_J} g_r \ot g^*_r \in B_J$ satisfies $f \beta=\beta f$ for any $f \in R$. \end{claim}

This is a standard fact about Frobenius algebras. One obtains the same element $\beta$ regardless of which basis and dual basis we choose. Alternatively, we may define $\beta$ as the
image of $1 \in R$ under the thick dot.

\begin{remark} In type A, there is a nice choice of basis and dual basis known as \emph{Schubert polynomials}. The advantage of Schubert polynomials in type A is that one may choose
other Schubert polynomials to form the dual basis. However, we never make any assumptions about the choice of basis in this paper. \end{remark}

The internet contains numerous easily-obtained resources on Frobenius algebras and Schubert polynomials.

As discussed above, the very thick trivalent vertex, seen as a map from $R \ot_{R^J} R \ot_{R^J} R \{-2d_J\} =B_J \ot B_J \to B_J$, applies the Demazure operator $\partial_J$ to the
middle term, while seen as a map from $B_J \to B_J \ot B_J$, it adds $1$ as the middle term. We now list some relations which can be deduced from the previous discussion without too
much effort, and are obvious generalizations of the one-color relations for $\DC$. Recall that a box with $f \in R$ inside represents multiplication by $f$.

\psfrag{psf}{$f$}
\psfrag{psdJf}{$\partial_J(f)$}
\psfrag{psSigma}{\large{$\sum$}}
\psfrag{psgr}{$g_r$}
\psfrag{psgdualr}{$g^*_r$}
\begin{equation} \ig{1}{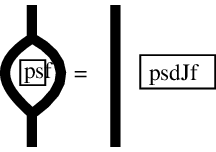} \label{thickcoxeter} \end{equation}
\begin{equation} \ig{1}{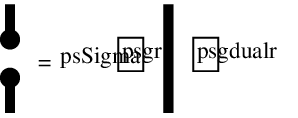} \label{thickbroken} \end{equation}
\begin{equation} \ig{1}{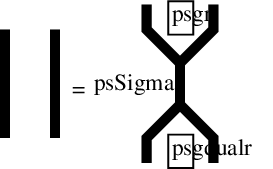} \label{thickthickidemp} \end{equation}	

The last relation expresses the decomposition $B_J \ot B_J \cong \qJ B_J$.

We make one further extension of the calculus.

\begin{defn} The \emph{generalized thick trivalent vertex}, with top and bottom boundary $J$ and right boundary $K$ for $K \subset J$, is constructed in the same fashion as the very
thick trivalent vertex, and generalizes both the very thick trivalent vertex and the thick trivalent vertex. \begin{equation} \ig{1}{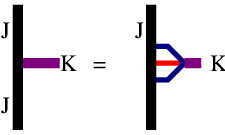} \label{defgenthicktri}
\end{equation} \end{defn}

The reader should easily be able to generalize the previous discussion to this new picture. In particular, it is cyclic, satisfies various associativity properties akin to
(\ref{thicktriass}) and (\ref{thicktriopp}), a unit axiom akin to (\ref{thickunit}) or (\ref{thicktridot}), and another one akin to (\ref{funkyunit}). It has distant sliding rules,
applies the Demazure operator $\partial_K$ to the middle term in $B_J \ot B_K$, and can be used to express the isomorphism $J \ot K \cong \qK J$.

\begin{remark} In fact, the category $\fooDC$ is generated entirely by generalized thick trivalent vertices, thick dots, and thick crossings. It is not too hard to see that the other
maps in $\fooDC$, namely the 6-valent vertex and the $J$-projections and $J$-inclusions, can actually be constructed out of these. Below, the thick line represents $J=\{1,2\}$.
\igc{1}{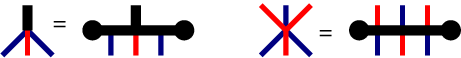} \end{remark}

There are presumably many more interesting relations to be found, but these shall be good enough for now. Finding the relations which intertwine thick lines in interesting ways is a more
difficult problem.

\begin{remark} \label{disconnectedremark} We have chosen to define thick strands only for connected parabolic subsets $J$. Suppose $J$ is a \emph{disconnected} parabolic subset, so $J =
J_1 \coprod \cdots \coprod J_k$ for connected, mutually distant parabolic subsets $J_i$. Thus $W_J = W_{J_1} \times \cdots \times W_{J_k}$, $w_J$ is the product of the various
$w_{J_i}$, and $B_J$ is the tensor product of the $B_{J_i}$ in $\SBim$. So the object $B_J$ is already isomorphic to an object in $\fooDC$.

However, if we wished, we could extend the diagrammatic calculus to include thick lines labeled by disconnected parabolic subsets $J$. In fact, all the diagrammatic conventions and
relations above work verbatim in this context, once one defines the morphism $a_i$ correctly. If we let $\sb_J$ (the source in the graph $\Gamma_{w_J}$) be the concatenation of
$\sb_{J_i}$ (and the same for $\tb$), then we may view the path morphism from $\sb_J \to \tb_J$ diagrammatically as the horizontal concatenations of the path morphism for each
component. To define $a_i$ for $J$, define it for the connected subgraph of which $i$ is a part, and then extend it via the identity map to the remainder of $J$, possibly crossing the
sideways $i$-strand across various distant identity maps. Thanks to distant sliding rules, this $a_i$ has all the desired properties. Theorem \ref{MainThm} is easily adapted to
disconnected $J$ as well. \end{remark}

\section{Induced Trivial Representations}
\label{sec-induced}
%
%
%

Recall from Sections \ref{subsec-hecke} and \ref{subsec-soergel} that the induced (left) representation $T_J$ is the left ideal generated by $b_J$ in $\HB$. To categorify this, we wish
to use the ``left ideal" generated by $B_J$ in $\fooBim$, or in the smaller partial idempotent completion $\BSBim(B_J)$. The diagrammatic version of this ideal in $\fooDC$ would be all
pictures which have a thick line labelled $J$ appearing on the right. However, because we will eventually take idempotent completions, the source and target of our diagram may be assumed to be an object in $\DC$ tensored with $J$. We may as well let the remainder of the diagram be a picture in $\DC$, and view the thick line on the right as some
sort of ``membrane" which interacts in a specific way with the morphisms in $\DC$. This is akin to the categorification of $T_J$ given by the category $\JBim$, which takes a Bott-Samelson bimodule and restricts from $R$ to $R^J$ on the right.

\begin{defn} Let $\TC_J$ be the category defined as follows. Objects are sequences $\ii$ of indices in $I$, just as for $\DC$. Morphisms between $\ii$ and $\jj$ are given by (linear
combinations of) pictures on the plane, with appropriate top and bottom boundary, which include a membrane on the right labelled $J$. These pictures are constructed out of the
generators of $\DC$ and the \emph{left-facing} thick trivalent vertex, which is the only interaction with the membrane (see the picture below). Strands involved in a trick trivalent
vertex must be labeled by $i \in J$. The relations between these morphisms are given by those relations of $\DC$ as well as the left versions of \eqref{thicktridot},
\eqref{thicktriass}, \eqref{thicktri6}, and \eqref{thicktri4}. This category is equipped with an obvious monoidal action of $\DC$ on the left. We view morphisms as being equipped with
the structure of a left $R$-module, by placing double dots on the left of the diagram, and a right $R^J$-module, by placing symmetric polynomials immediately to the left of the
membrane. \end{defn}

Here is an example morphism, for some $J$ containing $\{1,2,4\}$. Note that there is no assumption that $J$ be connected. \igc{1}{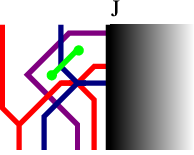}

\begin{remark} The right action of $R^J$ on morphisms is well-defined. That is, it does not matter in which region to the left of the membrane one places the polynomial, since the
polynomial is symmetric in $W_J$ and therefore slides freely across any line labelled $i \in J$. Only these lines may be involved in thick trivalent vertices with the membrane, so only
these lines separate the regions in question. \end{remark}

\begin{defn} There is a functor $\FC_J$ from $\TC_J$ to $\JBim$ defined as follows. The object $\ii$ is sent to $B_{\ii}$, restricted so that it is a right $R^J$-module
instead of a right $R$-module. Morphisms using usual Soergel diagrammatics are sent to their usual counterparts in $\BSBim$, which are obviously also right $R^J$-module maps. The image
of the thick trivalent vertex is:

\begin{eqnarray}
\ig{.6}{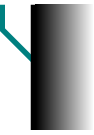} & \auptob{R}{R \ot_{R^i} R}  & \amapsuptob{f}{f \ot 1}\\
\igv{.6}{thicktrimembrane} & \auptob{R \ot_{R^i} R}{R}  & \amapsuptob{f \ot g}{f \partial_i(g)}
\end{eqnarray}
These are clearly maps of left $R$-modules and right $R^J$-modules. \end{defn}

\begin{defn} There is a functor from $\TC_J$ to the diagrammatic version of $\fooDC$ defined as follows. The object $\ii$ in $\TC_J$ is sent to $\ii \ot J$ in $\fooDC$. The map is
given on morphisms by interpreting the membrane as a thick line labelled $J$, with nothing to the right of it. When $J$ is disconnected, the object $J \in \fooDC$ is understood via Remark \ref{disconnectedremark}. \end{defn}

\begin{claim} These functors are well defined, and preserve the $(R,R^J)$-bimodule structure of Hom spaces. The composition of functors from $\TC_J$ to $\fooDC$ and then to $\fooBim$ is
equal to the composition of functors from $\TC_J$ to $\JBim$ and then to $\fooBim$ by inducing from $R^J$ to $R$. \label{functorscompose} \end{claim}

\begin{proof} Given the calculations and remarks of the previous chapter, it is entirely straightforward to check this on objects and on generating morphisms. \end{proof}

Consider the map $\Hom_{\TC_J}(\ii,\jj) \to \Hom_{\fooDC}(\ii J,\jj J)$ given by this functor. We will show that this map is an injection. It is not a surjection,
essentially because it misses polynomials to the right of the thick line. For example, consider the left side of \eqref{thickbroken}. This does not correspond to a map in $\TC_J$ from
$\emptyset \to \emptyset$ because it factors in $\fooDC$ through an object $\emptyset$ without $J$ on the right, so can not be described with a membrane. Thanks to \eqref{thickbroken},
however, we may express it as a linear combination of morphisms in the image of $\TC_J$ but with added polynomials on the right. This is the only obstruction to being a surjection, that
is, there will be an isomorphism \begin{equation} \label{basechangeTJ} \Hom_{\TC_J}(\ii, \jj) \ot_{R^J} R \to \Hom_{\fooDC}(\ii J, \jj J).\end{equation} Because any polynomial
in $R^J$ will slide across a thick $J$-colored strand, placing this polynomial to the left of the membrane and applying the functor agrees with placing it on the far right of the
diagram after applying the functor. Thus this map is well-defined.

Recall from Section \ref{subsec-soergel} that for two $R-R^J$-bimodules $X,Y$ in $\JBim$ we have an $R$-bimodule isomorphism $\Hom_{\SBim}(X \ot_{R^J} R,Y \ot_{R^J} R) \cong
\Hom_{\JBim}(X,Y) \ot_{R^J} R$.

\begin{thm} The functor from $\TC_J$ to $\JBim$ is an equivalence of categories. \label{mainthminduced} \end{thm}

\begin{prop} The map in \eqref{basechangeTJ} is an isomorphism. \label{homsinTJ} \end{prop}

\begin{cor} The space $\Hom_{\TC_J}(\ii,\jj)$ is a free left $R$-module of graded rank $v^{-d_J}\epsilon(b_Jb_{\ii}b_{\omega(\jj)})$. \end{cor}

\begin{proof} First we show that the map in \eqref{basechangeTJ} is surjective. Consider an arbitrary morphism in $\fooDC$ from $\ii J$ to $\jj J$. We can merge the two $J$ inputs (top right, bottom
right) using the relation \eqref{thickthickidemp}.

\psfrag{psSigma}{\large{$\sum$}}
\psfrag{psgr}{$g_r$}
\psfrag{psgdualr}{$g^*_r$}
\igc{1}{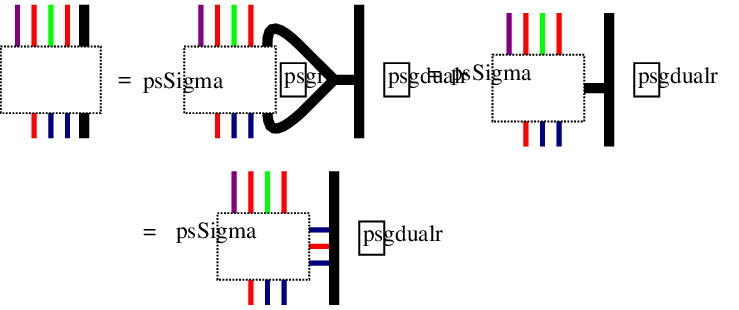}

Then, using the definition in $\fooDC$ of the very thick trivalent vertex, we may rewrite the equation as above. The lines emanating from the thick line in the final picture form some
reduced expression $\xb$ of $w_J$. The morphism inside the box (by which we always refer to the box in the bottom diagram above) may be different in each term of the sum. Note that the
morphism inside the box is a morphism between objects in $\DC$, and since $\DC$ includes fully faithfully into its partial idempotent completion $\fooDC$, we may assume that the
morphism inside the box only uses diagrams from $\DC$ (i.e. no further thick lines are present). Therefore, every morphism is in the image of \eqref{basechangeTJ}.

In fact, the calculation above implies that $\Hom_{\fooDC}(\ii J,\jj J)$ is precisely $M \ot_{R^J} R$, where $M = \Hom_{\DC}(\ii \xb,\jj) \phi_{\xb,\xb}$ is the space of all
morphisms in $\DC$ which fit inside the box, and which are unchanged when composed with the transition map $\phi_{\xb,\xb}$ along the $\xb$ input. Wh claim that any morphism in $\TC_J$
comes from a diagram in $M$. The point is that, using \eqref{thicktridot}, we can ensure that the sequence of colors which meets the membrane is arbitarily long, and then using \eqref{thicktriass}, \eqref{thicktri6}, and \eqref{thicktri4} we can reduce this sequence to any given expression $\xb$ for $w_J$.

\igc{1}{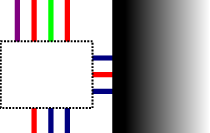}

Therefore we have a surjective map $M \to \Hom_{\TC_J}(\ii,\jj)$, which provides the inverse of the Hom space map above. This proves the proposition.

Because the various functors compose properly as in Claim \ref{functorscompose}, it would be impossible for the above map to be an isomorphism unless the functor from $\TC_J \to
\JBim$ is fully faithful. This proves the theorem. Thus, too, is the Corollary proven, since Hom spaces in $\JBim$ are free as right or left modules, and graded ranks in
$\JBim$ are known already to agree with this formula (see the end of Section \ref{subsec-soergel}). \end{proof}

\newpage

\appendix
\section{Summary of notation}
\label{listofnotation}
%

\begin{itemize}[leftmargin=*]

\item $W$ is the symmetric group $S_{n+1}$. $I = \{1, \ldots, n\}$ is the indexing set of simple reflections. $w_0$ is the longest element.

	\item $J$ and $K$ are \emph{parabolic subsets}, subsets of $I$.

	\item $W_J$ is the parabolic subgroup attached to $J$, with longest element $w_J$. The length of $w_J$ is $d_J$. The Poincar\'e polynomial of $W_J$ is $\qJ$.

	\item $\ii$ is a sequence of indices in $I$. $s_{\ii}$ is the corresponding element in $W$. $\JJ$ is a sequence of parabolic subsets.

	\item $\HB$ is the Hecke algebra of $S_{n+1}$. $\HB_J$ is the subalgebra for $W_J$.

	\item $\{b_w\}$ is the Kazhdan-Lusztig basis of $\HB$. $b_i = b_{s_i}$ and $b_J = b_{w_J}$.

	\item $T_J$ is the left ideal of $b_J$ in $\HB$.
	
	\bigskip

	\item $R$ is the polynomial ring in $n+1$ variables $f_i$, over a base field $\Bbbk$. It has an action of $W$. It is graded, with $f_i$ in degree $2$.

	\item $R^J$ are the invariants of $R$ under $W_J$. $R^i$ are the invariants under $s_i$.

	\item $\HOM$ denotes the graded vector space of morphisms of all degrees in a graded category.

	\item $B_i = R \ot_{R^i} R(1)$. $B_J = R \ot_{R^J} R(d_J)$. The grading shift places the \emph{1-tensor} $1 \ot 1$ in negative degree.

	\item $B_{\ii} = B_{i_1} \ot_R \cdots \ot_R B_{i_d}$. $B_{\JJ} = B_{J_1} \ot_R \cdots \ot_R B_{J_d}$.

	\item $\BSBim$ is the full subcategory of graded $R$-bimodules with objects $B_{\ii}$. $\DC$ is the diagrammatic category which is (usually) equivalent to it.

	\item $\fooBim$ is the full subcategory of graded $R$-bimodules with objects $B_{\JJ}$. $\fooDC$ is its diagrammatic version.

	\item $\JBim$ is the full subcategory of $(R,R^J)$-bimodules whose objects are restrictions of $B_{\ii}$. $\TC_J$ is the diagrammatic version.

	\item $\SBim$ is the Karoubi envelope of $\BSBim$.

	\item $\pa_i$ and $\pa_w$ are Demazure operators. See \cite{EKho}.

\bigskip

	\item $\tilde{\Gamma}_w$ is the (expanded) graph of all reduced expressions of $w$. If $w$ is omitted, it is usually $w_0$.

	\item $\Gamma_w$ is the conflated graph of reduced expressions.

	\item $\xb, \yb, \zb$ are vertices either in $\tilde{\Gamma}_{w}$ or in $\Gamma_w$. In $\tilde{\Gamma}_w$ they are sequences of indices, like $\ii$, but unlike $\ii$ they are always assumed to be reduced expressions of the element $w$ in question. In $\Gamma_w$ they are equivalence classes of reduced expressions.
	
	\item $B_{\xb}$ is the corresponding object in $\BSBim$ or in $\DC$. When $\xb$ is a vertex in $\Gamma_w$, this represents a canonical isomorphism class in $\DC$.

	\item $\xb \downto \yb$ is some oriented path in $\tilde{\Gamma}_w$ or $\Gamma_w$. $\xb \downoneto \yb$ is a single oriented edge in $\Gamma_w$. $\psi_{\xb \downto \yb}$ is the corresponding path morphism.

	\item The \emph{aborted 6-valent vertex} is the map \eqref{whatisaborted}, which has the same source as a 6-valent vertex, and can be placed in its stead in the \emph{aborted} version of a path morphism.

	\item $\overline{D}$ is the path $D$ in reverse, or the diagram $D$ turned upside-down.
	
	\bigskip
	
	\item $\sb$ is the unique source in $\Gamma_{w_0}$. $\tb$ is the unique sink.
	
	\item Anything with subscript $J$ is the corresponding thing for $\Gamma_{w_J}$ for a connected parabolic subset $J$.
	
	\item $Z = \psi_{\sb \downto \tb}$.
	
	\item $\phi_{\xb, \yb}$ corresponds to the path $\xb \downto \tb \upto \sb \downto \yb$.
	
	\item $\chi_{\xb, \yb}$ corresponds to the path $\xb \upto \sb \downto \tb \upto \yb$.
	
	 \bigskip
	
	\item Color conventions: blue red green purple black is $12345$.
	
	\item $n=5$: $\sb^R = 1\ 21\ 321\ 4321\ 54321$.  $\tb^R = 5\ 45\ 345\ 2345\ 12345$.
	
	\item $n=5$: $\sb^R_3 = 12345\ 1234\ 1\ 21\ 321$. It ends in $\sb^R_{\{1,2,3\}}$.  $\tb^R_3 = 54321\ 5432\ 5\ 45\ 345$. It ends in $\tb^R_{\{3,4,5\}}$.
	
	\item $F_{1,5}$ is the \emph{flip}, a path $123454321 \downto 543212345$. \igc{1.2}{flip}
	
	\item $FR_{\tb,i}$ is a particular sequence of flips, which ends in $\tb$, and whose source has $i$ on the right. $FR_{\sb, i}$ starts at $\sb$, and ends with $i$ on the right.
	
	\item $V$ is a particular path from $\sb^R$ to $\tb^R$, built inductively. $FR_{\tb,1}$ is the sequence of flips above the smaller $V$. \igc{1}{Vdefn}
	
	\item Color conventions: teal and brown are arbitrary indices in $J$.
	
	\item $a_i = a^R_{\xb,i}$ is a particular diagram with top and bottom boundary $\xb$ and side boundary $i$ on the right. It is only defined explicitly for $\xb = \sb$ or $\xb = \tb$. Here is $a^R_{\tb^R,i}$. \igc{1}{thicktridefn1}
	
	\item $\xi_J$ is a map from $B_J$ or $B_{\xb}$ to $\emptyset$, which puts a dot on every strand.
	
\end{itemize}
\newpage

%
%

\bibliographystyle{alpha}
\bibliography{mastercopy}{}

\begin{thebibliography}{KLMS12}

\bibitem[BNM06]{BarMor}
Dror Bar-Natan and Scott Morrison.
\newblock The {K}aroubi envelope and {L}ee's degeneration of {K}hovanov
  homology.
\newblock {\em Algebr. Geom. Topol.}, 6:1459--1469, 2006.

\bibitem[CKM14]{CKM}
Sabin Cautis, Joel Kamnitzer, and Scott Morrison.
\newblock Webs and quantum skew {H}owe duality.
\newblock {\em Math. Ann.}, 360(1-2):351--390, 2014.

\bibitem[EK10]{EKho}
Ben Elias and Mikhail Khovanov.
\newblock Diagrammatics for {S}oergel categories.
\newblock {\em Int. J. Math. Math. Sci.}, pages Art. ID 978635, 58, 2010.

\bibitem[Eli10]{ETemperley}
Ben Elias.
\newblock A diagrammatic {T}emperley-{L}ieb categorification.
\newblock {\em Int. J. Math. Math. Sci.}, pages Art. ID 530808, 47, 2010.

\bibitem[EW]{EWGr4sb}
Ben Elias and Geordie Williamson.
\newblock Soergel calculus.
\newblock Preprint.
\newblock arXiv:1309.0865.

\bibitem[EW14]{EWHodge}
Ben Elias and Geordie Williamson.
\newblock The {H}odge theory of {S}oergel bimodules.
\newblock {\em Ann. of Math. (2)}, 180(3):1089--1136, 2014.

\bibitem[Kau87]{Kauf}
Louis~H. Kauffman.
\newblock State models and the {J}ones polynomial.
\newblock {\em Topology}, 26(3):395--407, 1987.

\bibitem[KL09]{KhoLau09}
Mikhail Khovanov and Aaron~D. Lauda.
\newblock A diagrammatic approach to categorification of quantum groups. {I}.
\newblock {\em Represent. Theory}, 13:309--347, 2009.

\bibitem[KL11]{KhoLau11}
Mikhail Khovanov and Aaron~D. Lauda.
\newblock A diagrammatic approach to categorification of quantum groups {II}.
\newblock {\em Trans. Amer. Math. Soc.}, 363(5):2685--2700, 2011.

\bibitem[KLMS12]{KLMS}
Mikhail Khovanov, Aaron~D. Lauda, Marco Mackaay, and Marko Sto{\v{s}}i{\'c}.
\newblock Extended graphical calculus for categorified quantum {${\rm sl}(2)$}.
\newblock {\em Mem. Amer. Math. Soc.}, 219(1029):vi+87, 2012.

\bibitem[Kup96]{Kupe}
Greg Kuperberg.
\newblock Spiders for rank {$2$} {L}ie algebras.
\newblock {\em Comm. Math. Phys.}, 180(1):109--151, 1996.

\bibitem[Lau10]{LauSL2}
Aaron~D. Lauda.
\newblock A categorification of quantum {${\rm sl}(2)$}.
\newblock {\em Adv. Math.}, 225(6):3327--3424, 2010.

\bibitem[Lib08]{LibLL}
Nicolas Libedinsky.
\newblock Sur la cat\'egorie des bimodules de {S}oergel.
\newblock {\em J. Algebra}, 320(7):2675--2694, 2008.

\bibitem[Lib11]{LibNB}
Nicolas Libedinsky.
\newblock New bases of some {H}ecke algebras via {S}oergel bimodules.
\newblock {\em Adv. Math.}, 228(2):1043--1067, 2011.

\bibitem[Lus14]{LuszUnequal14}
George Lusztig.
\newblock Hecke algebras with unequal parameters.
\newblock Revision, 2014.
\newblock arXiv:0208154v2.

\bibitem[MS89]{ManSch}
Yu.~I. Manin and V.~V. Schechtman.
\newblock Arrangements of hyperplanes, higher braid groups and higher {B}ruhat
  orders.
\newblock In {\em Algebraic number theory}, volume~17 of {\em Adv. Stud. Pure
  Math.}, pages 289--308. Academic Press, Boston, MA, 1989.

\bibitem[MS08]{MazStr}
Volodymyr Mazorchuk and Catharina Stroppel.
\newblock Categorification of (induced) cell modules and the rough structure of
  generalised {V}erma modules.
\newblock {\em Adv. Math.}, 219(4):1363--1426, 2008.

\bibitem[Rou08]{Rouq2KM-pp}
Rapha{\"e}l Rouquier.
\newblock 2-{K}ac-{M}oody algebras.
\newblock Preprint, 2008.
\newblock arXiv:0812.5023.

\bibitem[SAV15]{SheSuh}
Seth Shelley-Abrahamson and Suhas Vijaykumar.
\newblock Higher {B}ruhat orders in type {B}.
\newblock Preprint, 2015.
\newblock arXiv 1506.05503.

\bibitem[Soe90]{Soer90}
Wolfgang Soergel.
\newblock Kategorie {$\mathcal{O}$}, perverse {G}arben und {M}oduln \"uber den
  {K}oinvarianten zur {W}eylgruppe.
\newblock {\em J. Amer. Math. Soc.}, 3(2):421--445, 1990.

\bibitem[Soe92]{Soer92}
Wolfgang Soergel.
\newblock The combinatorics of {H}arish-{C}handra bimodules.
\newblock {\em J. Reine Angew. Math.}, 429:49--74, 1992.

\bibitem[Soe07]{Soer07}
Wolfgang Soergel.
\newblock Kazhdan-{L}usztig-{P}olynome und unzerlegbare {B}imoduln \"uber
  {P}olynomringen.
\newblock {\em J. Inst. Math. Jussieu}, 6(3):501--525, 2007.

\bibitem[Sto11]{StosicSL3}
Marko Stosic.
\newblock Indecomposable 1-morphisms of $\dot{U}^+_3$ and the canonical basis
  of ${U}_q^+(sl_3)$.
\newblock Preprint, 2011.
\newblock arXiv:1105.4458.

\bibitem[VV11]{VarVas}
M.~Varagnolo and E.~Vasserot.
\newblock Canonical bases and {KLR}-algebras.
\newblock {\em J. Reine Angew. Math.}, 659:67--100, 2011.

\bibitem[Wil11]{WillSingular}
Geordie Williamson.
\newblock Singular {S}oergel bimodules.
\newblock {\em Int. Math. Res. Not. IMRN}, (20):4555--4632, 2011.

\bibitem[Wil13]{WillCounter}
Geordie Williamson.
\newblock {S}chubert calculus and torsion.
\newblock Preprint, 2013.
\newblock arXiv:1309.5055.

\end{thebibliography}

\vspace{0.1in}
 
\noindent
{\textsl \small Ben Elias, Department of Mathematics, University of Oregon, Eugene, OR 97403}

\noindent 
{\tt \small email: belias@uoregon.edu}

\end{document}